\documentclass[11pt]{amsart}
\newcommand{\preprint}[1]{}

\newcommand{\hide}[1]{}

\usepackage{amssymb}
\usepackage{amsbsy}
\usepackage{amscd}
\usepackage{amsmath}
\usepackage{amsthm}
\usepackage{enumerate}
\input xy
\xyoption{all}

\numberwithin{equation}{section}

\theoremstyle{plain}
\newtheorem{thm}{Theorem}[section]
\newtheorem{prop}[thm]{Proposition}

\newtheorem{conj}[thm]{Conjecture}

\newtheorem{cor}[thm]{Corollary}
\newtheorem{lem}[thm]{Lemma}
\newtheorem{assumption}[thm]{Assumption}

\theoremstyle{definition}
\newtheorem{defi}[thm]{Definition}

\theoremstyle{remark}
\newtheorem{example}[thm]{Example}

\newtheorem{rem}[thm]{Remark}

\newtheorem{convention}[thm]{Convention}

\newcommand{\A}{{\mathcal A}}

\newcommand{\C}{{\mathcal C}}
\newcommand{\CC}{{\mathbb C}}

\newcommand{\E}{{\mathcal E}}
\newcommand{\F}{{\mathcal F}}
\newcommand{\G}{{\mathcal G}}

\newcommand{\K}{{\mathcal K}}

\newcommand{\M}{{\mathcal M}}

\newcommand{\fM}{{\mathfrak M}}

\renewcommand{\P}{{\mathcal P}}
\newcommand{\PP}{{\mathbb P}}

\newcommand{\QQ}{{\mathbb Q}}
\newcommand{\R}{{\mathcal R}}
\newcommand{\RR}{{\mathbb R}}

\newcommand{\fS}{{\mathfrak S}}

\newcommand{\U}{{\mathcal U}}
\newcommand{\V}{{\mathcal V}}
\newcommand{\W}{{\mathcal W}}
\newcommand{\X}{{\mathcal X}}
\newcommand{\Y}{{\mathcal Y}}

\newcommand{\ZZ}{{\mathbb Z}}
\newcommand{\RealNumbers}{{\mathbb R}}
\newcommand{\Integers}{{\mathbb Z}}

\newcommand{\LieAlg}[1]{{\mathfrak #1}}

\newcommand{\Mon}{{\rm Mon}}

\newcommand{\RightArrowOf}[1]{\stackrel{#1}{\rightarrow}}
\newcommand{\LeftArrowOf}[1]{\stackrel{#1}{\leftarrow}}
\newcommand{\LongLeftArrowOf}[1]{\stackrel{#1}{\longleftarrow}}

\newcommand{\LongRightArrowOf}[1]{\stackrel{#1}{\longrightarrow}}

\newcommand{\StructureSheaf}[1]{{\mathcal O}_{#1}}

\newcommand{\restricted}[2]{#1_{\mid_{#2}}}

\newcommand{\rank}{{\rm rank}}

\newcommand{\Pic}{{\rm Pic}}
\newcommand{\Pont}{{\rm Pont}}
\newcommand{\Sym}{{\rm Sym}}

\newcommand{\Hom}{{\rm Hom}}
\newcommand{\Aut}{{\rm Aut}}
\newcommand{\End}{{\rm End}}

\newcommand{\SheafHom}{{\mathcal H}om}
\newcommand{\SheafEnd}{{\mathcal E}nd}

\newcommand{\Ideal}[1]{{\mathcal I}_{#1}}
\newcommand{\Hdg}{{\mathcal Hdg}}

\renewcommand{\span}{{\rm span}}

\begin{document}
\title{Rational Hodge isometries of hyper-K\"{a}hler varieties of $K3^{[n]}$-type are algebraic}
\author{Eyal Markman}
\address{Department of Mathematics and Statistics, 
University of Massachusetts, Amherst, MA 01003, USA}
\email{markman@math.umass.edu}

\date{\today}


\begin{abstract}
Let $X$ and $Y$ be compact hyper-K\"{a}hler manifolds deformation equivalence to the Hilbert scheme of length $n$ subschemes of a $K3$ surface. A class in $H^{p,p}(X\times Y,\QQ)$ is an {\em analytic correspondence}, if it belongs to the subring generated by Chern classes of coherent analytic sheaves.
Let $f:H^2(X,\QQ)\rightarrow H^2(Y,\QQ)$ be a rational Hodge isometry with respect to the Beauville-Bogomolov-Fujiki pairings. We prove that $f$ is induced by an analytic correspondence. We furthermore lift $f$ to an analytic correspondence $\tilde{f}: H^*(X,\QQ)[2n]\rightarrow H^*(Y,\QQ)[2n]$, which is a Hodge isometry 
with respect to the Mukai pairings and which preserves the gradings up to sign. When $X$ and $Y$ are projective the correspondences $f$ and $\tilde{f}$ are algebraic.
\end{abstract}

\maketitle
\tableofcontents

%
\section{Introduction}
\label{sec-introduction}
%
\subsection{Algebraicity of rational Hodge isometries}
An {\em irreducible holomorphic symplectic manifold (IHSM)} is a simply connected compact K\"{a}hler manifold $X$ such that $H^0(X,\Omega^2_X)$ is spanned by a nowhere degenerate holomorphic $2$-form. Such manifolds admit hyper-K\"{a}hler structures and are examples of compact hyper-K\"{a}hler manifolds. When two-dimensional, such a manifold is a $K3$ surface. If $X$ is a K\"{a}hler manifold which is deformation equivalent to the Hilbert scheme $S^{[n]}$ of length $n$ subschemes of a $K3$ surface $S$, then
$X$ is an irreducible holomorphic symplectic manifold \cite{beauville-varieties-with-zero-c-1}. The latter are said to be of {\em $K3^{[n]}$-type}. The second integral cohomology $H^2(X,\Integers)$ of an irreducible holomorphic symplectic manifold $X$
is endowed with a symmetric integral primitive non-degenerate bilinear pairing of signature $(3,b_2(X)-3)$ known as the Beauville-Bogomolov-Fujiky (BBF) pairing, where $b_2(X)$ is the second Betti number  \cite{beauville-varieties-with-zero-c-1}. A homomorphism $f:H^2(X,\Integers)\rightarrow H^2(Y,\Integers)$ between the second cohomologies of two  irreducible holomorphic symplectic manifolds is said to be an {\em integral Hodge isometry}, if it is an isometry with respect to the BBF-pairings as well as an isomorphism of Hodge structures. If $f:H^2(X,\QQ)\rightarrow H^2(Y,\QQ)$ has these properties it is said to be a {\em rational Hodge isometry}.

 Let $X$ be an irreducible holomorphic symplectic  manifold. 
 Given a class $\alpha$ in the even cohomology $H^{ev}(X,\QQ)$ of $X$, denote by 
 $\alpha_i$ its graded summand in $H^{2i}(X,\QQ)$. Let
 $\alpha^\vee\in H^{ev}(X,\QQ)$ be the class satisfying $(\alpha^\vee)_i=(-1)^i\alpha_i$.
 The {\em Mukai pairing} on the even cohomology $H^{ev}(X,\QQ)$ is defined by
 \[
 (\alpha,\beta) := \int_X \alpha^\vee\cup \beta.
 \]
 The {\em Mukai vector} of an object $F$ in the bounded derived category $D^b(X)$ of coherent sheaves on $X$ is $v(F):=ch(F)\sqrt{td_X}\in H^{ev}(X,\QQ).$ 
 Grothendieck-Riemann-Roch yields 
 \[
 (v(E),v(F)):=\chi(E^\vee\otimes F)
 \]
 for objects $E$ and $F$ in $D^b(X)$, where $E^\vee:=R\SheafHom(E,\StructureSheaf{X})$ is the derived dual object. 
 When $X$ is of $K3^{[n]}$-type the odd cohomology vanishes and $H^{ev}(X,\ZZ)=H^*(X,\ZZ)$ \cite{markman-integral-generators}. 
 
 Let $X$ be a $2n$-dimensional IHSM.
Denote by $H^*(X,\QQ)[2n]$ the shifted cohomology of $X$, where $H^k(X,\QQ)[2n]=H^{k+2n}(X,\QQ)$.
We say that a homomorphism $f:H^*(X,\QQ)[2n]\rightarrow H^*(Y,\QQ)[2n]$ is {\em degree reversing}, if it maps 
$H^k(X,\QQ)[2n]$ to $H^{-k}(Y,\QQ)[2n]$, for all $k$. The homomorphism is {\em degree preserving up to sign}, if it is either degree preserving or degree reversing.

Let $X$ and $Y$ be IHSMs of $K3^{[n]}$-type and $f:H^2(X,\QQ)\rightarrow H^2(Y,\QQ)$ a Hodge isometry. 

\begin{thm}
\label{main-thm}
There exists an analytic correspondence $\tilde{f}:H^*(X,\QQ)\rightarrow H^*(Y,\QQ)$, which is an isometry with respect to the Mukai pairings and satisfies one of the following:
\begin{enumerate}
\item
$\tilde{f}$ is degree preserving and it restricts to $H^2(X,\QQ)$ as $f$, 
\item
$\tilde{f}$ is degree reversing and the composition 
\[
H^*(X,\QQ)\LongRightArrowOf{c_2(X)^{n-1} \cup}
H^*(X,\QQ)\LongRightArrowOf{\tilde{f}} H^*(X,\QQ)
\]
of $\tilde{f}$ with cup product with $c_2(X)^{n-1}$
restricts to $H^2(X,\QQ)$ as a non-zero rational multiple of $f$. 
\end{enumerate}
In particular, $f$ is algebraic whenever $X$ and $Y$ are projective.
\end{thm}

The Theorem is proved in section \ref{sec-algebraicity}.
When $n=1$ and $X$  and $Y$ are $K3$  surfaces Theorem \ref{main-thm} was proved by Mukai for projective surfaces with Picard number $\geq 11$ \cite{mukai-hodge}, by Nikulin for projective surfaces with Picard number $\geq 5$ \cite{nikulin}, announced by Mukai in \cite{mukai-ICM}, and proved by Buskin in complete generality \cite{buskin}. Another proof for projective $K3$ surfaces $X$ and $Y$, without further restrictions on their Picard numbers, is due to Huybrechts \cite{huybrechts-rational-hodge-isometries}.

A more direct relationship between the Hodge isometries $f$ and $\tilde{f}$ in Theorem \ref{main-thm} is described in terms of Taelman's LLV lattice in Theorem \ref{thm-lifting-f-to-a-morphism-in-G} below.
\begin{defi}
Let $X_1$ and $X_2$ be deformation equivalent compact K\"{a}hler manifolds. An isomorphism  $f:H^*(X_1,\ZZ)\rightarrow H^*(X_2,\ZZ)$ is said to be a {\em parallel-transport operator}, if it arrises as follows.
There exist some smooth and proper family of compact K\"{a}hler manifolds $\pi:\X\rightarrow B$ over an analytic space $B$,  points $b_1$, $b_2$ in $B$, isomorphisms $g_i:X_i\rightarrow \X_{b_i}$ with the fibers $\X_{b_i}$ of $\pi$ over $b_i$, $i=1,2$, and a continuous path $\gamma$ from $b_1$ to $b_2$, such that parallel-transport along the path $\gamma$ in the local system $R\pi_*\Integers$ induces
$g_{2,*}\circ f\circ g_1^*:H^*(\X_{b_1},\ZZ)\rightarrow H^*(\X_{b_2},\ZZ)$. When $X_1=X_2$ the parallel-transport operator $f$ is also called a {\em monodromy operator}.
The {\em monodromy group} $\Mon(X)$ of a compact K\"{a}hler manifold $X$ is the subgroup of $GL(H^*(X,\Integers))$ consisting of all monodromy operators. 
\end{defi}

Recall that a groupoid is a category all of whose morphisms are isomorphisms. 
 Let 
 \begin{equation}
\label {eq-Groupoid-G}
 \G
 \end{equation} 
 be the groupoid, whose objects are pairs $(X,\epsilon)$ consisting of an irreducible holomorphic symplectic manifold $X$ and an orientation\footnote{The morphisms of $\G$ do not depend on the orientation. The orientation is not needed if $n:=\dim(X)/2$ is odd. If $n$ is even, the orientation is used in \cite{taelman} in the definition of the functor $\widetilde{H}$ introduced below in Equation (\ref{eq-functor-widetilde-H}).} 
 $\epsilon$ of $H^2(X,\QQ)$. If $\dim(X)$ is divisible by $4$, assume\footnote{The assumption holds for all known irreducible holomorphic symplectic manifolds.} that $b_2(X)$ is odd. 
Morphisms in $\Hom_\G((X,\epsilon),(Y,\epsilon'))$ are compositions of three types of rational isometries from $H^*(X,\QQ)$ to $H^*(Y,\QQ)$ with respect to the Mukai pairing:
\begin{enumerate}
\item
\label{item-parallel-transport-operator}
A parallel-transport operator $f:H^*(X,\QQ)\rightarrow H^*(Y,\QQ)$. 
\item
\label{item-correspondence-of-FM-kernel}
The isometry $H^*(X,\QQ)\rightarrow H^*(Y,\QQ)$ induced by the correspondence $ch(\P)\sqrt{td_{X\times Y}}$,
where $\P\in D^b(X\times Y)$ is the Fourier-Mukai kernel of an equivalence $D^b(X)\rightarrow D^b(Y)$ of derived categories between two deformation equivalent IHSMs $X$ and $Y$.
\item
\label{item-exponential-of-rational-1-1-class}
An isometry of $H^*(X,\QQ)$ induced by cup product with $\exp(\lambda):=\sum_{k=0}^{\dim(X)}\lambda^k/k!$ for some $\lambda\in H^2(X,\QQ)$. 
\end{enumerate}
The set of morphisms  $\Hom_\G((X,\epsilon),(Y,\epsilon'))$ is independent of the orientations and will be denoted also by $\Hom_\G(X,Y)$. Note that the latter is empty, by definition, if $X$ and $Y$ are not deformation equivalent.

\begin{defi}
\label{def-kappa}
Given an object $F$ of positive rank $r$ in the derived category of coherent sheaves, set $\kappa(F):=ch(F)\exp\left(-c_1(F)/r\right)$, where
$\exp(x)=\sum_{j=0}^\infty x^j/j!$. Then $\kappa(F)^r=ch(F^{\otimes r}\otimes\det(F)^{-1})$, where the tensor power is taken in the derived category. If $F$ is an object of positive rank $r$ in the derived category of twisted coherent sheaves with respect to some Brauer class, then $F^{\otimes r}\otimes\det(F)^{-1}$ is an untwisted object and we define $\kappa(F)$ as the $r$-th root of the Chern character of the latter with degree $0$ summand equal to $r$. If $F$ has negative rank, set $\kappa(F):=-\kappa(F[1])$.
\end{defi}

Let 
\begin{equation}
\label{eq-G-an}
\G_{an}
\end{equation} 
be the subgroupoid of $\G$ with the same objects. Morphisms of $\G_{an}$ are compositions of two types of rational isometries: (1) Parallel-transport operators $f:H^*(X,\QQ)\rightarrow H^*(Y,\QQ)$, which are Hodge isometries.
(2) Hodge isometries 
\begin{equation}
\label{eq-morphism-of-G-an-associated-to-a-locally-free-FM-kernel}
[\kappa(\P)\sqrt{td_{X\times Y}}]_*:H^*(X,\QQ)\rightarrow H^*(Y,\QQ)
\end{equation}
induced by classes $\kappa(\P)\sqrt{td_{X\times Y}}$, where $\P$ is a locally free twisted sheaf over $X\times Y$, such that the Azumaya $\StructureSheaf{X\times Y}$-algebra $\SheafEnd(\P)$ is a deformation (in the sense of Definition \ref{definition-deformation-equivalent-Azumaya-algebras}) of the endomorphism algebra of a locally free Fourier-Mukai kernel of an equivalence of derived categories.\footnote{
Note that $[\kappa(\P)\sqrt{td_{X\times Y}}]_*$ is indeed a morphism of $\G$, so that $\G_{an}$ is indeed a subgroupoid of $\G$. This is seen as follows.
If $(X\times Y,\SheafEnd(\P))$ is a deformation of $(X_0\times Y_0,\SheafEnd(\P_0))$ and $\P_0$ is a locally free untwisted sheaf, which is the
Fourier-Mukai kernel of an equivalence of derived categories, then 
$[\kappa(\P)\sqrt{td_{X\times Y}}]_*=f_2\circ [\kappa(\P_0)\sqrt{td_{X_0\times Y_0}}]_*\circ f_1$, where $f_1$ and $f_2$ are parallel-transport operators, and so morphisms in $\G$. Furthermore, writing $\rank(\P_0)=r$ and $c_1(\P_0)=\pi_{X_0}^*(\lambda_1)+\pi_{Y_0}^*(\lambda_2)$ we see that 
$[\kappa(\P_0)\sqrt{td_{X_0\times Y_0}}]_*(\bullet)=\exp(-\lambda_2/r)\cup [ch(\P_0)\sqrt{td_{X_0\times Y_0}}]_*(\exp(-\lambda_1/r)\cup(\bullet))$ is the composition of three morphisms in $\G$.
}
 Morphisms in $\G_{an}$ are induced by analytic correspondences. Indeed, this follows from the Torelli theorem for parallel-transport Hodge isometries (Theorem \ref{thm-two-inseparable-marked-pairs} and Lemma \ref{lemma-parallel-transports-with-same-restriction-to-H-2}) and by definition for morphisms of type (2) above. 

The {\em rational LLV lattice} of an irreducible holomorphic symplectic manifold $X$ is the vector space
\[
\widetilde{H}(X,\QQ) :=\QQ\alpha\oplus H^2(X,\QQ)\oplus\QQ\beta
\]
endowed with the non-degenerate symmetric bilinear pairing, which restricts to  $H^2(X,\QQ)$ as the BBF-pairing, such that  $\alpha$ and $\beta$ are isotropic and orthogonal to $H^2(X,\QQ)$ and satisfy $(\alpha,\beta)=-1$. An orientation for $H^2(X,\QQ)$  determines an orientation of $\widetilde{H}(X,\QQ)$, and so the LLV lattice of objects of $\G$ has a chosen orientation. 
We assign $\widetilde{H}(X,\QQ)$ the grading where $\alpha$ has degree $-2$, $\beta$ has degree $2$, and $H^2(X,\QQ)$
has degree $0$.

Let $\widetilde{\G}$ be the groupoid, whose objects are irreducible holomorphic symplectic manifolds and such that 
morphisms are isometries of their rational LLV lattices.
 Results of Taelman's yield a functor 
 \begin{equation}
 \label{eq-functor-widetilde-H}
 \widetilde{H}:\G\rightarrow \widetilde{\G}
 \end{equation} 
 (see Definition \ref{def-functor-tilde-H}).

We show that the morphism (\ref{eq-morphism-of-G-an-associated-to-a-locally-free-FM-kernel}) is degree reversing
 (Lemma \ref{lemma-phi-is-degree-reversing}). It follows that 
every morphism $\phi$  in $\Hom_{\G_{an}}((X,\epsilon),(Y,\epsilon'))$ preserves the grading up to sign. Consequently, $\widetilde{H}(\phi)$ preserves the grading of $\widetilde{H}(X,\QQ)$ up to sign and maps $H^2(X,\QQ)$ to $H^2(Y,\QQ)$ and restricts to $H^2(X,\QQ)$ as a Hodge isometry
$\widetilde{H}(\phi)_0$ (Lemma \ref{lemma-preservation-of-grading-of-LLV-lattice-implies-that-of-cohomology}). 

\medskip
Let $\Hdg$ be the groupoid, whose objects are irreducible holomorphic symplectic manifolds, and whose morphisms in $\Hom_{\Hdg}(X,Y)$ are rational Hodge isometries $f:H^2(X,\QQ)\rightarrow H^2(Y,\QQ)$ preserving the orientations of the positive cones (Definition \ref{def-orientation-of-the-positive-cone}). We get the functor
\begin{equation}
\label{eq-functor-widetilde-H_0}
\widetilde{H}_0:\G_{an}\rightarrow \Hdg
\end{equation}
forgetting the orientation of objects and mapping a morphism $\phi$ to $\widetilde{H}(\phi)_0$ or $-\widetilde{H}(\phi)_0$ (whichever preserves the orientations of the positive cones).
Let $\G_{an}^{[n]}$ be the full subgroupoid of $\G_{an}$, whose objects are of $K3^{[n]}$-type. Define the subgroupoid $\Hdg^{[n]}$ of $\Hdg$ analogously. The following theorem is proved in section \ref{sec-algebraicity}.

\begin{thm}
\label{thm-lifting-f-to-a-morphism-in-G}
The functor $\widetilde{H}_0:\G_{an}^{[n]}\rightarrow\Hdg^{[n]}$ is full.\footnote{This means that $\widetilde{H}_0:\Hom_{\G_{an}^{[n]}}(X_1,X_2)\rightarrow \Hom_{\Hdg}(X_1,X_2)$ is surjective, for any two objects $X_1$ and $X_2$ of $\G_{[an]}^{[n]}$.} 
In particular,  the analytic correspondence $\tilde{f}$ of Theorem \ref{main-thm} can be chosen 
to be a morphism in $\Hom_{\G^{[n]}_{an}}((X,\epsilon),(Y,\epsilon'))$
satisfying  $\widetilde{H}_0(\tilde{f})=f$, possibly after replacing $f$ by $-f$, so that it would preserve the orientations of the positive cones.
\end{thm}


Let $f:H^2(X,\QQ)\rightarrow H^2(Y,\QQ)$ be a rational isometry between two IHSMs of $K3^{[n]}$-type $X$ and $Y$.
We say that $f$ is of {\em $r$-cyclic type}, for some positive integer $r$, if there exists a primitive class $u\in H^2(X,\Integers)$, satisfying $(u,u)=2r$ and $(u,\bullet)\in H^2(X,\Integers)^*$ is primitive as well, and there exists a 
parallel-transport operator $g:H^2(X,\Integers)\rightarrow H^2(Y,\Integers)$, such that $f=-g\rho_u$, where $\rho_u\in O(H^2(X,\QQ))$ is the reflection in $u$ given by $\rho_u(x)=x-\frac{2(u,x)}{(u,u)}u.$ Such $f$ is necessarily compatible with the orientations of the positive cones.
The proof of Theorem \ref{thm-lifting-f-to-a-morphism-in-G} relies on the following.

\begin{thm}
\label{thm-introduction-algebraicity-in-the-double-orbit-of-a-reflection}
(Corollary \ref{cor-algebraicity-in-the-double-orbit-of-a-reflection}).
Let $f:H^2(X,\QQ)\rightarrow H^2(Y,\QQ)$ be a rational Hodge isometry of $r$-cyclic type 
between two IHSMs $X$ and $Y$ of $K3^{[n]}$-type.
There exists over $X\times Y$ a (possibly twisted) locally free coherent sheaf $E$ of rank $n!r^n$, such that 
$\phi:=[\kappa(E)\sqrt{td_{X\times Y}}]_*$ is a degree reversing morphism in $\Hom_{\G^{[n]}_{an}}(X,Y)$ and $\widetilde{H}_0(\phi)=f$.
\end{thm}

%
\subsection{Fourier-Mukai kernels of positive rank and the Lefschetz Standard Conjecture}
Theorem \ref{thm-introduction-algebraicity-in-the-double-orbit-of-a-reflection} leads to an alternative short proof of the Lefschetz standard conjecture for projective IHSMs of $K3^{[n]}$-type, originally proved in \cite{charles-markman}. Let $X$ be a projective IHSM of $K3^{[n]}$-type. Let $\lambda\in H^{1,1}(X,\QQ)$ be an ample class. Let $h\in \End(H^*(X,\QQ))$ multiply $H^{k}(X,\QQ)[2n]$ by $k$.
Denote by $e_\lambda\in\End(H^*(X,\QQ))$ the Lefschetz operator of cup product with $\lambda$ and let $e^\vee_\lambda\in \End(H^*(X,\QQ))$ be dual Lefschetz operator, i.e., the unique element satisfying
\[
[e_\lambda,e_\lambda^\vee]=h, \ [h,e_\lambda^\vee]=-2e_\lambda^\vee.
\]
Choose an IHSM $Y$, such that there exists a rational Hodge isometry $f:H^2(X,\QQ)\rightarrow H^2(Y,\QQ)$ of $r$-cyclic type, for some some integer $r>0$. The existence of such $Y$ follows, 
for every $r>0$, by Theorem \ref{thm-introduction-algebraicity-in-the-double-orbit-of-a-reflection} and the surjectivity of the period map\footnote{
Choose an isometry $\eta_X:H^2(X,\Integers)\rightarrow \Lambda$ with a fixed lattice $\Lambda$. The period map sends $(X,\eta)$ to $\eta(H^{2,0}(X))\in \PP(\Lambda\otimes_\ZZ\CC)$. Choose $Y$ to be an IHSM of $K3^{[n]}$-type with an isometry $\eta_Y:H^2(Y,\Integers)\rightarrow\Lambda$, such that $(X,\eta_X)$ and $(Y,\eta_Y)$ belong to the same connected component of the moduli space of marked IHSM and such that $\eta_Y(H^{2,0}(Y))=\eta_X(\rho_u(H^{2,0}(X)))$, for a class $u\in H^2(X,\ZZ)$ as in Theorem \ref{thm-introduction-algebraicity-in-the-double-orbit-of-a-reflection}. Then $f:=\eta_Y^{-1}\eta_X(-\rho_u):H^2(X,\QQ)\rightarrow H^2(Y,\QQ)$ is a rational Hodge isometry of $r$-cyclic type preserving the orientations of the positive cones.
} 
\cite{huybrects-basic-results}.
The morphism $\phi\in \Hom_{\G^{[n]}_{an}}(X,Y)$  in Theorem \ref{thm-introduction-algebraicity-in-the-double-orbit-of-a-reflection} is degree reversing. We get that $e_\lambda^\vee$ is algebraic, by the following elementary lemma.

Let $X$ and $Y$ be IHSMs and let $\phi\in\Hom_{\G_{an}}((X,\epsilon),(Y,\epsilon'))$ be a degree reversing morphism.
Then $\tilde{\phi}:=\widetilde{H}(\phi)$ is degree reversing as well (Lemma \ref{lemma-preservation-of-grading-of-LLV-lattice-implies-that-of-cohomology}). Let $t\in\QQ$ be such that $\tilde{\phi}(\beta)=t\alpha$.
Let $\lambda\in H^2(X,\QQ)$ be such that $(\lambda,\lambda)\neq 0$. 

\begin{lem}
\label{lemma-Lefschetz}
The dual Lefschetz operator  is given by 
$
e_\lambda^\vee=\frac{2}{t(\lambda,\lambda)}\phi^{-1}e_{\tilde{\phi}(\lambda)}\phi
$
and is algebraic if $\lambda$ is, as $e_{\tilde{\phi}(\lambda)}$ and the morphisms $\phi$ and $\phi^{-1}$ are.
\end{lem}

The short proof of the lemma is given at the end of Section \ref{sec-the-degree-reversing-Hodge-isometry-of-a-FM-kernel}.

\begin{example}
If $X$ is an IHSM of $K3^{[n]}$-type, $n=1,2$, then $\Aut_{\G^{[n]}_{an}}(X)$ contains a degree reversing involution. When $X$ is a $K3$ surface such an involution is induced by 
the ideal sheaf of the diagonal in $X\times X$, as the ideal sheaf is the kernel of a derived auto-equivalence. When $X$ is of $K3^{[2]}$-type,
a rank $2$ twisted reflexive sheaf $E$ was constructed on $X\times X$ in \cite{markman-BBF-class-as-characteristic-class}.
When $X=S^{[2]}$, for a $K3$ surface $S$, Addington showed that $E$ is the Fourier-Mukai kernel of an auto-equivalence of $D^b(S^{[2]})$ acting on cohomology as an involution \cite{addington}. It follows that $[\kappa(E)\sqrt{td_{X\times X}}]_*$
is a degree reversing involution in  $\Aut_{\G^{[2]}_{an}}(X)$, for every $X$ of $K3^{[2]}$-type. 
Does a degree reversing automorphism exist in $\Aut_{\G^{[n]}_{an}}(X)$ for every $X$ of $K3^{[n]}$-type for $n\geq 3$? What about other deformation types?
\end{example}
%
\subsection{The Pontryagin product}
In Section \ref{sec-pontryagin} we define a {\em Pontryagin product} $\star$ on the cohomology $H^*(X,\QQ)$ of every irreducible holomorphic symplectic manifold $X$ 
with vanishing odd cohomology.
The unit with respect to $\star$ is
 $c_X[pt]/n!$, where $\dim(X)=2n$, $[pt]\in H^{top}(X,\QQ)$ is the class Poincar\'{e} dual to the class of a point, and $c_X$ is the Fujiki constant.\footnote{
The {\em Fujiki constant} $c_X$ is the positive rational number such that the equality 
$\int_X\lambda^{2n}=\frac{(2n)!}{2^n n!}c_X(\lambda,\lambda)^n
$
holds for all $\lambda\in H^2(X,\QQ)$. The Fujiki constant is calculated for all known irreducible holomorphic symplectic manifolds in \cite{rapagnetta}. 
If $X$ of $K3^{[n]}$-type, then $c_X=1$.
} 
The product  $\star$ maps $H^k(X,\QQ)\otimes H^l(X,\QQ)$ to $H^{k+l-4n}(X,\QQ)$, so that $\star$ is compatible with the grading for which $H^k(X,\QQ)$ has degree $4n-k$. 
 Parallel-transport operators induce isomorphisms of the Pontryagin rings 
 (Lemma \ref{lemma-Pontyagin-product-independent-of-choice-of-monodromy-reflection}).

The Pontryagin product $\star$ on the cohomology $H^*(S,\ZZ)$ of a $K3$ surface $S$ is simple to describe. The unit with respect to $\star$ is the class $[pt]\in H^4(X,\ZZ)$ Poincar\'{e} dual to a point. The product $\star$ maps
$H^k(S,\ZZ)\otimes H^l(S,\ZZ)$ to $H^{k+l-4}(S,\ZZ)$ and satisfies $\lambda_1\star \lambda_2=(\lambda_1,\lambda_2)\in H^0(S,\ZZ)$, for classes $\lambda_1,\lambda_2\in H^2(S,\ZZ)$. Let
$
\rho_u:\widetilde{H}(S,\Integers)\rightarrow \widetilde{H}(S,\Integers),
$
 be the reflection $\rho_u(r,c,s)=(-s,c,-r)$ in the Mukai vector $u:=(1,0,1)$ of the structure sheaf $\StructureSheaf{S}$. 
Then $-\rho_u$ conjugates the usual product on $H^*(S,\Integers)$ to the Pontryagin product. Note that $-\rho_u$ corresponds to the action on cohomology of the autoequivalence $\Phi_{\Ideal{\Delta}}$ of $D^b(S)$, whose Fourier-Mukai kernel is the ideal sheaf $\Ideal{\Delta}$ of the diagonal $\Delta$ in $S\times S$.

Morphisms in $\G_{[an]}$, which are associated to equivalences of derived categories with a Fourier-Mukai kernel of positive rank, are degree reversing. This corresponds to a normalization which annihilates the first Chern class of the Fourier-Mukai  kernel. We introduce  a further normalization, which assures that the compositions of any two such degree reversing morphisms would map the usual identity $1\in H^0(X,\QQ)$ to itself. We speculate below that thus normalized, the degree reversing morphisms conjugate the cup product to the Pontryagin product. We introduce this further normalization in the next paragraph.

Let $\QQ^\times$ be the multiplicative group of rational numbers. Let
\[
\chi:\G_{an}\rightarrow \QQ^\times
\]
be the groupoid homomorphism sending every object to $\QQ^\times$ and sending a morphism $\phi$
to $t$, if $\phi$ is degree preserving and $\widetilde{H}(\phi)(\alpha)=t\alpha$, or if $\phi$ is degree reversing and
$\widetilde{H}(\phi)(\alpha)=\frac{1}{t}\beta.$
Given a non-zero rational number $t$, denote by $\mu_t:H^*(X,\QQ)[2n]\rightarrow H^*(X,\QQ)[2n]$ 
the graded linear transformation multiplying $H^{2k}(X,\QQ)[\dim(X)]$ by $t^k$. 
Let 
\begin{equation}
\label{eq-subgroupoid-R-[n]}
\R^{[n]}_{an}
\end{equation}
be the subgroupoid
of $\G^{[n]}_{an}$ with the same objects, but whose morphisms $\phi$ have the property that $\mu_{\chi(\phi)}\circ\phi$ conjugates the usual product on $H^*(X,\QQ)$ to itself, if $\phi$ is degree-preserving, and to $\star$ if $\phi$ is degree reversing. 

\begin{conj}
\label{conj-Pontryagin}
The subgroupoid $\R^{[n]}_{an}$ is in fact the whole of $\G^{[n]}_{an}$. Furthermore, if an object of positive rank $F$ in $D^b(X\times Y)$  is the Fourier-Mukai kernel of an equivalence $\Phi_F:D^b(X)\rightarrow D^b(Y)$ 
of derived categories of projective irreducible holomorphic symplectic manifolds $X$ and $Y$ of $K3^{[n]}$-type, then the rank of $F$ is $n!t^n$, for some rational number $t$, and 
setting $\phi:=\mu_t\circ [\kappa(F)\sqrt{td_{X\times Y}}]_*:H^*(X,\QQ)\rightarrow H^*(Y,\QQ)$ we have
\begin{equation}
\label{eq-normalized-FM-transformation-conjugate-cup-product-to-Pontryagin}
\phi(\gamma_1\cup\gamma_2)=\phi(\gamma_1)\star\phi(\gamma_2),
\end{equation}
for all $\gamma_1,\gamma_2\in  H^*(X,\QQ)$.
\end{conj}

The above stated equality $\rank(F)=n!t^n$, for some rational number $t$, is known \cite[Theorem 4.14]{beckmann}.
As evidence for the conjecture we show that the examples of morphisms of $\G^{[n]}_{an}$ used in the proof of Theorem \ref{thm-lifting-f-to-a-morphism-in-G} all satisfy Conjecture \ref{conj-Pontryagin}. 
Consequenly,  Theorem \ref{thm-lifting-f-to-a-morphism-in-G} holds even after replacing $\G^{[n]}_{an}$ by $\R^{[n]}_{an}$
and we get the following result proven in Section \ref{sec-pontryagin}.

\begin{prop} (Proposition 
\ref{prop-restriction-of-wildtilde-H-0-is-full-as-well}). 
The restriction of the functor $\widetilde{H}_0:\G^{[n]}_{an}\rightarrow\Hdg^{[n]}$ to the subgroupoid $\R^{[n]}_{an}$  remains full. 
\end{prop}

\hide{
assuming a weaker Conjecture \ref{conjecture-Theta-n-is-integration-of-infinitesimal-LLV-action}. 
We verify Equation (\ref{eq-normalized-FM-transformation-conjugate-cup-product-to-Pontryagin})
in these examples, without assuming Conjecture \ref{conjecture-Theta-n-is-integration-of-infinitesimal-LLV-action},
whenever $\gamma_1$ and $\gamma_2$ belong to the subring generated by $H^2(X,\QQ)$. 
When $n=2$ the ring $H^*(X,\QQ)$ is generated by $H^2(X,\QQ)$.
In particular, the restriction of the functor $\widetilde{H}_0$ to $\R^{[2]}_{an}$ is full, and Theorem \ref{thm-lifting-f-to-a-morphism-in-G} holds even after replacing $\G^{[2]}_{an}$ by $\R^{[2]}_{an}$ (see Remark \ref{rem-conjecture-ok-for-Verbitsky-component}).
}
 
 Let $X$ and $Y$ be  projective IHSMs of $K3^{[n]}$-type.
The proof of Proposition \ref{prop-restriction-of-wildtilde-H-0-is-full-as-well}  implies that the degree reversing correspondence $\kappa(E)\sqrt{td_{X\times Y}}$ in 
Theorem \ref{thm-introduction-algebraicity-in-the-double-orbit-of-a-reflection} satisfies Conjecture \ref{conj-Pontryagin}. 
Set $\phi:=\mu_t\circ [\kappa(E)\sqrt{td_{X\times Y}}]_*$, as in the conjecture.
Let $[\Delta]\in H^*(X^3,\QQ)$ the class of the small diagonal $\Delta$ in $X^3$.
The following Corollary is proved in Section \ref{sec-Pontryagin-product-fin-the-K3-[n]-type}.

\begin{cor}
\label{cor-Pontryagin-product-is-algebraic}
The Pontryagin product $\star: H^*(Y\times Y,\QQ)\rightarrow H^*(Y,\QQ)$ is induced by the algebraic class 
$\phi^3([\Delta])$ in $H^*(Y^3,\QQ)$. In particular, Pontryagin product  $\gamma\star :H^*(Y,\QQ)\rightarrow H^*(Y,\QQ)$ by an algebraic class $\gamma$ is induced by an algebraic correspondence.
\end{cor}

The algebraicity of the dual Lefschetz operator $e_\lambda^\vee$ in Lemma \ref{lemma-Lefschetz} is a special case of the above corollary, since $e_\lambda^\vee$ is a scalar multiple of the operator of Pontryagin product with the algebraic class 
$\phi^{-1}(\tilde{\phi}(\lambda))$.

%
\subsection{Outline of the proof of Theorem \ref{thm-lifting-f-to-a-morphism-in-G}}
The reader familiar with Buskin's proof of the case $n=1$ of the Theorem will notice the similarity of our strategy to his. 
Buskin considers examples of two dimensional moduli spaces $M$ of vector bundles over a $K3$ surface $S$ and the universal bundle $\U$ over $M\times S$. $M$ is a $K3$ surface and so the product $M\times S$ is a hyper-K\"{a}hler manifold. Buskin uses Verbitsky's results on hyperholomorphic vector bundles over hyper-K\"{a}hler manifolds to deform the pair $(M\times S,\U)$ to 
twisted vector bundles $(X_1\times X_2,\U')$ over products of not necessarily projective $K3$ surfaces $X_1$ and $X_2$. 
Buskin proves that the groupoid $\Hdg^{[1]}$ is generated by parallel-transport operators, which are isomorphisms of Hodge structures, and by the restrictions to $H^2(X_1,\QQ)$ of the degree reversing isometries $[\kappa(\U')\sqrt{td_{X_1\times X_2}}]_*:\widetilde{H}(X_1,\QQ)\rightarrow \widetilde{H}(X_2,\QQ)$. 

We first construct the functor 
$\widetilde{H}_0:\G_{an}\rightarrow \Hdg$  and establish its properties in Sections \ref{section-LLV} and \ref{sec-the-degree-reversing-Hodge-isometry-of-a-FM-kernel}. The construction depends heavily on the recent work of Taelman \cite{taelman}, as well as on the work of the author and independently of Beckmann describing the action of equivalences of derived categories on the LLV lattice \cite{beckmann,markman-modular}. 
In section \ref{sec-universal-bundle-over-product-of-Hilbert-schemes} we lift the equivalence of derived categories $\Phi_\U:D^b(M)\rightarrow D^b(S)$, with the universal bundle $\U$ as a kernel, to an equivalence $\Phi_{\U^{[n]}}:D^b(M^{[n]})\rightarrow D^b(S^{[n]})$ of derived categories of Hilbert schemes. This is done by conjugating the equivalence of equivariant derived categories on the cartesian products $\Phi_{\U^{\boxtimes n}}:D^b_{\fS_n}(M^n)\rightarrow D^b_{\fS_n}(S^n)$, obtained from the cartesian power of $\Phi_\U$, with the Briedgeland-King-Reid equivalence $BKR:D^b_{\fS_n}(S^n)\rightarrow D^b(S^{[n]})$ and its analogue for $M$. The Briedgeland-King-Reid equivalence is reviewed in Section \ref{sec-BKR}. 
The assignment $S\mapsto S^{[n]}$ and $\Phi_\U\mapsto \Phi_{\U^{[n]}}$ extends to define a functor 
\begin{equation}
\label{eq-introduction-functor-Theta-n}
\Theta_n:\G_{an}^{[1]}\rightarrow \G_{an}^{[n]}
\end{equation}
constructed in (\ref{eq-functor-Theta-n}).
The Fourier-Mukai kernel $\U^{[n]}$ is again a locally free sheaf over $M^{[n]}\times S^{[n]}$. 
Again we use Verbitsky's results on hyperholomorphic vector bundles on hyper-K\"{a}hler manifolds to deform the pair $(M^{[n]}\times S^{[n]},\U^{[n]})$ to pairs $(X_1\times X_2,E)$ consisting of a twisted vector bundle $E$  over the product $X_1\times X_2$ of not necessarily projective IHSMs of $K3^{[n]}$-type. The general techniques are developed in Section \ref{sec-hyperholomorphic-vector-bundles-deforming-a-FM-kernel} and applied to $\U^{[n]}$ in Section \ref{sec-stability-of-the-universal-bundle-over-product-of-hilbert-schemes}.

It remains to prove that we get a full functor already when we restrict the functor $\widetilde{H}_0:\G^{[n]}_{an}\rightarrow \Hdg^{[n]}$ to the subgroupoid of $\G_{an}^{[n]}$ generated by parallel-transport operators, which are Hodge isometries,  and morphisms in the image of $\Theta_n$. This is done as follows.

\underline{Step 1:} (Generators for the group of rational isometries)
Let $L$ be a lattice isometric to the second cohomology of a $K3$ surface. Let $\Lambda$ be the orthogonal direct sum $L\oplus \Integers\delta$, where $(\delta,\delta)=2-2n$. Then $\Lambda$ is isometric to the second cohomology of an irreducible holomorphic symplectic manifold $X$ of $K3^{[n]}$-type endowed with the BBF-pairing \cite{beauville-varieties-with-zero-c-1}. The monodromy group of $X$ acts on $H^2(X,\Integers)$ via a normal subgroup of $O(H^2(X,\Integers))$ and so corresponds to a well defined subgroup $\Mon(\Lambda)$ of $O(\Lambda)$, by \cite[Theorem 1.6 and Lemma 4.10]{markman-monodromy}. 
We first prove that the group $O(\Lambda_\QQ)$ of rational isometries is generated by $\Mon(\Lambda)$ and $O(L_\QQ)$ (Proposition \ref{prop-generators-of-the-isometry-group}).
It follows from the Cartan-Dieudonn\'{e} Theorem that every rational isometry $\psi:\Lambda_\QQ\rightarrow \Lambda_\QQ$, which preserves the orientation of the positive cone, decomposes as the composition $\psi=\psi_k\circ \cdots \circ \psi_1$,
where $\psi_i$ belongs to the double orbit $\Mon(\Lambda)(-\rho_{u})\Mon(\Lambda)$, where $\rho_u$ is the reflection in a primitive integral class $u\in L$ with $(u,u)>0$ (Corollary \ref{cor-decomposition-of-a-rational-isometry}). The double orbit $\Mon(\Lambda)(-\rho_{u})\Mon(\Lambda)$ depends only on $(u,u)$ and is otherwise independent of the class $u$, since so does the double orbit $O^+(L)(-\rho_{u})O^+(L),$ by  \cite[Prop. 3.3]{buskin}, and $\Mon(\Lambda)$ contains $O^+(L)$.  

\underline{Step 2:} (Reduction to the double orbit of a reflection)
Let $f:H^2(X,\QQ)\rightarrow H^2(Y,\QQ)$ be a rational Hodge isometry preserving the orientations of the positive cones, where $X$ and $Y$ are of $K3^{[n]}$-type. Choose isometries $\eta_X:H^2(X,\Integers)\rightarrow \Lambda$ and 
$\eta_Y:H^2(Y,\Integers)\rightarrow \Lambda$, so that the marked pairs $(X,\eta_X)$ and $(Y,\eta_Y)$ belong to the same connected component $\fM_\Lambda^0$ of the moduli space of marked IHSMs of $K3^{[n]}$-type.
Let $\Omega_\Lambda$ be the period domain of marked IHSMs of $K3^{[n]}$-type 
and let $P:\fM_\Lambda^0\rightarrow \Omega_\Lambda$ be the period map (see (\ref{eq-period-map}) for the definitions of $\Omega_\Lambda$ and $P$). Set
$\psi:=\eta_Y\circ f\circ \eta_X^{-1}\in O(\Lambda_\QQ)$ and choose a decomposition $\psi=\psi_k\circ \cdots \circ \psi_1$ as above. The group $O^+(\Lambda_\QQ)$ acts on $\Omega_\Lambda$ and $\psi(P(X,\eta_X))=P(Y,\eta_Y)$.
Set $\ell_i:=(\psi_i\circ \cdots\circ\psi_1)(P(X,\eta_X))$, $0\leq i \leq k.$
The surjectivity of the period map implies the existence of marked pairs $(X_i,\eta_i)\in \fM_\Lambda^0$, such that $P(X_i,\eta_i)=\ell_i$, where we choose $(X_0,\eta_0):=(X,\eta_X)$ and $(X_k,\eta_k):=(Y,\eta_Y)$. 
Set $f_i:=\eta_i^{-1}\psi_i\circ\eta_{i-1}:H^2(X_{i-1},\QQ)\rightarrow H^2(X_i,\QQ)$. Then $f=f_k\circ \cdots \circ f_1$.
If $k=0$, so that $\psi=\eta_Y\circ f\circ \eta_X^{-1}$ is the identity, then $f=\eta_Y^{-1}\eta_X$ is a parallel-transport operator and a Hodge isometry, and so $f$ lifts to a morphism  $\phi\in\Hom_{\G_{an}}(X,Y)$ with $\widetilde{H}_0(\phi)=f$, by Verbitsky's Torelli Theorem \ref{thm-two-inseparable-marked-pairs}.
It remains to prove that each $f_i$ belongs to the image of $\widetilde{H}_0$.

\underline{Step 3:} 
The equivalence of derived categories $\Phi_{\U^{[n]}}:D^b(M^{[n]})\rightarrow D^b(S^{[n]})$ described above 
yields a morphism in $\G^{[n]}_{an}$ which is mapped via $\widetilde{H}_0$ 
to a Hodge isometry $H^2(M^{[n]},\QQ)\rightarrow H^2(S^{[n]},\QQ)$ 
which in turn is simply the extension\footnote{
We note that $\widetilde{H}([ch(\U^{[n]})\sqrt{td_{M^{[n]}\times S^{[n]}}}]_*):\widetilde{H}(M^{[n]},\QQ)\rightarrow \widetilde{H}(S^{[n]},\QQ)$ is {\em not} the naive extension of $[ch(\U)\sqrt{td_{M\times S}}]_*:\widetilde{H}(M,\QQ)\rightarrow \widetilde{H}(S,\QQ)$, by \cite[Theorem 9.4]{taelman} in the $K3^{[2]}$-case and by \cite[Theorem 12.2]{markman-modular} and \cite[Theorem 7.4]{beckmann} in the $K3^{[n]}$-case, $n\geq 2$ (see 
Lemma \ref{lemma-composition-of-tilde-H-with-BKR-conjugate-of-powers-of-FM}). Nevertheless, when we replace $ch(\U^{[n]})$
by $\kappa(\U^{[n]})$ the resulting isometry $\widetilde{H}([\kappa(\U^{[n]})\sqrt{td_{M^{[n]}\times S^{[n]}}}]_*):\widetilde{H}(M^{[n]},\QQ)\rightarrow \widetilde{H}(S^{[n]},\QQ)$ {\em is} the naive extension of $[\kappa(\U)\sqrt{td_{M\times S}}]_*:\widetilde{H}(M,\QQ)\rightarrow \widetilde{H}(S,\QQ)$, by Corollary \ref{cor-psi-E-restricts-to-psi-U}.
} 
of the isometery $H^2(M,\QQ)\rightarrow H^2(S,\QQ)$ associated to $\Phi_\U$ (Corollary \ref{cor-psi-E-restricts-to-psi-U}). The resulting Hodge isometry $H^2(M^{[n]},\QQ)\rightarrow H^2(S^{[n]},\QQ)$ is thus in the double orbit 
$\Mon(\Lambda)(-\rho_{u})\Mon(\Lambda)$, where $\rho_u$ is the reflection in a primitive integral class $u\in L$ with $(u,u)>0$.
The proof of Corollary \ref{cor-psi-E-restricts-to-psi-U} relies on \cite[Theorem 12.2]{markman-modular} obtained independently by Beckman \cite[Theorem 7.4]{beckmann}. 
All such double orbits are obtained that way by varying the Mukai vector of the two-dimensional moduli space $M$, by \cite[Prop. 3.3]{buskin}.  
We prove that every rational Hodge isometry $f:H^2(X,\QQ)\rightarrow H^2(Y,\QQ)$, with $f$ in a fixed double orbit of a reflection in $u\in L$ and preserving the orientation of the positive cones, is of the form
$\widetilde{H}_0([\kappa(E)\sqrt{td_{X\times Y}}]_*)$, for a pair $(X\times Y,E)$, which is a deformation of $(M^{[n]}\times S^{[n]},\U^{[n]})$ for a suitable choice of $M$ (Corollary \ref{cor-algebraicity-in-the-double-orbit-of-a-reflection}).

%
\subsection{Why does the generalized Kummer deformation type require a modification?}
Let $X$ be an abelian surface. The fiber $Y$ over $0\in X$ of the albanese morphism $alb:X^{[n+1]}\rightarrow X$ is a $2n$-dimensional IHSM known as a {\em generalized Kummer variety} \cite{beauville-varieties-with-zero-c-1}. 
Section \ref{sec-rational-isometries} applies both to varieties of $K3^{[n]}$-type and to those of generalized Kummer deformation type. Sections \ref{section-LLV}
to \ref{sec-hyperholomorphic-vector-bundles-deforming-a-FM-kernel} are written for every IHSM. The lift of a derived equivalence of surfaces to their Hilbert schemes in Section \ref{sec-BKR} applies also for abelian surfaces. 
In particular, the functor $\widetilde{H}_0:\G_{an}\rightarrow \Hdg$ is well defined for the respective subgroupoids corresponding to $2n$-dimensional IHSMs of generalized Kummer deformation type, for all $n\geq 2$.
The difficulty is that for IHSMs of generalized Kummer deformation type we do not have an analogue of the functor $\Theta_n:\G^{[1]}\rightarrow\G^{[n]}$
(given in (\ref{eq-introduction-functor-Theta-n}) and constructed for $K3$ surfaces in Section 
\ref{sec-universal-bundle-over-product-of-Hilbert-schemes}). 

The albanese morphism 
$alb:X^{[n+1]}\rightarrow X$ is invariant under the action on $X^{[n+1]}$ of the 
subgroup $X[n+1]$ of $X$ generated by torsion points of order $n+1$. Hence,  $X[n+1]$ embeds as a subgroup $\Aut_0(Y)$ of the automorphism group of $Y$, as well as a subgroup $\Mon_0(Y)$ of its monodromy group. Indeed, $\Mon_0(Y)$  is the monodromy group of the isotrivial fibration $alb$. $\Aut_0(Y)$ is the subgroup of the automorphism group of $Y$ acting trivially on $H^k(Y,\Integers)$, for $k=2$ and $k=3$, by \cite[Theorem 3 and Corollary 5]{BNS} and
\cite[Theorem 1.3]{oguiso}. 

There are two obstructions to the construction of the naive analogue of $\Theta_n$. 
First, parallel-transport operators between abelian surfaces $X_1$ and $X_2$ lift only to cosets of parallel-transport operators of their $2n$-dimensional generalized Kummer varieties $Y_i$ by the subgroups $\Mon_0(Y_i)$  \cite[Theorem 1.4]{markman-generalized-kummers}. 
Furthermore, it is not clear that equivalences of derived categories of abelian surfaces induce equivalences of the derived categories of their generalized Kummers (see the discussion in Section 4.4 of Ploog's thesis \cite{ploog-thesis}). 

The group $\Aut_0(Y)$ varies with $Y$ in a local system over moduli,  by
\cite[Theorem 2.1]{hassett-tschinkel-lagrangian-planes}.
Denote by $H^*(Y,\QQ)^{\Aut_0(Y)}$ the subring of $H^*(Y,\QQ)$ consisting of elements invariant under $\Aut_0(Y)$.
The first obstruction mentioned above is resolved when we redefine the morphisms in $\G$ so that $\Hom_\G(Y_1,Y_2)$
consists of isometries from $H^*(Y_1,\QQ)^{\Aut_0(Y_1)}$ to $H^*(Y_2,\QQ)^{\Aut_0(Y_2)}$. 
We believe that the functor $\Theta_n$ has an analogue associating to every equivalence $\Phi$, of derived categories of abelian surfaces with a Fourier-Mukai kernel of non-zero rank, a degree reversing Hodge isometry $\Theta_n(\Phi)$ of the invariant subrings of their generalized Kummers. We plan to return to this construction in the near future.

\hide{
We expect an analogue of $\Theta_n$ to exist if we replace $\G^{[1]}$ by the analogue for abelian surfaces and replace $\G^{[n]}$ by the analogue for $(2n+4)$-dimensional smooth and projective moduli spaces $\M$ of stable sheaves on an abelian surface $X$ with a primitive Mukai vector. Denote by $\hat{X}$ the dual abelian surface. 
In contrast to the Hilbert scheme $X^{[n+1]}$, the product $\M:=X^{[n+1]}\times \hat{X}$, as well as more general such moduli spaces, admit a natural deformation over the moduli space of (marked) IHSMs of generalized Kummer type \cite[Sec. 12]{markman-generalized-kummers}. Parallel-transport operators of abelian surfaces lift to such for the moduli spaces $\M$ (\cite[Theorem 9.4]{markman-generalized-kummers}).


Set $A:=X\times \hat{X}$. 
A finite \'{e}tale cover of $\M$ is a product of $A \times Y$, where $Y$ is an IHSM of generalized Kummer type
and the pullback homomorphism $H^2(\M,\QQ)\rightarrow H^2(A \times Y,\QQ)$ is an isomorphism.
Hence, the LLV Lie algebra $\LieAlg{g}_\M$ of $\M$ is the direct sum $\LieAlg{g}_A\oplus \LieAlg{g}_Y$ of those of $A$ and $Y$, and the LLV Lie algebra of $X$ admits a natural diagonal embedding into that of $\M$. 
The rational LLV lattice of $\M$ is the direct sum 
$\widetilde{H}(\M,\QQ)=[H^1(A,\QQ)\oplus H^1(A,\QQ)^*]\oplus \widetilde{H}(Y,\QQ)$, where the second direct sum is orthogonal, and each of the two orthogonal direct summands is an irreducible representation of the corresponding direct summand of $\LieAlg{g}_\M$ \cite{looijenga-lunts}. Given a Fourier-Mukai kernel $E$ of non-zero rank over the product of two such moduli spaces $\M_1\times\M_2$, we expect $\phi:=[\kappa(E)\sqrt{td_{\M_1\times \M_2}}]_*$ to again induce a degree reversing Hodge isometry 
$\widetilde{H}(\phi):\widetilde{H}(Y_1,\QQ)\rightarrow \widetilde{H}(Y_2,\QQ)$, hence a Hodge isometry $\widetilde{H}_0(\phi):
H^2(Y_1,\QQ)\rightarrow H^2(Y_2,\QQ)$. This would provide the desired analogue of the functor $\widetilde{H}_0$ given in
(\ref{eq-functor-widetilde-H_0}).

The challenge is to find sufficiently many examples of $2$-dimensional moduli spaces $M$ over an abelian surface $X$, such that the BKR conjugate $\U^{[n]}$ over $X^{[n+1]}\times M^{[n+1]}$
of the equivariant cartesian power of the universal sheaf can be extended to a locally free Fourier-Mukai kernel over
$[X^{[n+1]}\times\hat{X}]\times [M^{[n+1]}\times \hat{M}]$, which deforms over the moduli space of pairs of generalized Kummers with a rational isometry between their second cohomologies (as in Section \ref{sec-hyperholomorphic-vector-bundles-deforming-a-FM-kernel}). We hope that this is the case for $\U^{[n]}\boxtimes \hat{\U}$, whenever the Fourier-Mukai kernel $\hat{\U}$ of the composition $D^b(\hat{M})\rightarrow D^b(M)\RightArrowOf{\Phi_\U} D^b(X)\rightarrow D^b(\hat{X})$ is represented by a locally free sheaf, where the left and right equivalences have kernel the Poincar\'{e} line bundle and its dual, respectively. 
The correspondence $[\kappa(\U^{[n]}\boxtimes \hat{\U})\sqrt{td_{X^{[n+1]}\times\hat{X}\times M^{[n+1]}\times \hat{M}}}]_*$
would then be the image of $[\kappa(\U)\sqrt{td_{M\times X}}]_*$ via the sought after analogue of the functor $\Theta_n$.
We hope to return to the generalized Kummer type in the near future.

}
%
\section{Generators for the rational isometry group}
\label{sec-rational-isometries}

Let $U$ be the rank $2$ lattice with basis $\{e_1,e_2\}$ satisfying $(e_i,e_i)=0$, $i=1,2$, and $(e_1,e_2)=-1$.
Let $L$ be an even unimodular lattice of signature $(3,s_-)$, $s_-\geq 3$, which contains a sublattice isometric to the orthogonal direct sum of three copies of $U$.
Let the lattice $\Lambda$ be the orthogonal direct sum $L\oplus \Integers\delta$, where $(\delta,\delta)=-2d$ for some positive integer $d$. Given a non-degenerate lattice $M$, denote by $O^+(M)$ the  subgroup of isometries of $M$ which preserve the orientation of the positive cone of $M\otimes_\ZZ\RR$ \cite[Sec. 4]{markman-survey}.
Let $\Gamma\subset O^+(\Lambda)$ be the subgroup 
which acts  on the discriminant group $\Lambda^*/\Lambda$ by multiplication by $\pm 1$. 
Set $\Lambda_\QQ:=\Lambda\otimes_\Integers\QQ$ and $L_\QQ:=L\otimes_\Integers\QQ$.
We regard $O(L_\QQ)$ as a subgroup of $O(\Lambda_\QQ)$ by extending each isometry so that it leaves $\delta$ invariant.

\begin{prop}
\label{prop-generators-of-the-isometry-group}
$O(\Lambda_\QQ)$ is generated by $\Gamma$ and $O(L_\QQ)$.
\end{prop}

The following Lemmas will be needed for the proof of the above Proposition.

\begin{lem}
\label{lemma-orbits-of-primitive-classes-intersect-K3-lattice}
Let $\alpha=\lambda+k\delta$, where $\lambda$ is a primitive class in $L$ and $k\in\ZZ$. There exists an isometry $g\in \Gamma$, such that $g(\alpha)$ belongs to $L$. 
\end{lem}

\begin{proof}
Let $\widetilde{\Lambda}$ be the orthogonal direct sum of $L$ and $U$ and let $\iota:\Lambda\rightarrow \widetilde{\Lambda}$
be the isometric embedding restricting to the identity on $L$ and sending $\delta$ to the element of the second direct summand $U$ with coordinate $(1,d)$
in the basis $\{e_1,e_2\}$. Let $v\in \widetilde{\Lambda}$ be the element with coordinates  $(1,-d)$ of the second direct summand $U$,
so that $(v,v)=2d$ and $v^\perp=\iota(\Lambda)$, where $v^\perp$ is the co-rank $1$ sublattice of $\widetilde{\Lambda}$ orthogonal to $v$. 
The sublattice $\Sigma_1:=\span_\ZZ\{\iota(\alpha),v\}$ of $\widetilde{\Lambda}$ is saturated. Indeed, if $a,b\in\QQ$ and $a\iota(\alpha)+bv=a\lambda+(ak+b)e_1+(adk-bd)e_2$ is an integral class of $\widetilde{\Lambda}$, then $a$  must be an integer, since $\lambda$ is primitive, hence $b$ must be an integer as well.
Let $\beta\in L$ be a primitive class satisfying $(\beta,\beta)=(\alpha,\alpha)$. The sublattice $\Sigma_2:=\span_\ZZ\{\beta,v\}$ is saturated in $\widetilde{\Lambda}$ and 
is isometric to $\Sigma_1$. There exists an isometry $g\in O(\widetilde{\Lambda})$ satisfying 
$g(v)=v$ and $g(\iota(\alpha))=\beta$, by results of Nikulin \cite[Lemma 8.1]{nikulin,markman-monodromy}. 
Set $\tilde{\gamma}:=g$, if  $g$ belongs to $O^+(\widetilde{\Lambda}),$ and $\tilde{\gamma}:=-g$ otherwise. Then 
$\tilde{\gamma}\circ\iota=\iota\circ\gamma$, for a unique element 
$\gamma\in\Gamma$, by \cite[Theorem 1.6 and Lemma 4.2]{markman-monodromy}, and $\gamma(\alpha)$ belongs to $L$.
\end{proof}

Let $G$ be the subgroup of $O(\Lambda_\QQ)$ generated by $\Gamma$ and $O(L_\QQ)$.

\begin{lem}
\label{lemma-G-orbit-intersects-rational-K3-lattice}
Let $\lambda\in L_\QQ$ and $t\in\QQ$ and assume that $(\lambda,\lambda)\neq 0$. There exists an isometry $g\in G$,
such that $g(\lambda+t\delta)$ belongs to $L_\QQ$.
\end{lem}

\begin{proof}
Let $q\in\QQ^\times$ be such that $q(\lambda+t\delta)$ is integral. Choose an integral and primitive class $\lambda'\in L$
satisfying $(\lambda',\lambda')=q^2(\lambda,\lambda)$.
There exists an isometry $f\in O(L_\QQ)$, such that $\lambda'=f(q\lambda)$, by \cite[Prop. 2.35]{gerstein}.
Then $f(q(\lambda+t\delta))=\lambda'+qt\delta$ is integral and there exists $h\in\Gamma$, such that $h(f(q(\lambda+t\delta)))$ is a primitive class 
in $L$, 
by Lemma \ref{lemma-orbits-of-primitive-classes-intersect-K3-lattice}. The statement follows, since $h\circ f$ belongs to $G$.
\end{proof}

\begin{proof}[Proof of Proposition \ref{prop-generators-of-the-isometry-group}]
Choose $u\in L$ satisfying $(u,u)=2d+2$. Then $(u+\delta,u+\delta)=2$.
Let $\rho_{u+\delta}$ be the reflection of $\Lambda$ in $(u+\delta)^\perp$
\[
\rho_{u+\delta}(x)=x-(u+\delta,x)(u+\delta).
\]
Then $\rho_{u+\delta}(\delta)=(2d)u+(1+2d)\delta$. Furthermore, $-\rho_{u+\delta}$ belongs to $\Gamma$.
There exists $g_1\in G$, such that $g_1(-\rho_{u+\delta}(\delta))$ belongs to $L_\QQ$, by Lemma \ref{lemma-G-orbit-intersects-rational-K3-lattice}.

Let $\phi$ be an isometry in $O(\Lambda_\QQ)$. Write $\phi(\delta)=\lambda+t\delta$. If $(\lambda,\lambda)=0$
choose $u\in L$, such that $(u,\lambda)=0$ and $(u,u)=2d+2$. Then $\rho_{u+\delta}(\phi(\delta))=\lambda+(2d)u+(1+2d)\delta$
and the self-intersection $(\lambda+(2d)u,\lambda+(2d)u)=4d^2(2d+2)$ is non-zero.
Now, $\phi$ belongs to $G$, if and only if $-\rho_{u+\delta}\phi$ belongs to $G$. Hence, we may assume that $(\lambda,\lambda)\neq 0$. Then there exists $g_2\in G$ 
such that $(g_2\circ\phi)(\delta)$ belongs to $L_\QQ$, by Lemma \ref{lemma-G-orbit-intersects-rational-K3-lattice}.
There exists $h\in O(L_\QQ)$, such that $h(g_2(\phi(\delta)))=g_1(-\rho_{u+\delta}(\delta))$,
by \cite[Prop. 2.35]{gerstein}. Hence,
$(-\rho_{u+\delta}g_1^{-1}hg_2\phi)(\delta)=\delta$, and so $-\rho_{u+\delta}g_1^{-1}hg_2\phi$ belongs to $O(L_\QQ)$.
Hence, $\phi$ belongs to $G$.
\end{proof}

\begin{lem}
\label{lemma-O-L-Q-is-generated-by-positive-reflections}
$O(L_\QQ)$ is generated by $O(L)$ and reflections $\rho_u$ in $u^\perp$, for $u\in L$ satisfying $(u,u)>0$.
\end{lem}

\begin{proof}
$O(L_\QQ)$ is generated by reflections $\rho_u$, with $(u,u)\neq 0$, by \cite[Prop. 2.36]{gerstein}. The isometry group $O(L)$ acts transitively on the set 
$\{u \ : (u,u)=2d\}$, $d\in\ZZ$, by \cite[Th. 1.14.4]{nikulin}. Consider the above basis $\{e_1,e_2\}$ of $U$ and observe that
$\rho_{e_1-de_2}\rho_{e_1+de_2}$ is minus the identity of $U$. Hence, for every element $u\in L$ with $(u,u)<0$, there exists an element $w\in L$ with $(w,w)=-(u,u)$, such that $\rho_u\rho_w$ belongs to $O(L)$. Thus, it sufficed to consider reflections $\rho_u$ with $(u,u)>0$.
\end{proof}

\begin{cor}
\label{cor-decomposition-of-a-rational-isometry}
Let $\phi$ be an isometry in $O^+(\Lambda_\QQ)$.
There exist a positive integer $k$, integral elements $u_i\in L$ satisfying $(u_i,u_i)\geq 2$, $1\leq i \leq k$, and elements
$\gamma_i\in \Gamma$, $0\leq i \leq k$, such that
\[
\phi=(-1)^k\gamma_k\rho_{u_k}\gamma_{k-1}\rho_{u_{k-1}}\cdots\gamma_1\rho_{u_1}\gamma_0,
\]
where $\rho_{u_i}(x)=x-\frac{2(u_i,x)}{(u_i,u_i)}u_i$ is the reflection of $\Lambda_\QQ$ in $u_i^\perp$.
\end{cor}

\begin{proof}
Follows immediately from Proposition \ref{prop-generators-of-the-isometry-group}, Lemma \ref{lemma-O-L-Q-is-generated-by-positive-reflections}, the fact that that $O^+(L)$ is contained in $\Gamma$, and the fact that 
$-\rho_{u_i}$ belongs to $O^+(\Lambda_\QQ)$ if $(u_i,u_i)>0$, since the signature of $\Lambda$ is $(3,1+s_-)$ and $3$ is odd.
\end{proof}

\begin{rem}
\label{rem-generators-for-the-rational-isometry-group-in-the-generalized-kummer-case}
Note that in Corollary \ref{cor-decomposition-of-a-rational-isometry} we can require $\gamma_i$ to belong to any subgroup $\Gamma_0$ of $\Gamma$, such that $\Gamma_0$ and $O^+(L)$ generate $\Gamma$. When $L$ is the $K3$ lattice, the group $\Gamma$ is the monodromy group of IHSMs of $K3^{[d+1]}$ deformation type \cite[Theorem 1.6]{markman-monodromy}. If $L=U^{\oplus 3}$, the monodromy of generalized Kummer varieties of dimension $2d-2$, for $d\geq 3$, is the following index $2$ subgroup $\Gamma_0$ of $\Gamma$.
$\Gamma$ acts of the discriminant froup $\Lambda^*/\Lambda$ by a character $\xi:\Gamma\rightarrow \{\pm 1\}$, by definition of $\Gamma$. Let $\det:\Gamma \rightarrow \{\pm 1\}$ be the determinant character. Then $\Gamma_0$ is the kernel of the product character $\det\cdot\xi:\Gamma\rightarrow \{\pm 1\}$, by \cite[Theorem 1.4]{markman-generalized-kummers}. The character $\det\cdot\xi$ has value $-1$ on a reflection 
$\rho_\omega\in O^+(L)$ in a class $w \in L$ with $(w,w)=-2$, and so $\Gamma_0$ and $O^+(L)$ generate $\Gamma$.
\end{rem}

%
\section{The LLV Lie algebra}
\label{section-LLV}
In Section \ref{sec-Verbitsky-component} we recall the LLV Lie algebra $\LieAlg{g}_X$ and its action on the cohomology of a $2n$-dimensional IHSM X. We review 
the $\LieAlg{g}_X$-equivariant isomorphism between the subring $SH^*(X,\QQ)$ of $H^*(X,\QQ)$ generated by $H^2(X,\QQ)$ and the $n$-th symmetric power of the LLV lattice $\widetilde{H}(X,\QQ)$.
In Section \ref{sec-functor-widetilde-H} we recall Taelman's definition of the functor $\widetilde{H}:\G\rightarrow \widetilde{\G}$,
metioned in (\ref{eq-functor-widetilde-H}), sending equivalences of derived categories to isometries of LLV lattices. Given a morphism $\phi\in \Hom_\G(X,Y)$, which is degree preserving up to sign, 
we relate in Lemmas \ref{lemma-widetilde-H-of-a-degree-preserving-morphism} and \ref{lemma-widetilde-H-of-a-degree-reversing-morphism}
the restrictions of $\phi$ and of $\widetilde{H}(\phi)$ to $H^2(X,\QQ)$.

%
\subsection{The subring generated by $H^2(X,\QQ)$}
\label{sec-Verbitsky-component}
Let $X$ be a $2n$-dimensional irreducible holomorphic symplectic manifold. 
Given a class $\lambda\in H^2(X,\QQ)$ denote by $e_\lambda:H^*(X,\QQ)\rightarrow H^*(X,\QQ)$ the endomorphism given by cup product with $\lambda$. 
The grading operator 
\begin{equation}
\label{eq-grading-operator}
h_X:H^*(X,\QQ)\rightarrow H^*(X,\QQ).
\end{equation} 
acts on $H^k(X,\QQ)[2n]=H^{2n+k}(X,\QQ)$, $-2n\leq k\leq 2n$, via multiplication by $k$.
A triple $(e,h,f)$ of endomorphisms of $H^*(X,\QQ)$ is called an {\em $\LieAlg{sl}_2$-triple}, if
\[
[e,f]=h, \ [h,e]=2e, \ [h,f]=-2f.
\]
Given $e$, if $f$ exists, then it is the unique endomorphism completing the pair $(e,h)$ to an $\LieAlg{sl}_2$-triple.
The {\em Looijenga-Lunts-Verbitsky (LLV) Lie algebra} $\LieAlg{g}_X\subset \End(H^*(X,\QQ))$ is the Lie algebra generated by all $\LieAlg{sl}_2$-triples
$(e_\lambda,h,f_\lambda)$, $\lambda\in H^2(X,\QQ)$. 

Denote by $S_{[n]}\widetilde{H}(X,\QQ)$ the subspace of $\Sym^n(\widetilde{H}(X,\QQ))$ generated by $n$-th powers of isotropic elements of $\widetilde{H}(X,\QQ)$. Verbitsky proved that the subring $SH^*(X,\QQ)$ is an irreducible  $\LieAlg{g}_X$-submodule of 
$H^*(X,\QQ)$, appearing with multiplicity one,
there exists a Lie algebra isomorphism
\begin{equation}
\label{eq-dot-rho}
\dot{\rho}:\LieAlg{so}(\widetilde{H}(X,\QQ))\rightarrow \LieAlg{g}_X,
\end{equation}
and there exists a unique graded isomorphism of $\LieAlg{so}(\widetilde{H}(X,\QQ))$-modules
\begin{equation}
\label{eq-taelmans-Psi}
\Psi:SH^*(X,\QQ)[2n]\rightarrow S_{[n]}\widetilde{H}(X,\QQ),
\end{equation}
satisfying $\Psi(1)=\frac{\alpha^n}{n!}$, where the $\LieAlg{so}(\widetilde{H}(X,\QQ))$-module structure of $SH^*(X,\QQ)[2n]$ is via $\dot{\rho}$  \cite[Prop. 3.5]{taelman}. 
The pairing on $\widetilde{H}(X,\QQ)$ induces a pairing $b_{[n]}$ on $S_{[n]}\widetilde{H}(X,\QQ)$, the Mukai pairing on $H^*(X,\QQ)$ restricts to a pairing $b_{SH}$ on $SH^*(X,\QQ)$, and 
 $c_X b_{[n]}(\Psi(x),\Psi(y))=b_{SH}(x,y),$ where $c_X$ is the Fujiki constant \cite[Prop. 3.8]{taelman}.
 
 Following is the explicit description of the isomorphism $\Psi$ (see \cite[Prop. 3.5]{taelman}). 
Given $\lambda\in H^2(X,\QQ)$, denote by 
\begin{equation}
\label{eq-e-lambda}
e_\lambda\in \LieAlg{so}(\widetilde{H}(X,\QQ))
\end{equation} 
also the endomorphism of $\widetilde{H}(X,\QQ)$ sending $\alpha$ to $\lambda$, annihilating $\beta$, and sending $\lambda'\in H^2(X,\QQ)$ to $(\lambda,\lambda')\beta$. Extend $e_\lambda$ as an endomorphism of $\Sym^n\widetilde{H}(X,\QQ)$ via the product rule, so that 
$
e_\lambda(x_1\cdots x_n)=\sum_{i=1}^n x_1\cdots e_\lambda(x_i) \cdots x_n.
$
Then 
\[
\Psi(\lambda_1\cdots \lambda_k)=e_{\lambda_1}\cdots e_{\lambda_k}(\alpha^n/n!).
\]

Let
\begin{equation}
\label{eq-grading-operator-of-LLV-lattice}
\tilde{h}_X:\widetilde{H}(X,\QQ)\rightarrow \widetilde{H}(X,\QQ).
\end{equation} 
be the grading operator multiplying $\alpha$ by $-2$, $\beta$ by $2$, and $H^2(X,\QQ)$ by $0$.

%
\subsection{The functor $\widetilde{H}$}
\label{sec-functor-widetilde-H}
Let $\phi$ be a morphism in $\Hom_\G((X,\epsilon),(Y,\epsilon'))$, where $\G$ is the groupoid (\ref{eq-Groupoid-G}). Then $\phi$ is in particular an isometry $\phi:H^*(X,\QQ)\rightarrow H^*(Y,\QQ)$ with respect to the Mukai pairing and we denote by 
\[
Ad_\phi:\End(H^*(X,\QQ))\rightarrow \End(H^*(Y,\QQ))
\]
the Lie algebra isomorphism given by $Ad_\phi(\xi)=\phi\xi\phi^{-1}$. Then $Ad_\phi$ restricts to a Lie algebra isomorphism 
from $\LieAlg{g}_X$ onto $\LieAlg{g}_Y$, by \cite[Theorem A]{taelman}. Hence, $\phi$ restricts to an isometry 
$\restricted{\phi}{}:SH^*(X,\QQ)\rightarrow SH^*(Y,\QQ)$. Consequently, when $X$ and $Y$ are deformation equivalent, so that $c_X=c_Y$, then 
$\Psi_Y\circ \restricted{\phi}{}\circ\Psi_X^{-1}:S_{[n]}\widetilde{H}(X,\QQ)\rightarrow S_{[n]}\widetilde{H}(Y,\QQ)$
is an isometry, which conjugates  $S_{[n]}\LieAlg{so}(\widetilde{H}(X,\QQ))$ to $S_{[n]}\LieAlg{so}(\widetilde{H}(Y,\QQ))$,
where $S_{[n]}\LieAlg{so}(\widetilde{H}(X,\QQ))$ is the image of $\LieAlg{so}(\widetilde{H}(X,\QQ))$ in $\End(S_{[n]}\widetilde{H}(X,\QQ))$. 

Given an isomorphism $f:V_1\rightarrow V_2$ of oriented vector spaces, we set $\det(f)=1$, if $f$ is orientation preserving and $\det(f)=-1$ otherwise.

\begin{prop} 
\label{prop-needed-to-define-functor-tilde-H}
\cite[Prop. 4.4, Theorem 4.10, and Theorem 4.11]{taelman}
Let $V_1$ and $V_2$ be $d$-dimensional vector spaces  over $\QQ$ with non-degenerate bilinear parings and $\phi:S_{[n]}V_1\rightarrow S_{[n]}V_2$ an isometry, such that $\phi (S_{[n]}\LieAlg{so}(V_1))\phi^{-1}=S_{[n]}\LieAlg{so}(V_2)$. If $n$ is even assume that $d$ is odd and choose an orientation for each of $V_1$ and $V_2$. Then there exists a unique isometry 
$\widetilde{H}(\phi):V_1\rightarrow V_2$, such that the restriction $S_{[n]}(\widetilde{H}(\phi))$ of its $n$-symmetric power to $S_{[n]}V_1$ satisfies
$S_{[n]}(\widetilde{H}(\phi))=\phi$, if $n$ is odd, and
$\det(\widetilde{H}(\phi))S_{[n]}(\widetilde{H}(\phi))=\phi$, if $n$ is even. 
\end{prop}

\begin{rem}
\label{remark-tilde-H-of-mixed-composition}
Fix a positive integer $n$.
Consider the two groupoids $\V$ and $S_{[n]}\V$, whose objects  are vector spaces over $\QQ$, which are odd dimensional if $n$ is even, with non-degenerate symmetric quadratic forms and with the additional choice of orientations, if $n$ is even. The morphisms of $\V$ are isometries. The morphisms in $\Hom_{S_{[n]}\V}(V_1,V_2)$ are isometries $\phi:S_{[n]}V_1\rightarrow S_{[n]}V_2$ satisfying $\phi (S_{[n]}\LieAlg{so}(V_1))\phi^{-1}=S_{[n]}\LieAlg{so}(V_2)$. We have the obvious functor $S_{[n]} : \V \rightarrow S_{[n]}\V$, which is the identity on objects, and which sends an isometry from $V_1$ to $V_2$ to the restriction of its $n$-th symmetric power to the subrepresentation $S_{[n]}V_1$.
Proposition \ref{prop-needed-to-define-functor-tilde-H} gives rise to a functor $\widetilde{H}:S_{[n]}\V\rightarrow \V$. If $n$ is odd, the functors $S_{[n]}$ and $\widetilde{H}$ are inverses of each other. If $n$ is even, then $S_{[n]}$ is not faithful, since $S_{[n]}(\phi)=S_{[n]}(-\phi)$.
Proposition \ref{prop-needed-to-define-functor-tilde-H} asserts that if $n$ is even, then the functor 
$F:\V\rightarrow S_{[n]}\V$, which is the identity on objects, and which sends $f\in \Hom_\V((V_1,\epsilon_1),(V_2,\epsilon_2))$ to $\det(f)S_{[n]}(f)$, is an equivalence and defines $\widetilde{H}$ as its inverse. 
We have $S_{[n]}(-f)=-S_{[n]}(f)$, if $n$ is odd, and $F(-f)=-F(f)$, if $n$ is even, and so $\widetilde{H}(-\phi)=-\widetilde{H}(\phi)$.
Let $f_i\in O(V_i)$, $i=1,2$. We have
\begin{equation}
\label{eq-composing-widetilde-H-with-symmetric-powers}
\widetilde{H}(S_{[n]}(f_2)\circ \phi\circ S_{[n]}(f_1))=
\left\{
\begin{array}{ccc}
f_2\circ \widetilde{H}(\phi)\circ f_1 & \mbox{if} & n \ \mbox{is odd},
\\
\det(f_1)\det(f_2)f_2\circ \widetilde{H}(\phi)\circ f_1 & \mbox{if} & n \ \mbox{is even}.
\end{array}
\right.
\end{equation}
\end{rem}

\begin{defi}
\label{def-functor-tilde-H}
The functor $\widetilde{H}:\G\rightarrow \widetilde{\G}$, given in (\ref{eq-functor-widetilde-H}), sends a morphism $\phi$
in $\Hom_\G((X,\epsilon_X),(Y,\epsilon_Y))$ to $\widetilde{H}(\Psi_Y\circ \restricted{\phi}{}\circ\Psi_X^{-1}):\widetilde{H}(X,\QQ)\rightarrow \widetilde{H}(Y,\QQ)$, where the latter functor $\widetilde{H}$ is from Remark \ref{remark-tilde-H-of-mixed-composition} and 
$\restricted{\phi}{}$ is the restriction of $\phi$ to $SH^*(X,\QQ)$.
\end{defi}

\begin{lem}
\label{lemma-preservation-of-grading-of-LLV-lattice-implies-that-of-cohomology}
Let $X$ and $Y$ be $2n$-dimensional irreducible holomorphic symplectic manifolds.
Let $\phi$ be a morphism in $\Hom_{\G}((X,\epsilon_X),(Y,\epsilon_Y))$. Set $\tilde{\phi}:=\widetilde{H}(\phi):\widetilde{H}(X,\QQ)\rightarrow \widetilde{H}(Y,\QQ)$. The equality $\tilde{\phi}\tilde{h}_X=(-1)^k\tilde{h}_Y\tilde{\phi}$, for some $k\in\{0,1\}$, is equivalent to  $\phi h_X=(-1)^kh_Y\phi$ for the same $k$.
\end{lem}

\begin{proof} The proof is modeled  after that of \cite[Theorem 5.3]{taelman}.
The isomorphism $\phi:H^*(X,\QQ)\rightarrow H^*(Y,\QQ)$ conjugates the LLV Lie algebra $\LieAlg{g}_X$ to  $\LieAlg{g}_Y$,
by \cite[Theorem A]{taelman}, and it conjugates $SH^*(X,\QQ)$ to $SH^*(Y,\QQ)$, by \cite[Theorem B]{taelman}.
Hence, $\phi h_X\phi^{-1}$ belongs to $\LieAlg{g}_Y$. 
Theorems 4.10 and 4.11 in \cite{taelman} yield the commutative diagram
\begin{equation}
\label{eq-commutative-diagram-of-Lie-algebras}
\xymatrix{
\LieAlg{so}(\widetilde{H}(X,\QQ)) \ar[r]^{ad(\tilde{\phi})} \ar[d] &
\LieAlg{so}(\widetilde{H}(Y,\QQ)) \ar[d]
\\
\LieAlg{sl}(SH^*(X,\QQ)) & \LieAlg{sl}(SH^*(Y,\QQ))
\\
\LieAlg{g}_X \ar[u] \ar@/^6pc/[uu]^{\cong}\ar[r]_{ad(\phi)}
& 
\LieAlg{g}_Y \ar[u] \ar@/_6pc/[uu]_{\cong}
}
\end{equation}
where the left bottom vertical arrow is the restriction to the subrepresentation $SH^*(X,\QQ)$ of $H^*(X,\QQ)$ and the right bottom vertical arrow is the analogue. The top left vertical arrow corresponds to 
$f\mapsto \Psi_X^{-1}\circ (S_{[n]}f) \circ\Psi_X$ and the right is the analogue.
All four vertical arrows are injective and each pair in the same column have the same image yielding the
curved arrows which are isomorphisms (inverses of $\dot{\rho}$ in (\ref{eq-dot-rho})).
The grading operator $h_X\in\LieAlg{g}_X$ restricts to the same element of $\LieAlg{sl}(SH^*(X,\QQ))$ as the image of the grading operator $\tilde{h}_X\in \LieAlg{so}(\widetilde{H}(X,\QQ))$, since $\Psi_X$ is a graded isomorphism. The same holds for $h_Y$ and $\tilde{h}_Y$.
Hence, the equality $\tilde{\phi}\tilde{h}_X\tilde{\phi}^{-1}=(-1)^k\tilde{h}_Y$ 
is equivalent to the equality of the restrictions of $\phi h_X\phi^{-1}$ and $(-1)^k h_Y$  to elements of $\LieAlg{sl}(SH^*(Y,\QQ))$, by the commutativity of the diagram and the injectivity of the top right vertical arrow. The latter is equivalent to the equality $\phi h_X\phi^{-1}=(-1)^kh_Y$, since the restriction homomorphism
$\LieAlg{g}_Y\rightarrow \LieAlg{sl}(SH^*(Y,\QQ))$
is injective.
\end{proof}

The relation between $h_X$ and $\tilde{h}_X$ explains the following result of 
Verbitsky and Looijenga-Lunts.

\begin{lem}
\label{lemma-commutator-of-degree-operatoor}
\cite[Prop. 4.5(ii)]{looijenga-lunts}
The commutator of $h$ in $\LieAlg{g}_X$ decomposes as a direct sum 
$\bar{\LieAlg{g}}_X\oplus \QQ h_X$, where $\bar{\LieAlg{g}}_X$ is the derived Lie subalgebra of the commutator, and
$\bar{\LieAlg{g}}_X$ 
 is isomorphic to $\LieAlg{so}(H^2(X,\QQ)).$
\end{lem}

The {\em Hodge operator}
\begin{equation}
\label{eq-Hodge-operator}
h'_X\in \End(H^*(X,\CC))
\end{equation}
acts on $H^{p,q}(X)$ by multiplication by $q-p$. Verbitsky proved that it is  an element of the subalgebra $\bar{\LieAlg{g}}_{X,\CC}$ of $\LieAlg{g}_{X,\CC}$
(\cite[Theorem 8.1]{verbitsky-mirror-symmetry}, see also \cite[Cor. 2.12]{taelman}).

\begin{lem}
\label{lemma-widetilde-H-of-a-degree-preserving-morphism}
Assume that $\phi\in \Hom_\G((X,\epsilon_X),(Y,\epsilon_Y))$ is degree preserving and let $\phi_2$ be the restriction of $\phi$ to $H^2(X,\QQ)$. Let $t\in\QQ^\times$ be such that $\widetilde{H}(\phi)(\alpha)=t^{-1}\alpha$. Then the restriction $\widetilde{H}(\phi)_0$ of $\widetilde{H}(\phi)$ to $H^2(X,\QQ)$ satisfies
\[
\widetilde{H}(\phi)_0=\left\{\begin{array}{ccc}
t^{n-1}\phi_2, & \mbox{if} & n \ \mbox{is odd},
\\
\det(\widetilde{H}(\phi))t^{n-1}\phi_2,  & \mbox{if} & n \ \mbox{is even}.
\end{array}
\right.
\]
\end{lem}

\begin{proof}
Assume that $\phi\in \Hom_\G((X,\epsilon_X),(Y,\epsilon_Y))$ is degree-preserving. Then so is $\widetilde{H}(\phi)$, by Lemma \ref{lemma-preservation-of-grading-of-LLV-lattice-implies-that-of-cohomology}.
Let $\lambda$ be a class in $H^2(X,\QQ)$.
We have
\begin{eqnarray*}
(S_{[n]}\widetilde{H}(\phi))(\Psi(\lambda)) & = &
(S_{[n]}\widetilde{H}(\phi))(e_\lambda(\alpha^n/n!))=(S_{[n]}\widetilde{H}(\phi))(\alpha^{n-1}\lambda/(n-1)!)
\\
&=&
\frac{t^{1-n}}{(n-1)!}\alpha^{n-1}\widetilde{H}(\phi)(\lambda),
\\
(S_{[n]}\widetilde{H}(\phi))(\Psi(\lambda)) & = & a \Psi(\phi(\lambda))=
a e_{\phi(\lambda)}(\alpha^n/n!)=a \alpha^{n-1}\phi(\lambda)/(n-1)!,
\end{eqnarray*}
where $a=1$, if $n$ is odd, and $a=\det(\widetilde{H}(\phi))$, if $n$ is even.
The statement follows from the equality of the right hand sides in the two equations above (Proposition \ref{prop-needed-to-define-functor-tilde-H}).
\end{proof}

\begin{lem}
\label{lemma-widetilde-H-of-a-degree-reversing-morphism}
Let $\phi\in\Hom_\G((X,\epsilon_X),(Y,\epsilon_Y))$ be a degree reversing morphism, where $X$ and $Y$ are $2n$-dimensional. Then the value $\widetilde{H}_0(\phi):H^2(X,\QQ)\rightarrow H^2(Y,\QQ)$ at $\phi$ of the functor $\widetilde{H}_0$ given in (\ref{eq-functor-widetilde-H_0}) is a scalar multiple of the composition
\begin{equation}
\label{eq-composition-with-power-of-c-2}
H^2(X,\QQ)\LongRightArrowOf{\cup c_2(X)^{n-1}} H^{4n-2}(X,\QQ) \LongRightArrowOf{\phi} H^2(Y,\QQ).
\end{equation}
\end{lem}

\begin{proof}
Note first that $H^2(X,\QQ)$ and $H^{4n-2}(X,\QQ)$ are subspaces of $SH^*(X,\QQ)$. Let $\restricted{\phi}{}$ be the restriction of $\phi$ to $SH^*(X,\QQ)$.
Given an isometry $f\in O(\widetilde{H}(X,\QQ))$, let 
\[
\eta_f : SH^*(X,\QQ)  \rightarrow  SH^*(X,\QQ) 
\]
be given by $\eta_f:=\Psi_X^{-1}\circ S_{[n]}(f)\circ \Psi_X$.
Then 
\[
\widetilde{H}(\restricted{\phi}{}\circ \eta_f)=
\left\{\begin{array}{ccc}
\widetilde{H}(\phi)\circ f, & \mbox{if} & n \ \mbox{is odd},
\\
\det(f)\widetilde{H}(\phi)\circ f,  & \mbox{if} & n \ \mbox{is even},
\end{array}
\right.
\]
by (\ref{eq-composing-widetilde-H-with-symmetric-powers}). 
The left hand side above is defined to be $\widetilde{H}(\Psi_X\circ(\restricted{\phi}{}\circ \eta_f)\circ\Psi_X^{-1})$
where $\widetilde{H}$ is the functor in Proposition \ref{prop-needed-to-define-functor-tilde-H}.

Let $\tau\in O(\widetilde{H}(X,\QQ))$ be the element which interchanges $\alpha$ and $\beta$ and multiplies $H^2(X,\QQ)$ by $-1$. Then $\tau\tilde{h}_X=-\tilde{h}_X\tau$ and so $\eta_\tau$ is degree reversing. Hence, $\restricted{\phi}{}\circ \eta_\tau$
is degree preserving and so the restriction $\widetilde{H}(\restricted{\phi}{}\circ \eta_\tau)_0$ of
$\widetilde{H}(\restricted{\phi}{}\circ \eta_\tau)$ to $H^2(X,\QQ)$ is a scalar multiple of $(\restricted{\phi}{}\circ \eta_\tau)_2$, by
Lemma \ref{lemma-widetilde-H-of-a-degree-preserving-morphism}. Furthermore, 
$\widetilde{H}(\restricted{\phi}{}\circ \eta_\tau)_0=\widetilde{H}(\phi)_0\circ \widetilde{H}(\eta_\tau)_0$ and
$\widetilde{H}(\eta_\tau)_0=-id,$ so $\widetilde{H}(\restricted{\phi}{}\circ \eta_\tau)_0=- \widetilde{H}(\phi)_0=\pm\widetilde{H}_0(\phi).$
We conclude that $\widetilde{H}_0(\phi)$ is a scalar multiple of $(\restricted{\phi}{}\circ \eta_\tau)_2$.

It remains to prove that $\eta_\tau:H^2(X,\QQ)\rightarrow H^{4n-2}(X,\QQ)$ is a scalar multiple of cup product with $c_2(X)^{n-1}$. Let $G$ be the subgroup of $SO(\widetilde{H}(X,\QQ))$ leaving each of $\alpha$ and $\beta$ invariant.
Then $G$ commutes with $\tau$ and its $\eta$-action on $SH^*(X,\QQ)$ is via degree preserving isometries with respect to which  $\eta_\tau$ is $G$-equivariant. 
Now, 
$H^2(X,\QQ)$ and $H^{4n-2}(X,\QQ)$ are dual irreducible $G$-representations and $\cup c_2(X)^{n-1}:H^2(X,\QQ)\rightarrow H^{4n-2}(X,\QQ)$ is $G$-equivariant as well and non-vanishing, by \cite[Theorem 4.7]{looijenga-lunts}. Hence, a scalar multiple of the latter is equal to the restriction of $\eta_\tau$ to $H^2(X,\QQ)$.
%
\end{proof}

Let $X$ and $Y$ be projective IHSMs and $\P\in D^b(X\times Y)$ the Fourier-Mukai kernel of an equivalence $\Phi_\P:D^b(X)\rightarrow D^b(Y)$. Assume that the rank $r$ of $\P$ does not vanish. Let $x$ be a point of $X$, denote by $\P_x\in D^b(Y)$ the restriction of $\P$ to $\{x\}\times Y$, and 
set $\lambda:=c_1(\P_x)\in H^2(Y,\Integers)$. Note that the object $\P_x$ of $D^b(Y)$ is the image of the sky-scraper sheaf at $x$ via $\Phi_\P$. Set $\widetilde{H}(\Phi_\P):=\widetilde{H}([ch(\P)\sqrt{td_{X\times Y}}]_*)$.
The following was proved independently by Beckmann \cite[Lemma 4.13(iv)]{beckmann} and the author \cite[Theorem 6.13(3)]{markman-modular}.

\begin{prop}
\label{prop-LLV-line}
The isometry $\widetilde{H}(\Phi_\P):\widetilde{H}(X,\QQ)\rightarrow \widetilde{H}(Y,\QQ)$ maps $\beta$ to 
the isotropic line spanned by $r\alpha+\lambda+s\beta$, where $s=(\lambda,\lambda)/2r$ so that the line is isotropic.
\end{prop}

%
\section{The degree reversing Hodge isometry of a Fourier-Mukai kernel of non-zero rank}
\label{sec-the-degree-reversing-Hodge-isometry-of-a-FM-kernel}

Consider an object $E$ of non-zero rank in the derived category of a product $X\times Y$ of $2n$-dimensional deformation equivalent irreducible holomorphic symplectic manifolds $X$ and $Y$. Assume that $E$ is the Fourier-Mukai kernel of an equivalence of derived categories. This means that the functor $\Phi_E:D^b(X)\rightarrow D^b(Y)$, given by
$\Phi_E(\bullet)=R\pi_{Y,*}(L\pi_X^*(\bullet)\otimes E)$, is an equivalence of triangulated categories.

\begin{lem}
\label{lemma-phi-is-degree-reversing}
\begin{enumerate}
\item
\label{lemma-item-phi-is-degree-reversing}
The class $\kappa(E)\sqrt{td_{X\times Y}}$, with $\kappa(E)$ as in Definition \ref{def-kappa}, induces a degree reversing Hodge isometry
\[
\phi:H^*(X,\QQ)\rightarrow H^*(Y,\QQ).
\]
\item
\label{lemma-item-phi-restricts-to-a-Hodge-isometry-psi-E}
The Hodge isometry $\widetilde{H}(\phi):\widetilde{H}(X,\QQ)\rightarrow \widetilde{H}(Y,\QQ)$, associated to $\phi$ and any choice of orientations of $H^2(X,\QQ)$ and $H^2(Y,\QQ)$, restricts to a Hodge isometry 
\begin{equation}
\label{eq-psi-E}
\widetilde{H}(\phi)_0:H^2(X,\QQ)\rightarrow H^2(Y,\QQ).
\end{equation}
\item
\label{lemma-item-Ad-phi-resricts-to-Ad-psi-E}
The Lie algebras isomorphism $Ad_\phi:\LieAlg{g}_X\rightarrow \LieAlg{g}_Y$ is graded and restricts to 
\[
Ad_{\widetilde{H}_0(\phi)}:\bar{\LieAlg{g}}_X\rightarrow \bar{\LieAlg{g}}_Y.
\]
\end{enumerate}
\end{lem}

\begin{proof}
Let $r$ be the rank of $E$ and let $\lambda_X\in H^2(X,\ZZ)$ and $\lambda_Y\in H^2(Y,\ZZ)$ satisfy $c_1(E)=\pi_X^*(\lambda_X)+\pi_Y^*(\lambda_Y)$. 
Given $\lambda\in H^2(X,\QQ)$, let $e_\lambda$ be the endomorphism of $\widetilde{H}(X,\QQ)$ given in (\ref{eq-e-lambda}).
Set $B_\lambda:=\exp(e_\lambda)$. 

\hide{
Given a point $x\in X$ denote by $\StructureSheaf{x}$ the sky-scraper sheaf at $x$ and by $E_x$ the restriction of $E$ to the fiber of $\pi_X$ over $x$. Then $\Phi_E(\StructureSheaf{x})=E_x$. Hence, 
$\widetilde{H}(\Phi_E)$ maps $\span\{\beta\}=\ell(\StructureSheaf{x})$ to
$\ell(E_x)=\span\{r\alpha+\lambda_Y+\frac{(\lambda_Y,\lambda_Y)}{2r}\beta\}$, by
\cite[Theorem 6.13]{markman-modular}, where the coefficient of $\beta$ is determined since $\ell(E_x)$ is isotropic (see also \cite{beckmann}). 
}
The isometry $\widetilde{H}(\Phi_E)$ maps $\span\{\beta_X\}$ to $\span\{r\alpha_Y+\lambda_Y+\frac{(\lambda_Y,\lambda_Y)}{2r}\beta_Y\}$, by Proposition \ref{prop-LLV-line}.
We have $B_{-\lambda_X/r}(\beta_X)=\beta_X$ and 
\[
B_{-\lambda_Y/r}\left(r\alpha_Y+\lambda_Y+\frac{(\lambda_Y,\lambda_Y)}{2r}\beta_Y\right)=r\alpha_Y.
\]
The equality 
$\widetilde{H}(\phi)=B_{-\lambda_Y/r}\widetilde{H}(\Phi_E)B_{-\lambda_X/r}$ implies that $\widetilde{H}(\phi)$ maps $\span\{\beta_X\}$ to 
$\span\{\alpha_Y\}$. The adjoint $\Phi_E^\dagger$ has Fourier-Mukai kernel $E^*[2n]$. The argument above applied to the adjoint shows that $\widetilde{H}(\phi)^{-1}$ maps $\span\{\beta_Y\}$ to 
$\span\{\alpha_X\}$. Hence, $\widetilde{H}(\phi)$ maps $\span\{\alpha_X\}$ to $\span\{\beta_Y\}$.
It follows that $\widetilde{H}(\phi)$ maps the subspace $H^2(X,\QQ)$ orthogonal to $\span\{\alpha_X,\beta_X\}$ to 
the subspace $H^2(Y,\QQ)$ orthogonal to $\span\{\alpha_Y,\beta_Y\}$. 
This proves part (\ref{lemma-item-phi-restricts-to-a-Hodge-isometry-psi-E}) of the Lemma. 
We denote by 
$\widetilde{H}(\phi)_0$
the restriction of $\widetilde{H}(\phi)$ to $H^2(X,\QQ)$.

The isometry $\widetilde{H}(\phi)$ conjugates the grading operator $\tilde{h}_X$ of $\widetilde{H}(X,\QQ)$, given in (\ref{eq-grading-operator-of-LLV-lattice}), to $-\tilde{h}_Y$, where $\tilde{h}_Y$ is the grading operator of 
$\widetilde{H}(Y,\QQ)$. 
Lemma \ref{lemma-preservation-of-grading-of-LLV-lattice-implies-that-of-cohomology} implies that $\phi$ is degree reversing and part (\ref{lemma-item-phi-is-degree-reversing}) is proven.

(\ref{lemma-item-Ad-phi-resricts-to-Ad-psi-E})
The subalgebra $\bar{\LieAlg{g}}_X$ is the semi-simple part of the degree $0$ subalgebra of the LLV algebra $\LieAlg{g}_X$, by 
\cite[Prop. 4.5(ii)]{looijenga-lunts}, and 
$Ad_\phi:\LieAlg{g}_X\rightarrow \LieAlg{g}_Y$ is a graded Lie algebra isomorphism, by part (\ref{lemma-item-phi-is-degree-reversing}), and so $Ad_\phi$ maps  $\bar{\LieAlg{g}}_X$  to the semi-simple part $\bar{\LieAlg{g}}_Y$ of the degree $0$ subalgebra of $\LieAlg{g}_Y$. Furthermore, the restriction of $Ad_\phi$ to $\bar{\LieAlg{g}}_Y$ is equal to $Ad_{\widetilde{H}(\phi)_0}$, by 
part (\ref{lemma-item-phi-restricts-to-a-Hodge-isometry-psi-E}) and the commutativity of the outer square in Diagram (\ref{eq-commutative-diagram-of-Lie-algebras}).
Finally, $Ad_{\widetilde{H}(\phi)_0}=Ad_{\widetilde{H}_0(\phi)}$, as
$\widetilde{H}_0(\phi)=\pm \widetilde{H}(\phi)_0$, by definition (\ref{eq-functor-widetilde-H_0}).
\end{proof}

Set $\gamma:=\kappa(E)\sqrt{td_{X\times Y}}$ and let $\gamma_i$ be the graded summand in $H^{2i}(X\times Y,\QQ)$. 

\begin{cor}
\label{corollary-kappa-E-equal-kappa-E-dual}
If the odd cohomology of $X$ vanishes, then $\gamma_i=0$, for odd $i$. In that case $\kappa(E)=\kappa(E^\vee)$.
\end{cor}

\begin{proof}
The class $\gamma_i$ induces the homomorphism
$H^{k-2n}(X,\QQ)[2n]\rightarrow H^{k+2i-6n}(Y,\QQ)[2n]$. As $\phi$ is degree reversing, this homomorphism vanishes unless $k+i-4n=0$. Hence, it vanishes unless the parities of $i$ and $k$  are the same, and it vanishes if $k$ is odd, by assumption. 
\end{proof}

\begin{proof}[Proof of Lemma \ref{lemma-Lefschetz}]
Set $\tilde{\phi}:=\widetilde{H}(\phi)$.
The morphism $\phi$ conjugates the LLV Lie algebra of $X$ to that of $Y$, by \cite[Theorem A]{taelman}.
The LLV Lie algebras act faithfully on the rational LLV lattices.  Hence, it suffices to verify the equality
$
e_\lambda^\vee=\frac{2}{t(\lambda,\lambda)}\tilde{\phi}^{-1}e_{\tilde{\phi}(\lambda)}\tilde{\phi}
$
for the endomorphisms of $\widetilde{H}(X,\QQ)$.
We have
\begin{eqnarray*}
(\tilde{\phi}^{-1}e_{\tilde{\phi}(\lambda)}\tilde{\phi})(\alpha) &=& \tilde{\phi}^{-1}(e_{\tilde{\phi}(\lambda)}(t^{-1}\beta))=0,
\\
(\tilde{\phi}^{-1}e_{\tilde{\phi}(\lambda)}\tilde{\phi})(\beta) &=& \tilde{\phi}^{-1}(e_{\tilde{\phi}(\lambda)}(t\alpha))=
t\tilde{\phi}^{-1}(\tilde{\phi}(\lambda))=t\lambda,
\\
(\tilde{\phi}^{-1}e_{\tilde{\phi}(\lambda)}\tilde{\phi})(\lambda') &=&  
\tilde{\phi}^{-1}((\tilde{\phi}(\lambda),\tilde{\phi}(\lambda'))\beta)=t(\lambda,\lambda')\alpha.
\end{eqnarray*}
Set $\psi:=\tilde{\phi}^{-1}e_{\tilde{\phi}(\lambda)}\tilde{\phi}$. We get 
\begin{eqnarray*}
\hspace{0ex}
[e_\lambda,\psi](\alpha)&=&-\psi(e_\lambda(\alpha))=-t(\lambda,\lambda)\alpha,
\\
\hspace{0ex}
[e_\lambda,\psi](\beta)&=&e_\lambda(t\lambda)=t(\lambda,\lambda)\beta,
\\
\hspace{0ex}
[e_\lambda,\psi](\lambda')&=&t(\lambda,\lambda')\lambda-t(\lambda,\lambda')\lambda=0.
\end{eqnarray*}
Hence, $[e_\lambda,\psi]=\frac{t(\lambda,\lambda)}{2}h.$ 
Clearly, $[h,\psi]=-2\psi$, since $\psi$ has degree $-2$.
So, $e_\lambda^\vee=\frac{2}{t(\lambda,\lambda)}\psi$.
\hide{
Being degree reversing, $\phi$ maps the class $[pt]\in H^{4n}(X,\QQ)$ to $H^0(X,\QQ)$.
The image $\phi([pt])$ must be equal to the rank $n!r^n$ of $E$, as $\phi([pt])=\kappa(E_x)\sqrt{td_Y}$, where $E_x$ is the restriction of $E$ to
$\{x\}\times Y$, for some point $x\in X$. Hence, $\tilde{\phi}(\beta)=t\alpha$, where $t=r$ or $t=-r$, by Definition \ref{def-functor-tilde-H} of $\widetilde{H}(\phi)$. 
Note that $e_\lambda^\vee(\alpha)=0$, $e_\lambda^\vee(\beta)=\frac{2}{(\lambda,\lambda)}\lambda$, and 
$e_\lambda^\vee(\lambda')=\frac{2(\lambda,\lambda')}{(\lambda,\lambda)}\alpha$.
A straightforward calculation yields
$
\tilde{\phi}^{-1}e_{\tilde{\phi}(\lambda)}\tilde{\phi}=\frac{t(\lambda,\lambda)}{2}e^\vee_\lambda.
$
}
\end{proof}
%
\section{Hyperholomorphic vector bundles deforming a Fourier-Mukai kernel}
\label{sec-hyperholomorphic-vector-bundles-deforming-a-FM-kernel}

Let $X$ and $Y$ be deformation equivalent IHSMs and $E$ a locally free sheaf over $X\times Y$ which is the Fourier-Mukai kernel of an equivalence $\Phi_E:D^b(X)\rightarrow D^b(Y)$ of derived categories. Let $\phi:H^*(X,\QQ)\rightarrow H^*(Y,\QQ)$ be the morphism in $\Hom_{\G_{an}}(X,Y)$ induced by $\kappa(E)\sqrt{td_{X\times Y}}$.
Let $\widetilde{H}_0(\phi):H^2(X,\QQ)\rightarrow H^2(Y,\QQ)$ be the value on $\phi$ of the functor $\widetilde{H}_0$, given in (\ref{eq-functor-widetilde-H_0}). Let $\pi_X$ and $\pi_Y$ be the projections from $X\times Y$.
We prove in this section the following theorem.

\begin{thm} 
\label{thm-deformability-of-slope-stable-FM-kernels}
(Proposition \ref{prop-deforming-E} and Lemma \ref{lemma-Pi-i-is-surjective}).
Assume that there exists an open subcone of the K\"{a}hler cone of $X$, such that $\widetilde{H}_0(\phi)$ maps each K\"{a}hler class $\omega_1$ in this  subcone to a K\"{a}hler class $\omega_2:=\widetilde{H}_0(\phi)(\omega_1)$ on $Y$ with respect to which $E$ is $[\pi_X^*(\omega_1)+\pi_Y^*(\omega_2)]$-slope stable. Then for every IHSM $X'$ deformation equivalent to $X$ there exists an IHSM $Y'$ and a locally free (possibly twisted) coherent sheaf $E'$ over $X'\times Y'$, such that the pair $(X'\times Y',E')$ is deformation equivalent to $(X\times Y,E)$.
\end{thm}

Note that the assumptions of the theorem are symmetric with respect to $X$ and $Y$, upon replacing $E$ by $E^*$ and $\phi$ by $\phi^{-1}$, and so the conclusions are symmetric as well. Proposition \ref{prop-deforming-E} is a slightly stronger version of the theorem, allowing the assumptions in the statement of the theorem to hold instead for a deformation $(X_0\times Y_0,E_0)$ of the Fourier-Mukai kernel $(X\times Y,E)$. Theorem \ref{thm-deformability-of-slope-stable-FM-kernels} may be viewed as a global analogue, in the special case of classical and gerby deformations of locally free Fourier-Mukai kernels between IHSMs, of the infinitesimal deformability result of Toda for generalized deformations of Fourier-Mukai kernels \cite[Theorem 1.1]{toda}.

The proof of Theorem \ref{thm-deformability-of-slope-stable-FM-kernels} uses Verbitsky's results on hyperholomorphic vector bundles over hyper-K\"{a}hler manifolds and their deformability along twistor paths in moduli \cite{kaledin-verbitsky-book}.
In Section \ref{sec-twistor-lines} we recall the notion of a twistor line in the moduli space of marked IHSMs.
In section \ref{sec-moduli-of-rational-Hodge-isometries} we define the moduli space $\fM_\psi$ of rational Hodge isometries. 
Given a lattice $\Lambda$ and a rational isometry $\psi\in O(\Lambda_\QQ)$, the moduli space $\fM_\psi$ parametrizes isomorphism classes of quadruples 
$(X_1,\eta_1,X_2,\eta_2)$ of deformation equivalent marked pairs $(X_i,\eta_i)$, where $\eta_i:H^2(X_i,\Integers)\rightarrow\Lambda$ is an isometry, $i=1,2$,  such that $\eta_2^{-1}\psi\eta_1:H^2(X_1,\QQ)\rightarrow H^2(X_2,\QQ)$ is a Hodge isometry which maps some K\"{a}hler class to a K\"{a}hler class.
In Section \ref{sec-indeed-a-twistor-line} we define diagonal twistor lines in $\fM_\psi$. In Section
\ref{sec-diagonal-twistor-paths} we define the notion of generic twistor paths and prove that every two points in the same connected component of $\fM_\psi$ are connected by such a path. In Section \ref{sec-deforming-vector-bundles-along-diagonal-twistor-paths} we review Verbitsky's result about hyperholomorphic vector bundles and their deformability along twistor paths in $\fM_\psi.$ In Section \ref{sec-deforming-the-FM-kernel}
we prove Theorem \ref{thm-deformability-of-slope-stable-FM-kernels}.

%
\subsection{Twistor lines}
\label{sec-twistor-lines}

A {\em marking} of an irreducible holomorphic symplectic manifold $X$ is an isometry $\eta:H^2(X,\Integers)\rightarrow \Lambda$ with a fixed lattice $\Lambda$. 
Two marked pairs $(X_1,\eta_1)$ and $(X_2,\eta_2)$ are isomorphic, if there exists  an isomorphism $f:X_1\rightarrow X_2$, such that 
$\eta_2=\eta_1\circ f^*$. 
The moduli space $\fM_\Lambda$  of isomorphism classes of marked irreducible holomorphic symplectic manifolds is  a  non-Hausdorff complex manifold of dimension $\rank(\Lambda)-2$ \cite{huybrects-basic-results}.

Set $\Omega_\Lambda:=\{\ell\in \PP(\Lambda_\CC) \ : \ (\ell,\ell)=0 \ \mbox{and} \ (\lambda,\bar{\lambda})>0 \ \mbox{for all non-zero} \ \lambda\in\ell\}$. The {\em period map} 
\begin{equation}
\label{eq-period-map}
P:\fM_\Lambda \rightarrow \Omega_\Lambda,
\end{equation}
is given by $P(X,\eta)=\eta(H^{2,0}(X))$. The restriction of the  period map to every connected component of $\fM_\Lambda$  is known to be surjective and generically injective, by Verbitsky's Torelli Theorem
\cite{huybrechts-torelli,verbitsky-torelli}. Furthermore, if $\Pic(X)$ is trivial, or cyclic generated by a class of non-negative BBF-degree, then $(X,\eta)$ is the unique point in the fiber of $P$ through the connected component of $\fM_\Lambda$ containing $(X,\eta)$. 

The Torelli Theorem solves the case of Theorem \ref{main-thm} in which the Hodge isometry $f$ is a parallel-transport operator. 

\begin{thm}
\label{thm-two-inseparable-marked-pairs}
Let $(X,\eta)$ and $(X',\eta')$ be two marked  pairs in the same connected component of $\fM_\Lambda$ satisfying
$P(X,\eta)=P(X',\eta')$. There exist an analytic cycle $Z$ in $X\times X'$, such that $[Z]_*:H^*(X',\Integers)\rightarrow H^*(X,\Integers)$ is a parallel-transport operator, which restricts to $H^2(X'_i,\Integers)$ as $\eta^{-1}\circ\eta'$.
%
\end{thm}

\begin{proof}
The assumption that  $(X,\eta)$ and $(X',\eta')$ belong to the same connected component of $\fM_\Lambda$ and have the same period implies that 
$(X,\eta)$ and $(X',\eta')$ are inseparable points of $\fM_\Lambda$, by Verbitsky's Torelli Theorem \cite{huybrechts-torelli,verbitsky-torelli}. Note also that 
$\eta'\circ\eta^{-1}$ is a parallel-transport operator, which is a Hodge isometry. The existence of the analytic cycle $Z$ with the claimed properties thus follows from results of Huybrechts in \cite{huybrechts-kahler-cone} (see \cite[Theorem 3.2]{markman-survey}).
\hide{
The composition $\eta_2^{-1}\psi\eta_1:H^2(X_1,\QQ)\rightarrow H^2(X_2,\QQ)$ is a Hodge isometry. Hence,
\[
P(X_2,\eta_2):=\eta_2(H^{2,0}(X_2))=[\eta_2\circ(\eta_2^{-1}\psi\eta_1)](H^{2,0}(X_1))=\psi(P(X_1,\eta_1)).
\]
Similarly, $P(X'_2,\eta'_2)=\psi(P(X'_1,\eta'_1)).$ The equality $P(X_2,\eta_2)=P(X'_2,\eta'_2)$ thus follows from the equality 
$P(X_1,\eta_1)=P(X'_1,\eta'_1)$. 

The assumption that $q$ and $q'$ belong to the same connected component of $\fM_\psi$ implies that $(X_i,\eta_i)$ and $(X'_i,\eta'_i)$ belong to the same connected component of $\fM_\Lambda$ for $i=1,2$. Thus, $\eta'_i\circ\eta_i^{-1}$ is a parallel-transport operator, which is a Hodge isometry. It follows that $(X_i,\eta_i)$ and $(X'_i,\eta'_i)$ are inseparable points of $\fM_\Lambda$, for $i=1,2$, by Verbitsky's Torelli Theorem \cite{huybrechts-torelli,verbitsky-torelli}. The existence of the analytic cycle $Z_i$ with the claimed properties thus follows from results of Huybrechts \cite[Theorem 3.2]{markman-survey}.
}
\end{proof}

\begin{lem}
\label{lemma-parallel-transports-with-same-restriction-to-H-2}
Let $X$ and $X'$ be IHSMs and let $f_i:H^*(X,\ZZ)\rightarrow H^*(X',\ZZ)$ be a parallel-transport operator, $i=1,2$. 
Assume that  $f_1$ and $f_2$ restrict to the same parallel-transport operator $f:H^2(X,\ZZ)\rightarrow H^2(X',\ZZ)$.
Then there exist automorphisms $g\in\Aut(X)$ and $g'\in \Aut(X')$, acting trivially on the second cohomologies, 
such that $f_2=g'_*\circ f_1\circ g_*$.
\end{lem}

\begin{proof}
There exist families $\pi_i:\X_i\rightarrow B_i$ over connected analytic spaces $B_i$, points $b_i$, $b_i'\in B_i$, and continuous paths $\gamma_i$ from $b_i$ to $b_i'$ in $B_i$, such that $f_i$ is the parallel-transport operator along $\gamma_i$ in the local system $R\pi_{i,*}\ZZ$, $i=1,2$. We may and do assume that $B_i$ is simply connected, possibly after passing to the universal cover and lifting the path $\gamma_i$. Choose an isometric trivialization $\eta_i:R^2\pi_{i,*}\ZZ\rightarrow \underline{\Lambda}$,
where $\underline{\Lambda}$ is the trivial local system over $B_i$ with a fixed lattice $\Lambda$, with $\eta_{1,b_1}=\eta_{2,b_2}:H^2(X,\ZZ)\rightarrow \Lambda$. Set $\eta:=\eta_{1,b_1}$. Set $\eta':=\eta\circ f^{-1}$.
We get that $\eta_{i,b_i'}=\eta'$, for $i=1,2$. 

Let $\fM_\Lambda^0$ be the connected component of the moduli space of isomorphism classes of marked IHSMs containing the isomorphism class $[(X,\eta)]$ of $(X,\eta)$. 
Let $s_i:B_i\rightarrow \fM_\Lambda^0$ be the classifying morphism associated to $\eta_i$. Set $q:=[(X,\eta)]=s_1(X,\eta_{1,b_1})=s_2(X,\eta_{2,b_2})$ and $q':=[(X,\eta')]=s_1(X',\eta_{1,b_1'})=s_2(X',\eta_{2,b_2'})$. 

There exists over $\fM_\Lambda^0$ a universal family $\Pi:\Y\rightarrow \fM_\Lambda^0$ and isomorphisms $h_i:\X_i\rightarrow s_i^{-1}(\Y)$ of the two families 
$\pi_i$ and $s_i^{-1}(\Pi)$ over $B_i$, by \cite[Theorem 1.1]{markman-universal-family} and our assumption that $B_i$ is simply connected, for $i=1,2$. Furthermore, the local system $R\Pi_*\ZZ$ over $\fM_\Lambda^0$ is trivial, by \cite[Lemma 2.1]{markman-universal-family}. Hence, the paths $s_i\circ \gamma_i$, $i=1,2$, from $q$ to $q'$ in $\fM_\Lambda^0$ induce the same parallel-transport operator from the fiber $Y_q$ of $\Y$ over $q$ to the fiber  $Y_{q'}$ of $\Y$ over $q'$.
Let $h_{i,b_i}:X\rightarrow Y_q$ and $h_{i,b_i'}:X'\rightarrow Y_{q'}$ be the isomorphism associated to $h_i$, $i=1,2$. 
We conclude that $h_{i,b_i'}\circ f_i\circ h^{-1}_{i,b_i}$ is the parallel-transport along $s_i\circ \gamma_i$, $i=1,2$, and so 
$h_{1,b_1'}\circ f_1\circ h_{1,b_1}^{-1}=h_{2,b_2'}\circ f_2\circ h_{2,b_2}^{-1}.$
Then
$g:=h_{1,b_1}^{-1}\circ h_{2,b_2}\in\Aut(X)$ and $g':=h_{2,b_2'}^{-1}\circ h_{1,b_1'}\in\Aut(X')$ are automorphisms satisfying $f_2=g'_*\circ f_1\circ g_*$.
\end{proof}

Given a positive definite $3$-dimensional subspace $W$ of $\Lambda_\RR$, set $W_\CC:=W\otimes_\RR\CC$.  We get the projective conic curve
$\PP(W_\CC)\cap\Omega_\Lambda$, which is called a {\em twistor line}.
Let $(X,\eta)$ be a marked pair in $\fM_\Lambda$ and $\omega\in H^{1,1}(X,\RealNumbers)$ a K\"{a}hler class. 
The intersection of $W_\CC:=H^{2,0}(X)\oplus H^{0,2}(X)\oplus \CC\omega$ with $H^2(X,\RR)$ is a positive definite $3$-dimensional subspace $W$ and the twistor line $Q_{(X,\eta,\omega)}:=\PP(\eta(W_\CC))\cap\Omega_\Lambda$ passes through $P(X,\eta)$. 
We endow $W$ with the orientation associated to the basis \begin{equation}
\label{eq-basis-inducing-orientation}
\{Re(\sigma),Im(\sigma),\omega\}, 
\end{equation}
for a non-zero class $\sigma\in H^{2,0}(X)$. 
There exists a family 
\begin{equation}
\label{eq-twistor-family}
\pi:\X\rightarrow Q_{(X,\eta,\omega)}
\end{equation}
of irreducible holomorphic symplectic manifolds with fiber $X$ over the point $P(X,\eta)$, known as the {\em twistor family associated to} $\omega$, with the following properties \cite[paragraph 1.17
page 76]{huybrects-basic-results}. The marking $\eta$ extends to a trivialization of the local system $R^2\pi_*\Integers$, as
$Q_{(X,\eta,\omega)}$ is simply connected. Denote by $\eta_t$ the resulting marking of the fiber $X_t$ of $\pi$ over $t\in Q_{(X,\eta,\omega)}$. Then 
\[
P(X_t,\eta_t)=t,
\]
for each $t\in Q_{(X,\eta,\omega)}$ and the line in $H^2(X_t,\RR)$, which is contained in $\eta_t^{-1}(\eta(W))$ and orthogonal to 
$H^{2,0}(X_t)$, is spanned by a K\"{a}hler class $\omega_t$. The twistor family depends only on the ray in the K\"{a}hler cone of $X$ through $\omega$.
Denote by
\begin{equation}
\label{eq-ray-R-t}
R_t \subset  \K_{X_t}\cap \eta_t^{-1}(\eta(W))
\end{equation}
the ray through $\omega_t$ in the K\"{a}hler cone $\K_{X_t}$ of $X_t$.
The curve
\[
\widetilde{Q}_{(X,\eta,\omega)}:=\{(X_t,\eta_t) \ : \ t\in Q_{(X,\eta,\omega)} 
\}
\]
in $\fM^0_\Lambda$ is mapped isomorphically onto $Q_{(X,\eta,\omega)}$ via the period map.
The curve $\widetilde{Q}_{(X,\eta,\omega)}$ is called the {\em twistor line through $(X,\eta)$ associated to} $\omega$. 
The twistor line $\widetilde{Q}_{(X,\eta,\omega)}$ is said  to be {\em generic}, if $\Pic(X_t)$ is trivial, for some $t\in  Q_{(X,\eta,\omega)}$. This is the case 
for $\omega$ in the complement of a countable union of hyperplanes in the K\"{a}hler cone.

%
\subsection{Moduli spaces of rational Hodge isometries}
\label{sec-moduli-of-rational-Hodge-isometries}

Let $X$ and $Y$ be the irreducible holomorphic symplectic manifolds in Lemma \ref{lemma-phi-is-degree-reversing}.
Choose markings $\eta_X:H^2(X,\ZZ)\rightarrow \Lambda$ and $\eta_Y:H^2(Y,\ZZ)\rightarrow \Lambda$ such that
$(X,\eta_X)$ and $(Y,\eta_Y)$ belong to the same connected component $\fM_\Lambda^0$ of $\fM_\Lambda$.

\begin{convention}
\label{convention-on-orientations-and-the-groupoid-G}
Choose an orientation $\epsilon_\Lambda$ of $\Lambda\otimes_\ZZ\QQ$. 
Given two pairs $(X_1,\eta_1)$ and $(X_2,\eta_2)$ in $\fM_\Lambda$, we endow $H^2(X_i,\QQ)$ with the orientation $\epsilon_i$, such that $\eta_i$ is orientation preserving. The homomorphism
\[
\widetilde{H}:\Hom_\G(X_1,X_2)\rightarrow \Hom(\widetilde{H}(X_1,\QQ),\widetilde{H}(X_2,\QQ)),
\]
induced by the functor (\ref{eq-functor-widetilde-H}),
is then always evaluated in terms of the orientations $\epsilon_1$ and $\epsilon_2$. The resulting homomorphism depends on the markings $\eta_1$ and $\eta_2$, but is independent of the choice of $\epsilon_\Lambda$. 
\end{convention}

The unit sphere in any positive definite $3$-dimensional subspace of $H^2(X,\RR)$ is a deformation retract of the positive  cone $\C'_X:=\{x\in H^2(X,\RR) \ ; \ (x,x)>0\}$ (see \cite[Sec 4]{markman-survey}). Hence, the second cohomology of the positive cone is a one-dimensional representation of $O(H^2(X,\QQ))$ corresponding to a character $\nu:O(H^2(X,\QQ))\rightarrow \mu_2:=\{\pm 1\}$. Set 
\[
O^+(H^2(X,\QQ)):=\ker(\nu).
\]
The orientation given by the basis (\ref{eq-basis-inducing-orientation}) is independent of the choices of a holomorphic $2$-form and of a K\"{a}hler class and determines a {\em positive} generator of the cyclic group $H^2(\C'_X,\Integers)$. 

\begin{defi}
\label{def-orientation-of-the-positive-cone}
The {orientation of the positive cone} is the positive generator of $H^2(\C'_X,\Integers)$.
An isometry
$g:H^2(X_1,\QQ)\rightarrow H^2(X_2,\QQ)$ is {\em orientation preserving}, if $g_*:H^2(\C'_{X_1},\ZZ)\rightarrow H^2(\C'_{X_2},\ZZ)$
maps the positive generator to the positive generator. 
\end{defi}
Set
\[
\nu(g) = \left\{
\begin{array}{ccc}
1 & \mbox{if} & g  \ \mbox{is orientation preserving}
\\
-1 & \mbox{if} & g \  \mbox{is orientation reveresing}.
\end{array}
\right.
\]
Note that if $g$ is a Hodge isometry, which maps some K\"{a}hler class to a K\"{a}hler class, then $\nu(g)=1$, as it maps a basis of the form (\ref{eq-basis-inducing-orientation}) to such a basis.

\begin{rem}
\label{remark-convention-behaves-well-with-respect-to-parallel-transport-operators}
Keep the notation of Convention \ref{convention-on-orientations-and-the-groupoid-G}.
Let $g:H^*(X_1,\ZZ)\rightarrow H^*(X_2,\ZZ)$ be a parallel-transport operator and denote by $\bar{g}$ its restriction to 
$H^2(X_1,\ZZ)$. 
One checks that if $\bar{g}$ is orientation preserving, so that 
$\det(\eta_2 \bar{g}\eta_1^{-1})=1$, 
then $\widetilde{H}(g)$ maps $\alpha$ to $\alpha$, $\beta$ to $\beta$, and restricts to $H^2(X_1,\QQ)$ as  $\bar{g}$, so that $\widetilde{H}_0(g)=\bar{g}$ (see\footnote{The statement is true for any parallel-transport operator, when $n:=\dim(X_i)/2$ is odd. When $n$ is even we explain in \cite[Remark 5.4]{markman-modular} Taelman's sign convention in terms of a functor $\chi$ from $\G$ to $\mu_2:=\{\pm 1\}$, such that $\chi(\phi)\widetilde{H}(\phi)$ is orientation preserving, for every $\phi\in\Hom_\G(X_1,X_2)$, and $\chi(g)=1$, for every parallel-transport operator. So when $n$ is  even, the restrictions of $\widetilde{H}(g)$ to $H^2(X_1,\Integers)$ is equal to that of $g$, if $g$ is orientation preserving, and is equal to that of $-g$, if $g$ is orientation reversing. 
} 
\cite[Remark 5.4]{markman-modular}). In particular, this is the case when $(X_1,\eta_1)$ and $(X_2,\eta_2)$ belong to the same connected component of $\fM_\Lambda$ and 
$\bar{g}=\eta_2^{-1}\eta_1$.
\end{rem}

Let $E$ be the Fourier-Mukai kernel of non-zero rank of an equivalence $\Phi_E:D^b(X)\rightarrow D^b(Y)$ of derived categories of two deformation equivalent IHSMs $X$ and $Y$. Set $\phi:=[\kappa(E)\sqrt{td_{X\times Y}}]_*:H^*(X,\QQ)\rightarrow H^*(Y,\QQ)$ and let $\widetilde{H}(\phi)_0$ be the restriction of $\widetilde{H}(\phi)$ to $H^2(X,\QQ)$.
Set 
\begin{equation}
\label{eq-psi-E+}
\widetilde{H}_0(\phi):=\nu(\widetilde{H}(\phi)_0)\widetilde{H}(\phi)_0, 
\end{equation}
We denote $\widetilde{H}_0(\phi)$ also by $\psi_E$ to emphasis it dependence on $E$.
\[
\psi_E:=\widetilde{H}_0([\kappa(E)\sqrt{td_{X\times Y}}]_*).
\]
Let $\psi:\Lambda_\QQ\rightarrow \Lambda_\QQ$, be the isometry given by
\begin{equation}
\label{eq-psi}
\psi=\eta_Y\widetilde{H}_0(\phi)\eta_X^{-1}.
\end{equation}
Then $\psi$ belongs to $O^+(\Lambda_\QQ)$ (see \cite[Sec 4]{markman-survey}).

\begin{defi}
An {\em Azumaya} $\StructureSheaf{X}$-algebra of rank $r$ over a complex manifold $X$ is a sheaf $\A$ of locally free coherent $\StructureSheaf{X}$-modules, with a global section $1_\A$, and an $\StructureSheaf{X}$-linear associative multiplication $m:\A\otimes_{\StructureSheaf{X}}\A\rightarrow \A$ with identity $1_\A$, admitting an open covering $\{U_\alpha\}$ of $X$ and an isomorphism $\eta_\alpha:\restricted{\A}{U_\alpha}\rightarrow \SheafEnd(F_\alpha)$ of unital associative algebras for some locally free $\StructureSheaf{U_\alpha}$-module $F_\alpha$ of rank $r$ over each $U_\alpha$. We will also refer to such $\A$ as an {\em Azumaya algebra over $X$}.
\end{defi}

\begin{rem}
The endomorphism sheaf $\SheafEnd(F)$ of a locally free twisted sheaf $F$ over $X$ is an Azumaya algebra over $X$.
Conversely, an Azumaya algebra corresponds to a natural equivalence class of twisted sheaves over $X$, modulo tensorization by line-bundles \cite{caldararu-thesis}. The class $\kappa(F)$ depends only on this equivalence class and so $\kappa(\A)$ is well defined \cite[Sec. 2.3]{markman-BBF-class-as-characteristic-class}
\end{rem}



\begin{assumption}
\label{assumption-quadruple-is-in-fM-psi}
\begin{enumerate}
\item
Assume that the Fourier-Mukai kernel $E$ in Lemma \ref{lemma-phi-is-degree-reversing} 
is represented by a locally free sheaf, which we denote by $E$ as well. 
\item
\label{assumption-item-psi-E-maps-a-Kahler-class-to-same}
Assume that there exist\footnote{If the isometry $\psi_E$, given in Equation (\ref{eq-psi-E+}), maps 
some K\"{a}hler class $\omega_1$ on $X$ to a K\"{a}hler class $\omega_2$ on $Y$, then one can take the constant families with $B$ a point. This is the case in the $K3$ surface examples used in \cite{buskin}. 
} 
proper and smooth families of K\"{a}hler manifolds $p:\X\rightarrow B$ and $q:\Y\rightarrow B$, over the same simply connected analytic base $B$, a possibly twisted vector bundle $\E$ over $\X\times_B \Y$, a point $b\in B$, isomorphisms
$f:X\rightarrow X_b$ and $g:Y\rightarrow  Y_b$, and an isomorphism of Azumaya algebras $\SheafEnd(E)\cong (f\times g)^*\SheafEnd(\E)$, where $X_b$ and $Y_b$ are the fibers of $\X$ and $\Y$ over $b$, such that the value of the parallel-transport $\psi_{E_0}$ of $\psi_E$ in the local system $(R^2p_*\QQ)^*\otimes R^2q_*\QQ$ at some point $0\in B$ maps some K\"{a}hler class $\omega_1$ on the fiber $X_0$ of $p$ to a K\"{a}hler class $\omega_2$ over the fiber $Y_0$ of $q$. 
\item
\label{assumption-item-stability}
$E_0$ is $[\pi_{X_0}^*(\omega)+\pi_{Y_0}^*(\psi_{E_0}(\omega))]$-slope-stable, for classes $\omega$ in a non-empty subcone of the intersection $\K_{X_0}\cap\psi_{E_0}^{-1}(\K_{Y_0})$.
\end{enumerate}
\end{assumption}

\begin{rem}
\label{rem-phi-0-is-degree-reversing}
Note that the class $\kappa(E_0)\sqrt{td_{X_0\times Y_0}}$ induces the degree reversing Hodge isometry (with respect to the Mukai pairings) 
\begin{equation}
\label{eq-phi-induced-by-kappa-E}
\phi_0:H^*(X_0,\QQ)\rightarrow H^*(Y_0,\QQ), 
\end{equation}
which is the parallel-transport of $\phi$ in Lemma \ref{lemma-phi-is-degree-reversing}, hence a morphism in $\Hom_\G(X_0,Y_0)$,  $\widetilde{H}(\phi_0)$ restricts to a Hodge isometry $\widetilde{H}(\phi_0)_0:H^2(X_0,\QQ)\rightarrow H^2(Y_0,\QQ)$, which is the parallel-transport of $\widetilde{H}(\phi)_0$ in the Lemma, and 
the 
parallel-transport $\psi_{E_0}:H^2(X_0,\QQ)\rightarrow H^2(Y_0,\QQ)$ of $\psi_E$ in Assumption \ref{assumption-quadruple-is-in-fM-psi}(\ref{assumption-item-psi-E-maps-a-Kahler-class-to-same})
is equal to $\widetilde{H}_0(\phi_0)$.
\end{rem}

Let $\fM_\psi$ be the moduli space  of isomorphism classes of quadruples $(X_1,\eta_1,X_2,\eta_2)$,
where $\eta_i:H^2(X_i,\Integers)\rightarrow \Lambda$ is a marking, such that 
$\eta_2^{-1}\psi\eta_1:H^2(X_1,\QQ)\rightarrow H^2(X_2,\QQ)$ is a Hodge isometry
which maps some K\"{a}hler class to a K\"{a}hler class, where $\psi$ is given in (\ref{eq-psi}). 
Two quadruples $(X_1,\eta_1,X_2,\eta_2)$ and $(X'_1,\eta'_1,X'_2,\eta'_2)$ are isomorphic, if
the marked pair $(X_i,\eta_i)$ is isomorphic to $(X'_i,\eta'_i)$, for $i=1,2$.
The moduli space $\fM_\psi$ has a natural structure of a non-Hausdorff complex manifold. The proof of the latter statement is identical to that  of \cite[Prop. 4.9]{buskin}. The isometry $\psi$ induces an automorphism of $\Omega_\Lambda$ which we denote by $\psi$ as well.

\begin{lem}
\label{lemma-periods-of-quadruple}
If $(X_1,\eta_1,X_2,\eta_2)$ belongs to $\fM_\psi$, then $P(X_2,\eta_2)=\psi(P(X_1,\eta_1)).$
\end{lem}

\begin{proof}
The composition $\eta_2^{-1}\psi\eta_1:H^2(X_1,\QQ)\rightarrow H^2(X_2,\QQ)$ is a Hodge isometry. Hence,\\
$
P(X_2,\eta_2):=\eta_2(H^{2,0}(X_2))=[\eta_2\circ(\eta_2^{-1}\psi\eta_1)](H^{2,0}(X_1))=\psi(P(X_1,\eta_1)).
$
\end{proof}

Let $\eta_{X_0}$ be the marking of $X_0$ resulting in the extension of $\eta_X$ to a trivialization of $R^2p_*\ZZ$ for the family $p:\X\rightarrow B$ in Assumption \ref{assumption-quadruple-is-in-fM-psi}. Define $\eta_{Y_0}$ similarly. 
Note that $\psi(\eta_{X_0}(\omega_1))=\eta_{Y_0}(\omega_2)$ for some K\"{a}hler classes $\omega_1$ and $\omega_2$, by Assumption \ref{assumption-quadruple-is-in-fM-psi},  and so $(X_0,\eta_{X_0},Y_0,\eta_{Y_0})$ belongs to $\fM_\psi$. 
Denote by 
\begin{equation}
\label{eq-M-psi-0}
\fM_\psi^0
\end{equation} 
the  connected component of $\fM_\psi$ containing $(X_0,\eta_{X_0},Y_0,\eta_{Y_0})$.

Given a quadruple $q:=(X_1,\eta_1,X_2,\eta_2)$ in $\fM_\psi$ set
\begin{equation}
\label{eq-psi-q}
\psi_q:=\eta_2^{-1}\psi\eta_1:H^2(X_1,\QQ)\rightarrow H^2(X_2,\QQ).
\end{equation}
Let $\omega_1$ be a K\"{a}hler class  on $X_1$, such that $\omega_2:=\psi_q(\omega_1)$ is a K\"{a}hler class on $X_2$. 
Note that $\psi:\Omega_\Lambda\rightarrow \Omega_\Lambda$ maps the twistor line $Q_{(X_1,\eta_1,\omega_1)}$
to the twistor line $Q_{(X_2,\eta_2,\omega_2)}$. 
Let $(X_{i,t},\eta_{i,t})$, $t\in Q_{(X_i,\eta_i,\omega_i)}$, be the point of $\widetilde{Q}_{(X_i,\eta_i,\omega_i)}$ over $t$, for $i=1,2$.

\begin{lem}
\label{lemma-twistor-line-of-quadruples}
The rays $R_t$ and $R_{\psi(t)}$ in the K\"{a}hler cones of $X_{1,t}$ and $X_{2,\psi(t)}$, given in (\ref{eq-ray-R-t}), satisfy the equality 
$\psi(\eta_{1,t}(R_t))=\eta_{2,\psi(t)}(R_{\psi(t)})$, 
for each $t\in Q_{(X_1,\eta_1,\omega_1)}$. Consequently, the quadrupe 
$(X_{1,t},\eta_{1,t},X_{2,\psi(t)},\eta_{2,\psi(t)})$ belongs to $\fM_\psi$, for all $t\in Q_{(X_1,\eta_1,\omega_1)}$.
\end{lem}

\begin{proof}
Let $\C_{X_i}:=\{\lambda \in H^{1,1}(X_i,\RR) \ : \ (\lambda,\lambda)>0\}$ be the positive cone, $i=1,2$.
Each $\C_{X_i}$ has two connected components and we denote  by $\C_{X_i}^+$ the one which contains the K\"{a}hler cone.
The isometry $\psi_q$ maps $\C_{X_1}$ isomorphically onto  $\C_{X_2}$ and 
$\psi_q(\K_{X_1})\cap K_{X_2}$ is non-empty, hence 
$\psi_q(\C_{X_1}^+)=\C_{X_2}^+$. Equivalently, 
$\psi(\eta_{1,t_0}(\C^+_{X_{t_0}}))=\eta_{2,\psi(t_0)}(\C^+_{X_{\psi(t_0)}})$, where $t_0=P(X_1,\eta_1)$.
Hence, $\psi(\eta_{1,t}(\C^+_{X_{t}}))=\eta_{2,\psi(t)}(\C^+_{X_{\psi(t)}})$, for all $t\in Q_{(X_1,\eta_1,\omega_1)}$,
by continuity and connectedness of $Q_{(X_1,\eta_1,\omega_1)}$.  We know that $\psi$ maps the line spanned by $\eta_{1,t}(R_t)$ to the line spanned by  $\eta_{2,\psi(t)}(R_{\psi(t)})$, since $\eta_{2,t}^{-1}\psi\eta_{1,t}$ is a Hodge isometry which maps the positive definite $3$-dimensional subspace $W_1$ associated to $(X_1,\eta_1,\omega_1)$ to the positive definite $3$-dimensional subspace $W_2$ associated to $(X_2,\eta_2,\omega_2)$ and $\eta_{1,t}(R_t)$ spans the line in $W_1$ orthogonal to $\eta_{1,t}(H^{2,0}(X_t))$ and $\eta_{2,t}(R_{\psi(t)})$ spans the line in $W_2$ orthogonal to $\eta_{2,t}(H^{2,0}(X_{\psi(t)}))$.
Now, $R_t$ is the intersection of the  line spanned by it with $\C^+_{X_{t}}$ and the similar statement holds for $R_{\psi(t)}$. Hence, $\psi(\eta_{1,t}(R_{t}))=\eta_{2,\psi(t)}(R_{\psi(t)})$.
\end{proof}

Lemma \ref{lemma-twistor-line-of-quadruples} enables us to define 
the {\em 
twistor line in $\fM_\psi^0$  through $(X_1,\eta_1,X_2,\eta_2)$ determined by $\omega_1\in \K_{X_1}\cap\psi_q^{-1}(\K_{X_2})$} 
to be the curve
\begin{equation}
\label{eq-twistor-line-through-quadruple}
\widetilde{Q}_{(X_1,\eta_1,X_2,\eta_2),\omega_1}:=
\{
(X_{1,t},\eta_{1,t},X_{2,t},\eta_{2,\psi(t)}) \ : \ t\in Q_{(X_1,\eta_1,\omega_1)}
\}.
\end{equation}

%
\subsection{Diagonal twistor lines}
\label{sec-indeed-a-twistor-line}
We explain next that the twistor line (\ref{eq-twistor-line-through-quadruple})
in the above definition is associated to a hyper-K\"{a}hler structure on the differential manifold underlying the product $X_1\times X_2$. This fact will be used in the proof of Proposition 
\ref{prop-slope-stable-vb-with-kappa-2-which-remains-Hodge-class-is-modular} below.
A K\"{a}hler class on an irreducible holomorphic symplectic manifold $X$ determines a Ricci-flat hermitian metric $g$ on the manifold $M$ underlying $X$, by the Calabi-Yau theorem \cite{beauville-varieties-with-zero-c-1}.
The metric $g$ admits three  complex structures $I$, $J$, and $K$, satisfying the quaternionic relation $IJ=K$, such that $X$ is the complex manifold $(M,I)$. For each point $t$ in the unit sphere $S^2:=\{(a,b,c) \ : \ a^2+b^2+c^2=1\}$, we get the
unit purely imaginary quaternion $I_t:=aI+bJ+cK$,  which is a complex structure on $M$. The $2$-form $\omega_{I_t}:=g(I_t(\bullet),(\bullet))$ is a K\"{a}hler form.  The original class $\omega$ is the cohomology class of the $2$-form $\omega_I:=g(I(\bullet),(\bullet))$. 
The classes of $\{\omega_I,\omega_J,\omega_K\}$ form a basis for the positive definite $3$-dimensional subspace $W\subset H^2(X,\RR)$ 
whose complexification is $H^{2,0}(X)\oplus H^{0,2}(X)\oplus \CC\omega$. The latter basis is orthogonal with respect to the BBF-pairing, and the BBF-degrees of the three classes are the same. Conversely, a choice of an orthogonal basis of $W$ with the same orientation consisting of $\omega$ and two classes of the same BBF-degree as $\omega$ determines complex structures $J$ and $K$ and an action of the unit purely imaginary quaternions on the tangent bundle of $M$ via parallel complex structures. The above $2$-sphere of complex structures is thus naturally identified with the unit sphere in  $W$.

The standard complex structure on the unit sphere combines with the above construction to define a complex structure on $S^2\times M$, such that the first projection $\pi:S^2\times M\rightarrow S^2$ is holomorphic and the fiber over $t$ is $(M,I_t)$, giving rise to the twistor family (\ref{eq-twistor-family}) (see \cite{HKLR}).
The base of the twistor family  is thus naturally identified with the unit sphere in the positive definite $3$-dimensional subspace $W$ of $H^2(X,\RealNumbers)$ via the isomorphism sending a class $\omega'$ in the unit sphere in $W$ to the unique\footnote{
The intersection of the conic $Q$ in $\PP(W_\CC)$ of isotropic lines in $W_\CC$ with the line $\PP((\omega')^\perp_\CC\cap W_\CC)$ orthogonal to $\omega'$ consists of two complex conjugate points inducing on $W$ distinct orientations.}
isotropic line in $\PP(W_\CC)$, which is orthogonal to $\omega'$ and induces on $W$ the same orientation induced by $H^{2,0}(X)$ in (\ref{eq-basis-inducing-orientation}).

Let $M_i$ be the differential manifold underlying $X_i$, $i=1,2$.
Let $g_1$ be the Ricci-flat hermitian metric on  $M_1$ associated with the K\"{a}hler class $\omega_1$ in 
(\ref{eq-twistor-line-through-quadruple}) and let $g_2$ be the metric on $M_2$ associated with the K\"{a}hler class $\omega_2:=\psi_q(\omega_1)$. 
Let $W_i$ be the positive definite $3$-dimensional subspace of $H^2(X_i,\RR)$, whose complexification is $H^{2,0}(X_i)\oplus H^{0,2}(X_i)\oplus\CC w_i$. We have the equality $W_2=\psi_q(W_1)$, and $\psi_q$ maps the unit sphere in $W_1$ to that in $W_2$, 
as $\psi_q$ is a Hodge isometry. Hence, $\psi$ restricts to an isometry from $\eta_1(W_1)$ onto $\eta_2(W_2)$, which maps the twistor line $Q_{(X_1,\eta_1,\omega_1)}$ isomorphically onto $Q_{(X_2,\eta_2,\omega_2)}$. 
Denote by
\[
Q_{(X_1,\eta_1,X_2,\eta_2),\omega_1} \subset Q_{(X_1,\eta_1,\omega_1)}\times Q_{(X_2,\eta_2,\omega_2)}
\]
the graph of the restriction of $\psi$. Let $\pi_i:\X_i\rightarrow Q_{(X_i,\eta_i,\omega_i)}$ be the twistor family through $(X_i,\eta_i)$ associated to $\omega_i$, $i=1,2$.
The twistor line (\ref{eq-twistor-line-through-quadruple})
is the lift of $Q_{(X_1,\eta_1,X_2,\eta_2),\omega_1}$ to a curve in $\fM^0_\psi$ obtained via the extensions of $\eta_i$ to a trivialization of $R\pi_{i,*}^2\ZZ$, $i=1,2$. Choose an orthogonal basis of $W_1$, extending $\omega_1$ and consisting of classes of the same BBF-degree, and let $I_1$, $J_1$, and $K_1$ be the resulting basis of the purely imaginary quaternions acting via complex structures on the tangent bundle of $M_1$. Define the complex structures $I_2$, $J_2$, $K_2$ on $M_2$ associated to the orthogonal basis of $W_2$, which is the 
image via $\psi_q$ of the chosen basis of $W_1$. Given a point $t$ in the unit sphere in $W_i$ denote by $I_{i,t}$ the corresponding complex structure on $M_i$ such that $g_i(I_{i,t}(\bullet),\bullet)$ is the K\"{a}hler form with class $t$, $i=1,2$. We get the $2$-sphere of complex structures $I_t:=(I_{1,t},I_{2,\psi(t)})$ on $M_1\times M_2$ compatible with the Ricci-flat hermitian metric $(g_1,g_2)$ on $M_1\times M_2$. 
Then the fiber product
\[
\X_1\times_{Q_{(X_1,\eta_1,X_2,\eta_2),\omega_1}}\X_2\rightarrow Q_{(X_1,\eta_1,X_2,\eta_2),\omega_1}
\]
of the two twistor families is the twistor family associated to the above sphere of complex structures on $M_1\times M_2$. 

%
\subsection{Diagonal twistor paths}
\label{sec-diagonal-twistor-paths}
Let 
\[
\Pi_i:\fM_\psi\rightarrow \fM_\Lambda
\] 
be the morphism sending 
the isomorphism class of $(X_1,\eta_1,X_2,\eta_2)$ to that of $(X_i,\eta_i)$. Note that $\Pi_i$ maps
$\widetilde{Q}_{(X_1,\eta_1,X_2,\eta_2),\omega_1}$ isomorphically onto $\widetilde{Q}_{(X_i,\eta_i,\omega_i)}$, $i=1,2$.

\begin{defi}
\label{def-twistor-path}
(1)
A {\em twistor path in} $\fM_\Lambda^0$ from a marked pair $(X,\eta)$ to a  marked pair $(X',\eta')$  consists of the data of a sequence of marked pairs $(X_i,\eta_i)$, $0\leq i\leq k+1$, with $(X_0,\eta_0)=(X,\eta)$ and $(X_{k+1},\eta_{k+1})=(X',\eta')$, together with a twistor line $\widetilde{Q}_{(X_i,\eta_i,\omega_i)}$ through $(X_i,\eta_i)$ passing through $(X_{i+1},\eta_{i+1})$ as well, for $0\leq i \leq k$.
The twistor path is said to be {\em generic} if $\Pic(X_i)$  is trivial, for $1\leq i\leq k$. 
\\
(2)
A {\em twistor path in} $\fM_\psi^0$ from a quadruple $(X_1,\eta_1,X_2,\eta_2)$ to a  quadruple $(Y_1,e_1,Y_2,e_2)$
consists of the data of a sequence  $(X_{1,i},\eta_{1,i},X_{2,i},\eta_{2,i})$, $0\leq i\leq k+1$, with 
$(X_{1,0},\eta_{1,0},X_{2,0},\eta_{2,0})=(X_1,\eta_1,X_2,\eta_2)$ and $(X_{1,k+1},\eta_{1,k+1},X_{2,k+1},\eta_{2,k+1})=(Y_1,e_1,Y_2,e_2)$, together with a twistor line $\widetilde{Q}_{(X_{1,i},\eta_{1,i},X_{2,i},\eta_{2,i}),\omega_i)}$ through 
$(X_{1,i},\eta_{1,i},X_{2,i},\eta_{2,i})$ passing through 
$(X_{1,i+1},\eta_{1,i+1},X_{2,i+1},\eta_{2,i+1})$ as well, for $0\leq i \leq k$.
The twistor path is said to be {\em generic} if $\Pic(X_{1,i})$  is trivial, for $1\leq i\leq k$. 
\end{defi}

\begin{lem}
\label{lemma-Pi-i-is-surjective}
The morphism $\Pi_i:\fM_\psi^0\rightarrow \fM_\Lambda^0$ is surjective, for $i=1,2$.
\end{lem}

\begin{proof}
Let $q:=(X_1,\eta_1,X_2,\eta_2)$ be a quadruple in $\fM_\Lambda^0$ and let $(Y'_1,\eta'_1)$ be a marked pair in $\fM_\Lambda^0$.
Choose a generic K\"{a}hler class $\omega_1$ in $\K_{X_1}\cap \psi_q^{-1}(\K_{X_2})$ so that the twistor line $\widetilde{Q}_{(X_1,\eta_1,\omega_1)}$ is generic and let $(Y_1,e_1)$ be a marked pair in $\widetilde{Q}_{(X_1,\eta_1,\omega_1)}$ with a trivial $\Pic(Y_1)$. A pair $(Y_2,e_2)$ in the fiber of the period map through $\fM_\Lambda^0$ over $\psi(P(Y_1,e_1))$
must be the unique marked pair in this fiber, since $\Pic(Y_2)$ is trivial as well. The equalities
$\K_{Y_i}=\C_{Y_i}^+$, $i=1,2$,  imply that $\K_{Y_1}= (e_1^{-1}\psi e_2)(\K_{Y_2})$ and so the quadruple
$(Y_1,e_1,Y_2,e_2)$ belongs to $\fM_\psi^0$. Furthermore, every twistor path $Q_{(Y_1,e_1,\omega'_1)}$
through $(Y_1,e_1)$ lifts to a twistor path $Q_{(Y_1,e_1,Y_2,e_2),\omega'_1}$ through $(Y_1,e_1,Y_2,e_2)$, by Lemma \ref{lemma-twistor-line-of-quadruples}.
There exists a generic twistor path from $(Y_1,e_1)$ to $(Y'_1,\eta'_1)$, by \cite[Theorems 3.2 and 5.2.e]{verbitsky-cohomology}. Repeating the above argument we get that this twistor path lifts to a twistor path from $(Y_1,e_1,Y_2,e_2)$ to 
$(Y'_1,\eta'_1,Y'_2,\eta'_2)$, for some $(Y'_2,\eta'_2)$ in $\fM^0_\Lambda$,
by the genericity assumption. Hence, $(Y'_1,\eta'_1)$ belongs to the image of $\Pi_1$ and $\Pi_1$ is surjective. 
The isomorphism $\fM_\psi\rightarrow \fM_{\psi^{-1}}$ sending $(X_1,\eta_1,X_2,\eta_2)$ to $(X_2,\eta_2,X_1,\eta_1)$
interchanges the roles of $\Pi_1$ and $\Pi_2$. It follows that $\Pi_2$ is surjective as well.
\end{proof}

\begin{lem}
\label{lemma-twistor-path-connected}
Every two points in $\fM^0_\psi$ are connected by a generic twistor path in $\fM^0_\psi$.
\end{lem}
\begin{proof}
Let $q:=(X_1,\eta_1,X_2,\eta_2)$ be a quadruple in $\fM^0_\psi$, such that $\Pic(X_1)$ is trivial. Then $\Pic(X_2)$ is trivial as well, since $\psi_q:H^2(X_1,\QQ)\rightarrow H^2(X_2,\QQ)$ is a Hodge isometry. Hence, $(X_i,\eta_i)$ is the unique pair  in its connected component in $\fM_\Lambda$ over its period $P(X_i,\eta_i)$, by Verbitski's Torelli theorem \cite{huybrechts-torelli,verbitsky-torelli}. Consequently, $q$ is the unique point in $\fM^0_\psi$ mapping via $P\circ \Pi_i$ to $P(X_i,\eta_i)$, for each of $i=1,2$. It follows that every two such quadruples are connected by a generic twistor path in $\fM^0_\psi$, by the analogous theorem for marked pairs and the liftability of such paths to $\fM^0_\psi$ demonstrated in the proof of Lemma \ref{lemma-Pi-i-is-surjective}. Now, through each point in $\fM^0_\psi$ passes a generic twistor line. Hence, every two points in $\fM^0_\psi$ are connected by a generic twistor path.
\end{proof}

%
\subsection{Deforming vector bundles along diagonal twistor paths}
\label{sec-deforming-vector-bundles-along-diagonal-twistor-paths}

A universal family $f:\X\rightarrow \fM_\Lambda^0$ exists, by \cite{markman-universal-family}. 
We get the universal families $\Pi_i^*(f):\Pi_i^*\X\rightarrow \fM_\psi^0$, $i=1,2$, as well as the universal fiber product
\begin{equation}
\label{eq-universal-fiber-product}
\tilde{f}:\Pi_1^*\X\times_{\fM_\psi^0}\Pi_2^*\X\rightarrow \fM_\psi^0.
\end{equation}
The local system $Rf_*\Integers$ over $\fM_\Lambda^0$ is trivial, by \cite[Lemma 2.1]{markman-universal-family}.
Hence, so are the local system $R\tilde{f}_*\ZZ$ and $R\tilde{f}_*\QQ$. 

\begin{defi}
Let $M$ be a $d$-dimensional compact K\"{a}hler manifold. Set $H^{ev}(M,\QQ):=\oplus_{i=0}^{2d} H^{2i}(M,\QQ)$.
Given a class $\gamma\in H^{ev}(M,\QQ)$ denote by $\gamma_i$ the graded summand in $H^{2i}(M,\QQ)$. We say that $\gamma$ is a {\em Hodge class}, if $\gamma_i$ belongs to $H^{i,i}(M)$, for all $0\leq i\leq 2d$.
\end{defi}

\begin{defi}
\label{def-remains-of-Hodge-type}
Let $q:=(X_1,\eta_1,X_2,\eta_2)$ be a quadruple in $\fM^0_\psi$ and let $\gamma$ be a Hodge class 
in $H^{ev}(X_1\times X_2,\QQ)$. We say that $\gamma$ {\em remains a Hodge class under every deformation in $\fM^0_\psi$}, if for every quadruple $q':=(X'_1,\eta'_1,X'_2,\eta'_2)$ in $\fM^0_\psi$ the parallel-transport $\gamma'\in H^{ev}(X'_1\times X'_2,\QQ)$ of $\gamma$ in the local system $R\tilde{f}_*\QQ$  along some, hence any, continuous path from $q$ to $q'$ in $\fM^0_\psi$ is a Hodge class.
\end{defi}

\begin{defi}
\label{definition-deformation-equivalent-Azumaya-algebras}
Let $M_1$ and $M_2$ be compact K\"{a}hler manifolds. Let $\A_i$ be an Azumaya algebra over $M_i$. We say that $\A_2$ is  {\em deformation equivalent to} $\A_1$, if there exists a proper family $\pi:\M\rightarrow B$ of compact K\"{a}hler manifolds over a connected analytic base $B$, points $b_i\in B$, isomorphisms $f_i:M_i\rightarrow \M_{b_i}$ with the fiber of $\pi$ over $b_i$, and an Azumaya algebra $\A$ over $\M$, such that $f_i^*\A$ is isomorphic to $\A_i$, for $i=1,2$.
\end{defi}

\begin{prop}
\label{prop-slope-stable-vb-with-kappa-2-which-remains-Hodge-class-is-modular}
Let $q:=(X_1,\eta_1,X_2,\eta_2)$ be a quadruple in $\fM^0_\psi$ and let $F$ be a locally free sheaf over $X_1\times X_2$ with the following properties.
\begin{enumerate}
\item
\label{prop-assumption-slope-stable-wrt-every-kahler-class}
$F$ is $\left(\pi_{X_1}^*(\omega)+\pi_{X_2}^*(\psi_q(\omega))\right)$-slope-stable, for every K\"{a}hler class  $\omega$ in some non-empty open subcone $\C_F$ of $\K_{X_1}\cap\psi_q^{-1}(\K_{X_2})$. 
\item
\label{prop-assumption-c-2-remains-a-Hodge-class}
The class $c_2(\SheafEnd(F))$ remains of Hodge type $(2,2)$ 
along every deformation in $\fM^0_\psi$.
\end{enumerate}
For every point $(X'_1,\eta'_1,X'_2,\eta'_2)$ in $\fM^0_\psi$ 
there exists a possibly twisted locally free sheaf F' over $X'_1\times X'_2$, such that the Azumaya algebra $\SheafEnd(F')$ 
is deformation equivalent to $\SheafEnd(F)$.
\end{prop}

\begin{proof}
The slope stability property (\ref{prop-assumption-slope-stable-wrt-every-kahler-class})  of $F$ implies  that the first Chern class of  every direct summand of $\SheafEnd(F)$ vanishes, by \cite[Lemma 7.2]{markman-BBF-class-as-characteristic-class}. 
Property (\ref{prop-assumption-c-2-remains-a-Hodge-class}), that $c_2(\SheafEnd(F))$ remains of Hodge type, implies that the Azumaya algebra $\SheafEnd(F)$ deforms along the twistor family over a generic twistor line 
$\widetilde{Q}_{(X_1,\eta_1,X_2,\eta_2),\omega_1}$
in $\fM^0_\psi$ associated to a K\"{a}hler class $\omega_1$ in $\C_F$, by \cite[Theorem 3.19]{kaledin-verbitsky-book}, which is generalized to Azumaya algebras in \cite[Cor. 6.12]{markman-BBF-class-as-characteristic-class}. Here we use the fact that the twistor line in $\fM^0_\psi$ is associated to a hyper-K\"{a}hler structure, as explained in Section \ref{sec-indeed-a-twistor-line}, as well as the fact that the first Chern class of  every direct summand of $\SheafEnd(F)$ vanishes. Choose a quadruple $q_0:=(Y_1,e_1,Y_2,e_2)$ in $\widetilde{Q}_{(X_1,\eta_1,X_2,\eta_2),\omega_1}$ with a trivial Picard group $\Pic(Y_1)$ and let $\A$ be the Azumaya algebra over $Y_1\times Y_2$ resulting by deforming $\SheafEnd(F)$ along $\widetilde{Q}_{(X_1,\eta_1,X_2,\eta_2),\omega_1}$.
The Azumaya algebra $\A$ further deforms along every generic twistor path in $\fM^0_\psi$ from $q_0$ to any point in $\fM^0_\psi$, by \cite[Prop. 6.17]{markman-BBF-class-as-characteristic-class}. The existence of an Azumaya algebra $\A'$ over $X'_1\times X'_2$ deformation equivalent to $\SheafEnd(F)$ follows from that fact that every point $q'$ in $\fM_\psi^0$ is connected to the point $q_0$ via a generic twistor path, by Lemma \ref{lemma-twistor-path-connected}. The existence of a locally free sheaf $F'$, such that the Azumaya algebras $\SheafEnd(F')$ and $ \A'$ are isomorphic, follows from the well known correspondence between Azumaya algebras and equivalence classes of twisted sheaves modulo tensorization by line bundles \cite{caldararu-thesis}.
\end{proof}

%
\subsection{Deforming the Fourier-Mukai kernel}
\label{sec-deforming-the-FM-kernel}

\begin{defi}
\label{def-quadruple-supports-a-compatible-vector-bundle}
A quadruple $(X_1,\eta_1,X_2,\eta_2)$ in $\fM_\psi$ is said to {\em support a compatible vector bundle} if 
there exists a (possibly twisted) locally free coherent sheaf $F$ over $X_1\times X_2$ with the following properties.
\begin{enumerate}
\item
\label{prop-item-phi-is-degree-reversing}
The class $\kappa(F)\sqrt{td_{X_1\times X_2}}$ induces a morphism
\[
\phi_F:H^*(X_1,\QQ)\rightarrow H^*(X_2,\QQ)
\]
in $\Hom_{\G_{an}}(X_1,X_2)$, where $\G_{an}$ is the groupoid given in (\ref{eq-G-an}). 
\item 
\label{prop-item-phi-restricts-to-a-Hodge-isometry-psi}
The equality $\psi=\eta_2\widetilde{H}_0(\phi_F)\eta_1^{-1}$ holds, where $\widetilde{H}_0$ is given in (\ref{eq-psi-E+}).
\end{enumerate}
\end{defi}

\begin{prop}
\label{prop-deforming-E}
Every quadruple $(X_1,\eta_1,X_2,\eta_2)$ in the connected component $\fM_\psi^0$ given in (\ref{eq-M-psi-0}) supports a compatible vector bundle.
\end{prop}

\begin{proof}
\underline{Step 1:} 
Let $q\in\fM_\psi^0$ be the isomorphism class of $(X_1,\eta_1,X_2,\eta_2)$. 
Consider the local system $R\tilde{f}_*\QQ$ over $\fM^0_\psi$ associated to the universal fiber product 
(\ref{eq-universal-fiber-product}).
Let $\gamma_q\in H^*(X_1\times X_2,\QQ)$ be a flat deformation of the class 
\[
\gamma:=\kappa(E_0)\sqrt{td_{X_0\times Y_0}}
\] 
via a path in $\fM_\psi^0$ from $(X_0,\eta_{X_0},Y_0,\eta_{Y_0})$ to $q$. 
If $\tilde{\eta}_1:H^*(X_1,\QQ)\rightarrow H^*(X_0,\QQ)$  denotes the parallel-transport operator in the local system $\Pi_1^*Rf_*\QQ$ and $\tilde{\eta}_2:H^*(X_2,\QQ)\rightarrow H^*(Y_0,\QQ)$  denotes the parallel-transport operator in the local system $\Pi_2^*Rf_*\QQ$, both with respect to  the chosen path (backward), then 
$\gamma=(\tilde{\eta}_1\otimes\tilde{\eta}_2)(\gamma_q)$.
The class $\gamma_q$ is independent of the path chosen,  since the local system $R\tilde{f}_*\QQ$ over $\fM_\Lambda^0$ is trivial, by \cite[Lemma 2.1]{markman-universal-family}.
We prove in this step that $\gamma_q$ is a Hodge class. In particular, $c_2(\SheafEnd(E_0))$  remains of Hodge type over 
$X_1\times X_2$, by the equality $c_2(\SheafEnd(E_0))=-2\,\rank(E_0)\kappa_2(E_0)$.
Once it is proven that $\gamma_q$ is a  Hodge class, for all $q\in \fM^0_\psi$, Proposition 
\ref{prop-slope-stable-vb-with-kappa-2-which-remains-Hodge-class-is-modular} implies the existence  of a locally free sheaf $F$ over $X_1\times X_2$ such that  $\SheafEnd(F)$ is deformation equivalent to $\SheafEnd(E_0)$.

The isometry\footnote{The isomorphism $\phi_0$ is an isometry with respect to the Mukai pairings. When the odd cohomology of $X$ vanishes $\phi_0$ is also an isometry with respect to the Poincar\'{e} pairings, by Corollary \ref{corollary-kappa-E-equal-kappa-E-dual}, but we do not assume the vanishing of the odd cohomology here.} 
$\phi_0:H^*(X_0,\QQ)\rightarrow H^*(Y_0,\QQ)$, 
induced by $\gamma$,
is invariant under the action of 
$GL(H^*(X_0,\QQ))$ on $\Hom(H^*(X_0,\QQ),H^*(Y_0,\QQ))$, where an element $g\in GL(H^*(X_0,\QQ))$ acts by $h\mapsto (\phi_0 g\phi_0^{-1})h g^{-1}$. The isometry $\phi_0$ is related to the class $\gamma:=\kappa(E_0)\sqrt{td_{X_0\times Y_0}}$ via
\[
\phi_0\in \Hom(H^*(X_0),H^*(Y_0))\cong H^*(X_0)^*\otimes H^*(Y_0)\cong H^*(X_0)\otimes H^*(Y_0) \cong H^*(X_0\times Y_0)\ni\gamma,
\]
where the middle isomorphism is induced by the Poincar\'{e}  duality isomorphism
$PD:H^*(X_0,\QQ)\rightarrow H^*(X_0,\QQ)^*$.  Assume that $g$ is
an isometry with respect to the Poincar\'{e} pairing and so $(g^{-1})^*= PD\circ g \circ PD^{-1}$. In this case $\gamma$ is invariant with respect to 
the action by $g\otimes (\phi_0 g\phi_0^{-1})$. 
The Lie algebra $\bar{\LieAlg{g}}_{X_0,\RealNumbers}$ acts on $H^*(X_0,\RealNumbers)$ by derivations \cite[Prop. 4.5(ii)]{looijenga-lunts}, and  so it annihilates the cup product and is thus contained in the Lie algebra of the isometry group with respect to the Poincar\'{e} pairing. 
We get that the class $\gamma$  is annihilated by $(\xi\otimes 1)+(1\otimes Ad_{\phi_0}(\xi))$, for every element $\xi\in\bar{\LieAlg{g}}_{X_0}$. We conclude that $\gamma$ is annihilated by the  Lie subalgebra
\[
\bar{\LieAlg{g}}_{\phi_0}
\ :=\{(\xi\otimes 1)+(1\otimes Ad_{\phi_0}(\xi)) \ : \ \xi\in\bar{\LieAlg{g}}_{X_0}\} \ 
\subset \ \bar{\LieAlg{g}}_{X_0}\times \LieAlg{g}_{Y_0},
\]
where the inclusion follows from the equality $Ad_{\phi_0}(\LieAlg{g}_{X_0})=\LieAlg{g}_{Y_0}$ established in \cite[Theorem A]{taelman}.
Note that $\bar{\LieAlg{g}}_{\phi_0}$ is in fact the subalgebra  
\[
\bar{\LieAlg{g}}_{\psi_{E_0}}
\ := \
\{(\xi\otimes 1)+(1\otimes Ad_{\psi_{E_0}}(\xi)) \ : \ \xi\in\bar{\LieAlg{g}}_{X_0}\} \ 
\subset \ 
\bar{\LieAlg{g}}_{X_0}\times \bar{\LieAlg{g}}_{Y_0},
\]  
by Lemma \ref{lemma-phi-is-degree-reversing}(\ref{lemma-item-Ad-phi-resricts-to-Ad-psi-E}) and Remark \ref{rem-phi-0-is-degree-reversing}.

We have the commutative diagram
\[
\xymatrix{
H^2(X_0,\QQ) \ar[r]^{\eta_1^{-1}\eta_{X_0}} \ar[d]_{\eta_{Y_0}^{-1}\psi\eta_{X_0}=\psi_{E_0}}&
H^2(X_1,\QQ) \ar[d]^{\psi_q:=\eta_2^{-1}\psi\eta_1}
\\
H^2(Y_0,\QQ) \ar[r]_{\eta_2^{-1}\eta_{Y_0}} & H^2(X_2,\QQ),
}
\]
whose horizontal arrows are parallel-transport operators.
The Lie algebra $\LieAlg{g}_{X_0}$ deforms flatly to $\LieAlg{g}_{X_1}$, by the topological nature of both, and so the semi-simple part of its degree $0$ summand $\bar{\LieAlg{g}}_{X_0}$ deforms flatly to $\bar{\LieAlg{g}}_{X_1}$.
The Lie subalgebra $\LieAlg{g}_{\psi_{E_0}}$ deforms flatly to 
\[
\bar{\LieAlg{g}}_{\psi_q}
\ := \
\{(\xi\otimes 1)+(1\otimes Ad_{\psi_q}(\xi)) \ : \ \xi\in\bar{\LieAlg{g}}_{X_1}\} \ 
\subset \ 
\bar{\LieAlg{g}}_{X_1}\times \bar{\LieAlg{g}}_{X_2},
\]  
by the commutativity of the above diagram. Indeed, 
\begin{eqnarray*}
(Ad_{\eta_1^{-1}\eta_{X_0}}\otimes Ad_{\eta_2^{-1}\eta_{Y_0}})
\left((\xi\otimes 1)+(1\otimes Ad_{\psi_{E_0}}(\xi))\right)
&=&
(Ad_{\eta_1^{-1}\eta_{X_0}}(\xi)\otimes 1)+(1\otimes Ad_{\eta_2^{-1}\eta_{Y_0}\psi_{E_0}}(\xi))
\\
&=&
(\xi'\otimes 1)+(1\otimes Ad_{\psi_q}(\xi')),
\end{eqnarray*}
where $\xi':=Ad_{\eta_1^{-1}\eta_{X_0}}(\xi)\in\bar{\LieAlg{g}}_{X_1}$. Let 
\[
\gamma_q\in H^*(X_1\times X_2,\QQ)
\]
be the flat deformation of the class $\gamma$ via a path in $\fM_\psi^0$.  It follows that $\gamma_q$ is annihilated by $\bar{\LieAlg{g}}_{\psi_q}$, as we have seen that $\gamma$ is annihilated by  $\bar{\LieAlg{g}}_{\psi_{E_0}}$.

Let $h'_{X_i}$ be the Hodge operator of $X_i$, $i=1,2$, given in (\ref{eq-Hodge-operator}). Then $h'_{X_2}=Ad_{\psi_q}(h'_{X_1})$, since $\psi_q$ is an isomorphism of Hodge structures. Hence, the Hodge operator of $X_1\times X_2$
\[
(h'_{X_1}\otimes 1)+(1\otimes h'_{X_2})=(h'_{X_1}\otimes 1)+(1\otimes Ad_{\psi_q}(h'_{X_1}))
\]
belongs to $\bar{\LieAlg{g}}_{\psi_q}$ and thus annihilates $\gamma_q$. We conclude that $\gamma_q$ is a Hodge class.
%

\underline{Step 2:} Consider first the quadruple $q_0:=(X_0,\eta_{X_0},Y_0,\eta_{Y_0})$, used in the definition of the connected component $\fM_\psi^0$ in (\ref{eq-M-psi-0}), and set $F:=E_0$. Property
(\ref{prop-item-phi-is-degree-reversing}) 
in Definition \ref{def-quadruple-supports-a-compatible-vector-bundle} follows in this case 
from Assumption \ref{assumption-quadruple-is-in-fM-psi} and 
property (\ref{prop-item-phi-restricts-to-a-Hodge-isometry-psi}) follows from the definition of $\psi$ in Equation (\ref{eq-psi}).

Let $\phi_{\gamma_q}:H^*(X_1,\QQ)\rightarrow H^*(X_2,\QQ)$ be the homomorphism induced by $\gamma_q$. The equality
$\phi_{\gamma_q}=\tilde{\eta}_2^{-1}\phi_0\tilde{\eta}_1$ exhibits $\phi_{\gamma_q}$ as a composition of morphisms in the groupoid $\G$, hence in $\Hom_\G(X_1,X_2)$. 
Choose a twistor path $C$ in $\fM^0_\psi$ from $q_0$ to $q$ along which the Azumaya algebra $\SheafEnd(E_0)$ deforms via
an Azumaya algebra $\A$ over the twistor family $\Pi_1^*\X\times_C\Pi_2^*\X\rightarrow C$. 
Such a path $C$ exists, by Assumption \ref{assumption-quadruple-is-in-fM-psi} and 
Proposition  \ref{prop-slope-stable-vb-with-kappa-2-which-remains-Hodge-class-is-modular}.
There exists a locally  free twisted sheaf 
$\F$ over $\Pi_1^*\X\times_C\Pi_2^*\X$, such that $\SheafEnd(\F)$ is isomorphic to $\A$. 
$\F$ restricts to a locally free twisted sheaf $F_0$ over $X_0\times Y_0$ and a locally free sheaf $F$ over $X_1\times X_2$, such that 
the Azumaya  algebras $\SheafEnd(F_0)$ is isomorphic to $\SheafEnd(E_0)$ and
$\kappa(F)$ is a parallel-transport of $\kappa(F_0)$ along a (real) path in $C$.  Hence, $\gamma_q$ is equal to $\kappa(F)\sqrt{td_{X_1\times X_2}}$ and property (\ref{prop-item-phi-is-degree-reversing}) 
in Definition \ref{def-quadruple-supports-a-compatible-vector-bundle} follows. Property  (\ref{prop-item-phi-restricts-to-a-Hodge-isometry-psi}) in Definition \ref{def-quadruple-supports-a-compatible-vector-bundle} is equivalent to the equality $\psi_q=\widetilde{H}_0([\kappa(F)\sqrt{td_{X_1\times X_2}}]_*)$. The latter equality 
follows at $q$  via Remark \ref{remark-convention-behaves-well-with-respect-to-parallel-transport-operators} from the fact that it holds for $q_0$. Explicitly,
\begin{eqnarray*}
\psi_q&=&(\eta_2^{-1}\eta_{Y_0})\psi_{E_0}(\eta_{X_0}^{-1}\eta_1)=
(\eta_2^{-1}\eta_{Y_0})\widetilde{H}_0([\kappa(E_0)\sqrt{td_{X_0\times Y_0}}]_*)(\eta_{X_0}^{-1}\eta_1)
\\
&\stackrel{{\rm Rem.} \  \ref{remark-convention-behaves-well-with-respect-to-parallel-transport-operators}}{=}&
\widetilde{H}_0\left((\tilde{\eta}_2^{-1})\circ [\kappa(E_0)\sqrt{td_{X_0\times Y_0}}]_*\circ \tilde{\eta}_1\right)=
\widetilde{H}_0([\kappa(F)\sqrt{td_{X_1\times X_2}}]_*).
\end{eqnarray*}
\end{proof}

%
\section{The BKR equivalence}
\label{sec-BKR}
Let $S$ be a projective surface and denote by $S^{[n]}$  the Hilbert scheme of length $n$ subscheme of $S$. 
The symmetric group $\fS_n$ on $n$ letters acts on the cartesian product $S^n$ permuting the factors. 
Let $Hilb_{\fS_n}(S^n)$ be the Hilbert scheme of length $n!$ subschemes $Z$ of $S^n$, which are $\fS_n$-invariant, and such that $H^0(Z,\StructureSheaf{Z})$ is the regular representation of $\fS_n$. Then 
$S^{[n]}$ is isomorphic to the closure in $Hilb_{\fS_n}(S^n)$ of the locus of reduced subschemes which are $\fS_n$-orbits of 
$n$-tuples consisting of $n$ distinct points of $S$, by \cite{haiman}. The restriction
\[
\Gamma_S\subset S^{[n]}\times S^n,
\]
to $S^{[n]}$ of the universal subscheme of $Hilb_{\fS_n}(S^n)\times S^n$ is called\footnote{
The isospectral Hilbert scheme $\Gamma_S$ is defined in \cite{haiman} as the reduced subscheme associated to the support of the fiber product of $S^{[n]}$ and $S^n$ over the symmetric product $S^{(n)}$. Haiman proves that $\Gamma_S$ is flat over $S^{[n]}$, and so it is indeed the universal subscheme.
}
the {\em isospectral Hilbert scheme}.
The projections $S^{[n]}\LeftArrowOf{q} \Gamma_S\RightArrowOf{b} S^n$ have the following properties.
The morphism $q$ is flat and $\fS_n$-invariant of degree $n!$, by the proof of
\cite[Prop. 3.7.4]{haiman}. The morphism $b$ restricts to an isomorphism over the complement of the union of the diagonals in $S^n$. The isospectral Hilbert scheme $\Gamma_S$ is Gorenstein, by \cite[Theorem 3.1]{haiman},  and its dualizing sheaf $\omega_{\Gamma_S}$ is isomorphic to the line bundle $q^*(\omega_{S^{[n]}}(\delta))$, by 
\cite[Theorem 3.1, Prop. 3.4.3, comment after Lemma 3.4.2]{haiman}. Above $\delta$ is the divisor class, such that 
$\StructureSheaf{S^{[n]}}(-\delta)\cong \left(q_*(\StructureSheaf{\Gamma_S})\right)^\chi$, where $\chi$ is the sign character of $\fS_n$. 
Note that $2\delta$ is the class of the effective divisor in $S^{[n]}$ of non-reduced subschemes. 

Let $D^b_{\fS_n}(S^n)$ be the bounded derived category of $\fS_n$-equivariant coherent sheaves on $S^n$
and $D^b_{\fS_n}(S^{[n]}\times S^n)$ its analogue with respect to the permutation action of $\fS_n$ on the factor $S^n$ and the trivial action on the factor $S^{[n]}$. 
Let $\rho$ be the natural $\fS_n$-linearization of $\StructureSheaf{\Gamma_S}$. The object 
$(\StructureSheaf{\Gamma_S},\rho)$ of $D^b_{\fS_n}(S^{[n]}\times S^n)$ is the Fourier-Mukai kernel of an equivalence of derived categories
\[
BKR : D^b_{\fS_n}(S^n)\rightarrow D^b(S^{[n]}),
\]
given by the composition $Rq_*^{\fS_n}Lb^*$ of the functors $Lb^*:D^b_{\fS_n}(S^n)\rightarrow D^b_{\fS_n}(\Gamma_S)$ and
$Rq_*^{\fS_n}:D^b_{\fS_n}(\Gamma_S)\rightarrow D^b(S^{[n]})$, by \cite{BKR}.

%
\section{A universal bundle over $M^{[n]}\times S^{[n]}$ from a universal bundle over $M\times S$}
\label{sec-universal-bundle-over-product-of-Hilbert-schemes}
In Section \ref{subsection-a-locally-free-FM-kernel} we recall the construction of a locally free Fourier-Mukai kernel of an equivalence of the derived categories of the Hilbert schemes $S^{[n]}$ and $M^{[n]}$, where $M$ is a $2$-dimensional smooth and projective moduli space of vector bundles on a $K3$ surface $S$.
In Section \ref{sec-functor-widetilde-Theta_n} we construct a functor $\widetilde{\Theta}_n:\G^{[1]}\rightarrow\G^{[n]}$,
which sends a morphism associated to an equivalence $\Phi$ of derived categories of $K3$ surfaces to a morphism associated to the equivalence of the derived categories of their Hilbert schemes, which is the BKR-conjugate of the cartesian power $\Phi^n$.
In Section \ref{sec-functor-Theta_n} we normalize $\widetilde{\Theta}_n$ to obtain the functor $\Theta_n:\G^{[1]}\rightarrow\G^{[n]}$, mentioned in the introduction in (\ref{eq-introduction-functor-Theta-n}), which maps $\G^{[1]}_{an}$ to $\G^{[n]}_{an}$.
The functor $\Theta_n$ sends a parallel-transport operator between two $K3$ surfaces $S_1$ and $S_2$ to the associated parallel-transport operator between their Hilbert schemes $S_1^{[n]}$ and $S_2^{[n]}$. If $\phi\in\Hom_{\G_{an}^{[1]}}(S_1,S_2)$ is induced by the class $\kappa(\U)\sqrt{td_{S_1\times S_2}}$ 
of the Fourier-Mukai kernel $\U$ of an equivalence $\Phi_\U:D^b(S_1)\rightarrow D^b(S_2)$, then $\Theta_n(\phi)$
is induced by the analogous class associated to the BKR-conjugate of the $n$-th cartesian power of $\Phi_\U$. 
In Corollary \ref{cor-Theta_n-maps-G^[1]_an-to-G^[n]_an} we show that $\widetilde{H}(\Theta_n(\phi)):\widetilde{H}(S_1^{[n]},\QQ)\rightarrow \widetilde{H}(S_2^{[n]},\QQ)$
is the naive extension of $\widetilde{H}(\phi):\widetilde{H}(S_1,\QQ)\rightarrow \widetilde{H}(S_2,\QQ)$.

%
\subsection{A universal bundle over $M^{[n]}\times S^{[n]}$}
\label{subsection-a-locally-free-FM-kernel}
Assume that $S$ is a $K3$ surface,  $v=(r,\lambda,s)\in \widetilde{H}(S,\Integers)$ is a primitive and isotropic Mukai vector of rank $r\geq 2$, and $H$ is a $v$-generic polarization. Then every $H$-slope-semistable sheaf on $S$ is $H$-slope-stable and locally free. Let $M$ be the moduli space of $H$-slope-stable sheaves on $S$ with Mukai vector $v$.
Then $M$ is a $K3$ surface, by \cite{mukai-hodge}.
Assume that a universal sheaf $\U$ exists over $M\times S$. Then $\U$ is the Fourier-Mukai kernel of an equivalence
\[
\Phi_\U:D^b(M)\rightarrow D^b(S)
\]
(see \cite[Prop. 10.25]{huybrechts-FM}).
We get the equivalence
\[
\Phi_\U^{\boxtimes n}: D^b_{\fS_n}(M^n)\rightarrow D^b_{\fS_n}(S^n)
\]
and its conjugate 
\begin{equation}
\label{eq-Phi-U-[n]}
\Phi_\U^{[n]}:=BKR\circ \Phi_\U^{\boxtimes n} \circ BKR^{-1}:D^b(M^{[n]})\rightarrow D^b(S^{[n]}).
\end{equation}

Let $\fS_{n,\Delta}$ be the diagonal subgroup of $\fS_n\times \fS_n$. Let $\U^{\boxtimes n}:=\U\boxtimes \cdots\boxtimes \U$
be the exterior tensor power of $\U$ over $M^n\times S^n$. Then $\U^{\boxtimes n}$ admits the permutation $\fS_{n,\Delta}$-linearization $\rho_\boxtimes$. Denote by $\chi$ the sign character of $\fS_{n,\Delta}$. We get the object
$(\U^{\boxtimes n},\rho_\boxtimes\otimes \chi)$ in $D^b_{\fS_{n,\Delta}}(M^n\times S^n)$.

We use the same notation for the morphisms $M^n\LeftArrowOf{b}\Gamma_M\RightArrowOf{q} M^{[n]}$ from the isospectral Hilbert scheme of $M$. We get the product morphisms
\[
M^n\times S^n\LongLeftArrowOf{b\times b}\Gamma_M\times\Gamma_S\LongRightArrowOf{q\times q} M^{[n]}\times S^{[n]}.
\]

\begin{lem}
The equivalence $\Phi_\U^{[n]}$
has the rank $n!r^n$ locally free Fourier-Mukai kernel
\begin{equation}
\label{eq-universal-bundle-over-product-of-Hilbert-schemes}
E:=
\pi_{M^{[n]}}^*(\StructureSheaf{M^{[n]}}(\delta))\otimes(q\times q)_*^{\fS_{n,\Delta}}((b\times b)^*(\U^{\boxtimes n},\rho_\boxtimes\otimes \chi)).
\end{equation}
\end{lem}

\begin{proof}
The statement is proved in \cite[Sec. 11.1.2, Eq. (11.6)]{markman-modular}
\end{proof}

%
\subsection{The functor $\widetilde{\Theta}_n:\G^{[1]}\rightarrow\G^{[n]}$}
\label{sec-functor-widetilde-Theta_n}
In the remainder of Section \ref{sec-universal-bundle-over-product-of-Hilbert-schemes} we will keep the assumption that the rank of the object $\U\in D^b(S\times M)$ is non-zero, but drop the assumption that $\U$ is represented by a coherent sheaf.
Let $\G^{[n]}$ be the subgroupoid of the groupoid $\G$ given in (\ref{eq-Groupoid-G}) whose objects are of $K3^{[n]}$-type.
We define next the functor $\widetilde{\Theta}_n:\G^{[1]}\rightarrow \G^{[n]}$.
Define $\widetilde{\Theta}_n(X,\epsilon)=(X^{[n]},\epsilon^{[n]})$,
where the orientation $\epsilon^{[n]}$ of $H^2(X^{[n]},\QQ)$ is such that the class $\delta$ extends a basis compatible with the orientation $\epsilon$ to one compatible with $\epsilon^{[n]}$. 

(1) $\widetilde{\Theta}_n$ maps a parallel-transport operator, associated to two fibers over two points $b_0$ and $b_1$  in the base of a family $\X\rightarrow B$ of $K3$ surfaces and a path $\gamma$ from $b_0$ to $b_1$, to the parallel-transport operator associated to the same path and the relative Hilbert scheme $\X^{[n]}\rightarrow B$. 
Denote by 
\begin{equation}
\label{eq-theta}
\theta:H^2(S,\Integers)\rightarrow H^2(S^{[n]},\Integers)
\end{equation}
the composition of the isomorphism
$H^2(S,\Integers)\cong H^2(S^{(n)},\Integers)$ with the pullback $H^2(S^{(n)},\Integers)\rightarrow H^2(S^{[n]},\Integers)$
via the Hilbert-Chow morphism. 

(2) $\widetilde{\Theta}_n$ sends cup product with $\exp(\lambda)$, $\lambda\in H^2(X,\QQ)$, to cup product with $\exp(\theta(\lambda))$.

(3) $\widetilde{\Theta}_n$ sends the isometry $\phi_\P:H^*(X,\QQ)\rightarrow H^*(Y,\QQ)$ induced by the correspondence
$\sqrt{td_{X\times Y}}ch(\P)$, where $\P\in D^b(X\times Y)$ is the Fourier-Mukai kernel of an equivalence $\Phi_\P:D^b(X)\rightarrow D^b(Y)$, to the isometry 
induced by the correspondence
$\sqrt{td_{X^{[n]}\times Y^{[n]}}}ch(\P^{[n]})$, where $\P^{[n]}$ is the Fourier-Mukai kernel in (\ref{eq-universal-bundle-over-product-of-Hilbert-schemes}) of the BKR-conjugate of
$\Phi_\P^{\boxtimes n}$ as in (\ref{eq-Phi-U-[n]}). 

The image of morphisms under the functor $\widetilde{\Theta}_n$ was defined above separately for the three types of morphisms. Let $[\StructureSheaf{\Gamma_S}]$ be the class of $\StructureSheaf{\Gamma_S}$, with its natural linearization, in the topological $\fS_n$-equivariant K-ring $K_{\fS_n}(S^n\times S^{[n]})$. We get the correspondence
$[\StructureSheaf{\Gamma_S}]_*:K_{\fS_n}(S^n)\rightarrow K(S^{[n]})$ on the level of topological K-rings. 
The inverse is induced by a class in $K_{\fS_n}(S^n\times S^{[n]})$ as well.
The value of $\widetilde{\Theta}_n$ on a morphism in $\G^{[1]}$ of type (3) associated to an equivalence $\Phi_\U$ with Fourier-Mukai
kernel $\U$ involves first composing the correspondence 
$[\U^{\boxtimes n},\rho_\boxtimes]_*: K_{\fS_n}(M^n)\rightarrow K_{\fS_n}(S^n)$ with $[\StructureSheaf{\Gamma_S}]_*$
and the inverse of $[\StructureSheaf{\Gamma_M}]_*$, to obtain the correspondence from $K(M^{[n]})$ to $K(S^{[n]})$, and then using the Chern character isomorphism and its inverse to obtain the cohomological morphism in $\Hom_{\G^{[n]}}(M^{[n]},S^{[n]})$ (see \cite[Sec. 10.1]{BKR}). 
The value of $\widetilde{\Theta}_n$ on morphisms of types (1) and (2) may be obtained by first conjugating via the Chern character their cartesian powers to lift these morphisms to correspondences from $K_{\fS_n}(M^n)$ to $K_{\fS_n}(S^n)$.
The check that $\widetilde{\Theta}_n$ maps compositions of morphisms to compositions reduces to showing that its
value on (1) morphisms that are parallel-transport operators 
or (2) morphisms that are multiplication by $\exp(\lambda)$, $\lambda\in H^2(X,\QQ)$, 
is induced by composing the correspondences from $K_{\fS_n}(M^n)$ to $K_{\fS_n}(S^n)$
with the same correspondences, $[\StructureSheaf{\Gamma_S}]_*$ and the inverse of $[\StructureSheaf{\Gamma_M}]_*$, used for morphisms of type (3). 
For parallel-transport operators this is the case, since
the 
class $[\StructureSheaf{\Gamma_S}]$ in $K_{\fS_n}(S^n\times S^{[n]})$ 
is a global class in the corresponding local system of K-rings over the base of every family of $K3$ surfaces, hence the correspondences $[\StructureSheaf{\Gamma_S}]$ and $[\StructureSheaf{\Gamma_M}]$ intertwine parallel-transport operators associated to paths in the base of such families. 
The check for multiplication by $\exp(\lambda)$, $\lambda\in H^2(X,\QQ)$, reduces to checking for integral classes, as the correspondences are all linear. In this case the action corresponds to tensorization by topological line bundles on the level of K-rings. The correspondence  $[\StructureSheaf{\Gamma_S}]_*$ intertwines tensorization in $K_{\fS_n}(S^n)$
by the cartesian power $L^{\boxtimes n}$ of a line bundle on $S$ and by the line bundle $\tilde{L}$ on $S^{[n]}$ with $c_1(\tilde{L})=\theta(c_1(L))$, 
since $\theta(c_1(L))$ is defined as the pullback of the class associated to $c_1(L)$ in the symmetric product $S^{(n)}$, 
and the isospectral Hilbert scheme is a reduced fiber product of $S^n$ and $S^{[n]}$ over $S^{(n)}$. Hence, the projection formula yields that $[\StructureSheaf{\Gamma_S}]_*$ indeed intertwines the two tensorizations.

%
\subsection{The functor $\Theta_n:\G^{[1]}\rightarrow\G^{[n]}$}
\label{sec-functor-Theta_n}

Let $\phi:\widetilde{H}(M,\Integers)\rightarrow \widetilde{H}(S,\Integers)$ be the isometry induced by the correspondence
$ch(\U)\sqrt{td_{M\times S}}$.
The morphism $\widetilde{\Theta}_n(\phi):H^*(M^{[n]},\QQ)\rightarrow H^*(S^{[n]},\QQ)$ is induced by the correspondence
\[
ch(E)\sqrt{td_{M^{[n]}\times S^{[n]}}},
\]
where $E$ is given in terms of $\U$ in Equation (\ref{eq-universal-bundle-over-product-of-Hilbert-schemes}).
We recall next the relationship between $\phi$ and 
$\widetilde{H}(\widetilde{\Theta}_n(\phi)):\widetilde{H}(M^{[n]},\QQ)\rightarrow \widetilde{H}(S^{[n]},\QQ)$.
Let 
\begin{equation}
\label{eq-tilde-theta}
\tilde{\theta}:\widetilde{H}(S,\QQ)\rightarrow \widetilde{H}(S^{[n]},\QQ)
\end{equation} 
be the extension of $\theta:H^2(S,\Integers)\rightarrow H^2(S^{[n]},\Integers)$, given in (\ref{eq-theta}), mapping $\alpha$ to $\alpha$ and $\beta$ to $\beta$. Let
\begin{equation}
\label{eq-tilde-iota}
\tilde{\iota}:O(\widetilde{H}(S,\QQ))\rightarrow O(\widetilde{H}(S^{[n]},\QQ))
\end{equation}
be the embedding given by $\tilde{\iota}_g(\tilde{\theta}(x))=\tilde{\theta}(g(x))$ and $\tilde{\iota}_g$ acts as the identity on the $1$-dimensional subspace orthogonal to the image of $\tilde{\theta}$. Given an isometry $g:\widetilde{H}(M,\QQ)\rightarrow \widetilde{H}(S,\QQ)$ we denote by $\tilde{\iota}_g:\widetilde{H}(M^{[n]},\QQ)\rightarrow \widetilde{H}(S^{[n]},\QQ)$ the isometry satisfying $\tilde{\iota}_g(\tilde{\theta}(x))=\tilde{\theta}(g(x))$ and mapping $\delta\in H^2(M^{[n]},\QQ)$ to $\delta\in H^2(S^{[n]},\QQ)$.

Choose markings $\eta_S:H^2(S,\Integers)\rightarrow \Lambda_{K3}$ and $\eta_M:H^2(M,\Integers)\rightarrow \Lambda_{K3}$, where $\Lambda_{K3}$ is the $K3$ lattice, and let $\Lambda$ be the orthogonal direct sum 
\begin{equation}
\label{eq-K3-n-lattice}
\Lambda_{K3}\oplus \Integers\delta_0, 
\end{equation}
where $(\delta_0,\delta_0)=2-2n$, $n\geq 2$. 
Let 
\begin{equation}
\label{eq-iota}
\iota:O(\Lambda_{K3})\rightarrow O(\Lambda)
\end{equation} 
be the embedding, where $\iota_g$ is the extension of $g$ satisfying $\iota_g(\delta_0)=\delta_0$.
Choose an orientation of $\Lambda_{K3}$ and an orientation of $\Lambda$, so that $\delta_0$ extends a basis compatible with the orientation of the former to one compatible with that of  the latter.
Set $\det(\phi):=\det(\eta_S\circ\phi\circ\eta_M^{-1})$. 
Extend the marking $\eta_S$ to a marking 
\begin{equation}
\label{eq-eta-S-[n]}
\eta_{S^{[n]}}:H^2(S^{[n]},\Integers) \rightarrow \Lambda
\end{equation} 
sending  $\delta$ to $\delta_0$, where $\delta$ is half the class of the divisor of non-reduced subschemes. Define $\eta_{M^{[n]}}$ similarly. 
Define $\widetilde{H}(\widetilde{\Theta}_n(\phi))$ using Convention \ref{convention-on-orientations-and-the-groupoid-G}.

\begin{lem}
\label{lemma-composition-of-tilde-H-with-BKR-conjugate-of-powers-of-FM}
$
\widetilde{H}(\widetilde{\Theta}_n(\phi))=\det(\phi)^{n+1}(B_{-\delta/2}\circ \tilde{\iota}_\phi\circ B_{\delta/2}).
$
\end{lem}

\begin{proof}
Let $\zeta:H^*(M,\QQ)\rightarrow H^*(S,\QQ)$ be a parallel-transport operator such that $\det(\eta_S\circ\bar{\zeta}\circ\eta_M^{-1})=1,$ where $\bar{\zeta}$ is the restriction of $\zeta$ to $H^2(M,\QQ)$.
Then $\zeta$ lifts to a parallel-transport operator $\zeta^{[n]}:H^*(M^{[n]},\QQ)\rightarrow H^*(S^{[n]},\QQ)$, since associated to a family of $K3$ surfaces is a canonical family of relative Douady spaces. Let 
$\bar{\zeta}^{[n]}$ be the restriction of $\zeta^{[n]}$ to $H^2(M^{[n]},\QQ)$.
Then
$\det(\eta_{S^{[n]}}\circ\bar{\zeta}^{[n]}\circ\eta_{M^{[n]}}^{-1})=1$. We get
$\widetilde{H}(\zeta^{[n]})=\tilde{\iota}_{\zeta}=B_{-\delta/2}\circ \tilde{\iota}_\zeta\circ B_{\delta/2}$, where the first equality is by 
Remark \ref{remark-convention-behaves-well-with-respect-to-parallel-transport-operators}. 
The composition $\phi_1:=\zeta\circ \phi$ is an element of 
$\Aut_{\G^{[1]}}(S)$.
We conclude that the equality in the statement of the lemma is equivalent to the equality
\[
\widetilde{H}(\widetilde{\Theta}_n(\phi_1))=\det(\phi_1)^{n+1}(B_{-\delta/2}\circ \tilde{\iota}_{\phi_1}\circ B_{\delta/2}).
\]
The latter equality holds for all element in 
$\Aut_{\G^{[1]}}(S)$ by 
\cite[Theorem 12.2]{markman-modular} (see also \cite[Theorem 7.4]{beckmann} and in case $n=2$ Taelman's earlier result  \cite[Theorem 9.4]{taelman}).
\end{proof}

Let $\varphi:\widetilde{H}(M,\QQ)\rightarrow \widetilde{H}(S,\QQ)$ be the isometry induced by 
$\kappa(\U)\sqrt{td_{M\times S}}$.
Let 
\[
\varphi^{[n]}:H^*(M^{[n]},\QQ)\rightarrow H^*(S^{[n]},\QQ)
\] 
be the isometry induced by the correspondence
$
\kappa(E)\sqrt{td_{M^{[n]}\times S^{[n]}}}.
$
Again, $\varphi^{[n]}$ is a morphism in $\Hom_{\G^{[n]}}(M^{[n]},S^{[n]})$.
As above, we denote by $\widetilde{H}(\varphi)_0$ the restriction of $\varphi$ to $H^2(M,\QQ)$  and set 
$\psi_\U:=\widetilde{H}_0(\varphi)$, where $\widetilde{H}_0(\varphi):=\nu(\widetilde{H}(\varphi)_0)\widetilde{H}(\varphi)_0$ as in (\ref{eq-psi-E+}).
We define $\widetilde{H}(\varphi^{[n]})_0$ as the restriction of $\widetilde{H}(\varphi^{[n]})$ to $H^2(M^{[n]},\QQ)$ and set 
$\psi_E:=\widetilde{H}_0(\varphi^{[n]})$, where $\widetilde{H}_0(\varphi^{[n]}):=\nu(\widetilde{H}(\varphi^{[n]})_0)\widetilde{H}(\varphi^{[n]})_0$.

\begin{cor}
\label{cor-psi-E-restricts-to-psi-U}
\begin{equation}
\label{eq-tilde-H-of-varphi-[n]}
\widetilde{H}(\varphi^{[n]})=\det(\varphi)^{n+1}\tilde{\iota}_\varphi.
\end{equation}
Consequently, $\psi_E\circ\theta=\theta\circ\psi_\U$ and $\psi_E(\delta)=\delta$.
\end{cor}

\begin{proof}
Let $a_1\in H^2(M,\Integers)$ and  $a_2\in H^2(S,\Integers)$ be the classes satisfying $c_1(\U)=\pi_M^*(a_1)+\pi_S^*(a_2)$.
Let $e_1\in H^2(M^{[n]},\Integers)$ and $e_2\in H^2(S^{[n]},\Integers)$ be the classes satisfying
$c_1(E)=\pi_{M^{[n]}}^*(e_1)+\pi_{S^{[n]}}^*(e_2)$.
Set $R:=n!r^n$. We have
\begin{eqnarray*}
\widetilde{H}(\varphi^{[n]})
&=&
B_{-e_2/R}\circ \widetilde{H}(\widetilde{\Theta}_n(\phi))\circ B_{-e_1/R}
\\ &=&
\det(\phi)^{n+1}B_{-(\delta/2+e_2/R)}\circ \iota_\phi\circ B_{\delta/2-e_1/R}
\\ &=&
\det(\phi)^{n+1}B_{(\theta(a_2)/r-\delta/2-e_2/R)}\circ \iota_\varphi\circ B_{(\theta(a_1)/r+\delta/2-e_1/R)}
\\
&=& \det(\phi)^{n+1}B_{\lambda_2}\circ \iota_\varphi\circ B_{\lambda_1},
\end{eqnarray*}
where
$\lambda_1=\theta(a_1)/r+\delta/2-e_1/R$ and $\lambda_2=\theta(a_2)/r-\delta/2-e_2/R$.
The first equality follows from the definition of $\kappa(E)$, the second from Lemma \ref{lemma-composition-of-tilde-H-with-BKR-conjugate-of-powers-of-FM}, the third from the definition of $\kappa(\U)$, and the fourth is obvious.
Note that $\det(\phi)=\det(\varphi)$. In order to establish Equation (\ref{eq-tilde-H-of-varphi-[n]}) it remains to prove  that $\lambda_i=0$, $i=1,2$, or equivalently, 
\begin{eqnarray}
\label{eq-e-1}
e_1&=&n!r^n\left(\frac{\theta(a_1)}{r}+\frac{\delta}{2}\right),
\\
\label{eq-e-2}
e_2&=&n!r^n\left(\frac{\theta(a_2)}{r}-\frac{\delta}{2}\right).
\end{eqnarray}
Both $\varphi$ and $\varphi^{[n]}$ are degree reversing, by Lemma \ref{lemma-phi-is-degree-reversing}. Hence,
both $\widetilde{H}(\varphi^{[n]})$ and $\iota_\varphi$ are degree reversing, the first by Lemma \ref{lemma-preservation-of-grading-of-LLV-lattice-implies-that-of-cohomology}. 
If $\lambda_1\neq 0$, then $(\lambda_1,\lambda'_1)\neq 0$, for some $\lambda_1'\in H^2(M^{[n]},\Integers)$ and the coefficient of $\beta$ in $B_{\lambda_1}(\lambda'_1)$ is non-zero, which implies that the coefficient of $\alpha$ in
$\iota_{\varphi}(B_{\lambda_1}(\lambda'_1))$ is non-zero, and hence so is the coefficient of $\alpha$ in 
$\widetilde{H}(\varphi^{[n]})(\lambda_1')$. This contradicts the fact that $\widetilde{H}(\varphi^{[n]})$ maps
$H^2(M^{[n]},\QQ)$ to $H^2(S^{[n]},\QQ)$. Hence, $\lambda_1=0$.

If $\lambda_2\neq 0$, then $(\lambda_2,\lambda'_2)\neq 0$, for some $\lambda_2'\in H^2(S^{[n]},\Integers)$ and the coefficient of $\beta$ in $B_{\lambda_2}(\lambda'_2)$ is non-zero. There exists a class $\lambda_1'\in H^2(M^{[n]},\QQ)$, such that 
$\iota_\varphi(\lambda_1')=\lambda_2'$, as $\iota_\varphi$ is an isometry.
Hence, the coefficient of $\beta$ in $\widetilde{H}(\varphi^{[n]})(\lambda_1')$ does not vanish. A contradiction. We conclude that $\lambda_2=0$.

Equality (\ref{eq-tilde-H-of-varphi-[n]}) implies that $\nu(\widetilde{H}(\varphi^{[n]})_0)=\det(\varphi)^{n+1}\nu(\widetilde{H}(\varphi)_0)$. Hence $\psi_E$ is the extension of $\psi_\U$ mapping $\delta\in H^2(M^{[n]},\QQ)$ to $\delta\in H^2(S^{[n]},\QQ)$.
\end{proof}

Let $\G^{[n]}$ be the subgroupoid of the groupoid $\G$ given in (\ref{eq-Groupoid-G}) whose objects are of $K3^{[n]}$-type.
Let 
\begin{equation}
\label{eq-functor-Theta-n}
\Theta_n:\G^{[1]}\rightarrow \G^{[n]}
\end{equation}
be the following normalization of the functor $\widetilde{\Theta}_n$
defined in Section \ref{sec-functor-widetilde-Theta_n}.
The functor $\Theta_n$ sends an object of $\G^{[1]}$ to the same object that $\widetilde{\Theta}_n$ does.
Given a morphism $\phi\in \Hom_{\G^{[1]}}((X,\epsilon),(Y,\epsilon'))$, we define
\[
\Theta_n(\phi)=B_{\delta/2}\circ\widetilde{\Theta}_n(\phi)\circ B_{-\delta/2},
\]
where we denote by $\delta$ half the class of the divisor of non-reduced subschemes in both $X^{[n]}$ and $Y^{[n]}$ and
$B_{\lambda}$ is cup product with $\exp(\lambda)$.
Note that if $\phi$ is $B_\lambda\in\Aut(X,\epsilon)$, $\lambda\in H^2(X,\QQ)$, or if $\phi$ is a parallel-transport operator, then $\Theta_n(\phi)=\widetilde{\Theta}_n(\phi)$.

\begin{cor}
\label{cor-Theta_n-maps-G^[1]_an-to-G^[n]_an}
The functor $\Theta_n$ maps $\G^{[1]}_{an}$ to $\G^{[n]}_{an}$. Furthermore,
$\Theta_n(\varphi)=\varphi^{[n]},$ where $\varphi$ and $\varphi^{[n]}$ are the morphisms in Corollary \ref{cor-psi-E-restricts-to-psi-U}.
\end{cor}

\begin{proof}
The statement is clear for parallel-transport operators. Hence, it sufices to prove the equality $\Theta_n(\varphi)=\varphi^{[n]}$.
Keep the notation of the proof of Corollary \ref{cor-psi-E-restricts-to-psi-U}.
The second equality below follows from Equations (\ref{eq-e-1}) and (\ref{eq-e-2}).
\begin{eqnarray*}
\varphi^{[n]}&=&B_{-(e_2/R)}\circ \widetilde{\Theta}_n(\phi)\circ B_{-(e_1/R)}=
B_{-(\theta(a_2/r))}\circ B_{\delta/2}\circ \widetilde{\Theta}_n(\phi)\circ B_{-\delta/2}\circ B_{-(\theta(a_1)/r)}
\\
&=&
\Theta_n(B_{-a_2/r})\circ \Theta_n(\phi)\circ \Theta_n(B_{-a_1/r})
=\Theta_n(B_{-a_2/r}\circ\phi\circ B_{-a_1/r})=\Theta_n(\varphi).
\end{eqnarray*}
The equality $\Theta_n(\varphi)=\varphi^{[n]}$ follows.
\end{proof}

%
\section{Deforming the universal vector bundle over $M^{[n]}\times S^{[n]}$}
\label{sec-stability-of-the-universal-bundle-over-product-of-hilbert-schemes}
We consider in this section a sequence of examples of $2$-dimensional moduli spaces $M$ of stable vector bundles over a $K3$ surface $S$,
such that the vector bundle $E$ over $M^{[n]}\times S^{[n]}$, given in (\ref{eq-universal-bundle-over-product-of-Hilbert-schemes})
as the Fourier-Mukai kernel of an equivalence of derived categories, 
can be deformed with $M^{[n]}\times S^{[n]}$ over the whole connected component $\fM_\psi^0$ of a moduli space of rational Hodge isometries introduced in Section \ref{sec-moduli-of-rational-Hodge-isometries}.

Keep the notation of Section \ref{sec-universal-bundle-over-product-of-Hilbert-schemes}.
Assume that $\Pic(S)$ is cyclic, generated by the class $h:=c_1(H)$ of an ample line bundle $H$ with $(h,h)=2rs$, where $r$ and $s$ are positive integers and $\gcd(r,s)=1$. The Mukai vector $v:=(r,h,s)$ is then isotropic.
The following properties of the moduli space $M$ of $H$-stable sheaves with Mukai vector $v$ are proved in \cite{mukai-duality-of-K3}.
$M$ is a $K3$ surface with a cyclic Picard group generated by an ample line bundle $\hat{H}$ with class $\hat{h}:=c_1(\hat{H})$ satisfying $(\hat{h},\hat{h})=(h,h)$.  
$M$ parametrizes locally free $H$-slope-stable sheaves. For each integer $k$, such that $ks\equiv 1$ (mod $r$) there exists over $M\times S$ a universal vector bundle $\U$ with $c_1(\U)=\pi_M^*(k\hat{h})+\pi_S^*(h)$. The universal vector bundle $\U$ restricts to an $\hat{H}$-slope-stable vector bundle on $M\times\{x\}$ of Mukai vector $\hat{v}:=(r,k\hat{h},k^2s)$, for each point $x\in S$, and the classification morphism induces an isomorphism between $S$ and the moduli space $M_{\hat{H}}(\hat{v})$ of $\hat{H}$-stable sheaves on $M$, by \cite[Theorem 1.2]{mukai-duality-of-K3}.

\begin{prop}
\label{buskin-prop-4.19}
(\cite[Prop. 4.19]{buskin}\footnote{Buskin states it for $\kappa(\U^*)$, but the latter is equal to $\kappa(\U)$, by Corollary \ref{corollary-kappa-E-equal-kappa-E-dual}.} )
The class $-\kappa(\U)\sqrt{td_{M\times S}}$ induces a degree reversing  rational Hodge isometry
$\phi_\U:\widetilde{H}(M,\QQ)\rightarrow \widetilde{H}(S,\QQ)$, which restricts to a rational Hodge isometry
$\psi_\U:H^2(M,\QQ)\rightarrow H^2(S,\QQ)$ satisfying $\psi_\U(\hat{h})=h$, such that $\eta_S\circ\psi_\U\circ\eta_M^{-1}$  belongs to the double orbit $O(\Lambda_{K3})\rho_u O(\Lambda_{K3})$ of a reflection $\rho_u$ in the co-rank $1$ lattice $u^\perp\subset \Lambda_{K3}$ of a primitive class $u$ with $(u,u)=2r$.
\end{prop}

 We claim that $\U$ is slope-stable with respect to every K\"{a}hler class on $M\times S$. It suffice to show it for every ample class, i.e., for $\omega:=a\pi_S^*(h)+b\pi_M^*(\hat{h})$, for all $a>0$ and $b>0$.
Let $\W$ be a saturated non-zero proper subsheaf of $\U$. Then $\W$ is reflexive, hence it is locally free away from a closed subvariety of dimension at most $1$. It follows that the restriction of $\W$ to the generic fiber of $\pi_S$ and of $\pi_M$ is locally free. Let $W_x$ (resp. $W_m$) be such a restriction to the fiber over $x\in S$ (resp. $m\in M$). 
Define $U_x$ and $U_m$ similarly in terms of $\U$.
Then $c_1(\W)=\pi_M^*c_1(W_x)+\pi_S^*c_1(W_m)$ and $\omega^3=a^2bh^2\hat{h}+ab^2h\hat{h}^2$. Set $\rho:=\rank(\W)$.
We have
\begin{eqnarray*}
\mu_\omega(\W) &=& \frac{1}{\rho}c_1(\W)\omega^3=
a^2b(h,h)\mu_{\hat{h}}(W_m)+ab^2(\hat{h},\hat{h})\mu_h(W_x)
\\ &<&
a^2b(h,h)\mu_{\hat{h}}(U_m)+ab^2(\hat{h},\hat{h})\mu_h(U_x)=\mu_\omega(\U).
\end{eqnarray*} 
Hence, $\U$ is slope-stable with respect to  every  K\"{a}hler class on $M\times S$. The universal vector bundle $\U$ thus satisfies the assumptions of Proposition \ref{prop-deforming-E} (stated for $E$). 
Set $\psi:=\eta_S\circ\psi_\U\circ\eta_M^{-1}$.
Let $\fM_\psi^0$ be the connected component of $\fM_\psi$ through $q_0:=(M,\eta_M,S,\eta_S)$.

Choose a generic twistor line $Q$ in $\fM_\psi^0$ through $q_0$ and let $(S',\eta_{S'},M',\eta_{M'})$ be a quadruple in this line with $\Pic(S')=0$. There exists classes $\zeta_{S'}\in H^2(S',\StructureSheaf{S'}^*)$ and 
$\zeta_{M'}\in H^2(M',\StructureSheaf{M'}^*)$
and a $\zeta$-twisted locally free sheaf $\U'$ over $M'\times S'$, with $\zeta:=\pi_{S'}^*(\zeta_{S'})\pi_{M'}^*(\zeta_{M'})$,
such that $(M'\times S',\SheafEnd(\U'))$ is deformation equivalent to  $(M\times S,\SheafEnd(\U))$, by Proposition \ref{prop-deforming-E}.
Moreover, $\U'$ is $\omega$-slope-stable with respect to some K\"{a}hler class on $S'\times M'$, again by Proposition \ref{prop-deforming-E}. We note that the existence of such $\U'$ follows already from \cite[Theorem 5.1]{buskin} in this case where $n=1$. 
It follows that $\U'$ does not have any non-zero proper saturated subsheaf, since $\Pic(M'\times S')$ is trivial. 

The construction of the sheaf $E$ over $M^{[n]}\times S^{[n]}$, given in (\ref{eq-universal-bundle-over-product-of-Hilbert-schemes}), has a natural relative extension over the chosen twistor line $Q$ in $\fM^0_\psi$ through $q_0$ resulting in  the sheaf $E'$ over $M'^{[n]}\times S'^{[n]}$ given by substituting $\U'$ for $\U$ in equation (\ref{eq-universal-bundle-over-product-of-Hilbert-schemes}). The check that equation (\ref{eq-universal-bundle-over-product-of-Hilbert-schemes}) is still meaningful in the twisted setting reduces to checking that 
$(q\times q)_*^{\fS_{n,\Delta}}$ is well defined. This is the case, since $\U'^{\boxtimes n}$ is $\zeta^{\boxtimes n}$-twisted,  and $(b\times b)^*(\zeta^{\boxtimes n})$ is the pullback via $q\times q$ of a Brauer class on $M'^{[n]}\times S'^{[n]}$ (which is  itself the pullback of a Brauer class on $M'^{(n)}\times S'^{(n)}$),
as explained in \cite[Sec. 10, Proof of Theorem 1.4]{markman-modular}. 

Let $\psi_{\U'}:H^2(M',\QQ)\rightarrow H^2(S',\QQ)$ and $\psi_{E'}:H^2(M'^{[n]},\QQ)\rightarrow H^2(S'^{[n]},\QQ)$ be the parallel-transports of $\psi_\U$ and $\psi_\E$ in the relevant local systems over the chosen twistor line $Q$. 

\begin{lem}
\label{lemma-psi-E-prime-maps-a-Kahler-class-to-such}
The isometry $\psi_{E'}$ maps some K\"{a}hler class $\omega_1$ on $M'^{[n]}$ to a K\"{a}hler class on $S'^{[n]}$.
\end{lem}

\begin{proof}
The isometry $\psi_{\U'}$ maps some K\"{a}hler class $\omega_{M'}$ to a K\"{a}hler class $\omega_{S'}$, as 
$(S',\eta_{S'},M',\eta_{M'})$ is a point of 
$\fM_\psi$. The equalities $\psi_{E'}\circ\theta=\theta\circ\psi_{\U'}$ and $\psi_{E'}(\delta)=\delta$ follow from Corollary \ref{cor-psi-E-restricts-to-psi-U}. The Picard groups  of both Douady spaces $M'^{[n]}$ and $S'^{[n]}$ are cyclic generated by the classes $\delta$ and their cones of effective curves are thus rays spanned by the class $\delta^\vee:=\frac{2(\delta,\bullet)}{n-1}$
in $H^2(M'^{[n]},\QQ)^*$ (were we identify $H_2(M'^{[n]},\QQ)$ with $H^2(M'^{[n]},\QQ)^*$), by \cite[Cor. 3.6]{markman-exceptional}. Choose a sufficiently large  real number $t$, such that $(t\theta(\omega_{M'})-\delta,t\theta(\omega_{M'})-\delta)>0$.
Then $\omega_1:=t\theta(\omega_{M'})-\delta$ is a class in the connected component of the positive cone in $H^{1,1}(M'^{[n]},\RR)$, which contains the K\"{a}hler cone, as it pairs positively with $\theta(\omega_{M'})$, and $\omega_1$ pairs positively with the class of every effective curve. Hence, $\omega_1$ is a K\"{a}hler class, by \cite[Cor. 3.4]{huybrechts-kahler-cone}. Now, $\psi_{E'}(\omega_1)=t\theta(\omega_{S'})-\delta$
is a K\"{a}hler class, by the same reason.
\end{proof}

Let $\eta_{S'^{[n]}}$ be the marking of $S'^{[n]}$ induced by extending the marking $\eta_{S^{[n]}}$ given in (\ref{eq-eta-S-[n]})
to a trivialization of the local system over the chosen twistor path $Q$. Define $\eta_{M'^{[n]}}$ similarly in terms of $\eta_{M^{[n]}}$. 
Set $\psi^{[n]}:=\eta_{S'^{[n]}}^{-1}\psi_{E'}\eta_{M'^{[n]}}.$
Let 
\begin{equation}
\label{eq-M-psi-[n]}
\fM^0_{\psi^{[n]}}
\end{equation}
be the connected component of $\fM_{\psi^{[n]}}$ containing the quadruple $(M'^{[n]},\eta_{M'^{[n]}},S'^{[n]},\eta_{S'^{[n]}})$.

\begin{lem}
\label{lemma-stability-of-E-prime}
The vector bundle $E'$ is slope-stable with respect to every K\"{a}hler class on $M'^{[n]}\times S'^{[n]}$.
\end{lem}

\begin{proof}
We have already noted that $\U'$ does not have any non-zero proper saturated subsheaf. 
Hence,
$\U'^{\boxtimes n}$ does not have any non-zero proper saturated subsheaves, by 
\cite[Prop. 10.1]{markman-modular}. Given $\sigma\in \fS_n$, let $(1\times \sigma):M^n\times S^n\rightarrow M^n\times S^n$
act as the identity on $M^n$ and via $\sigma$ on $S^n$.
The vector bundles $\U^{\boxtimes n}$ and $(1\times\sigma)^*\U^{\boxtimes n}$ are non-isomorphic, for every $\sigma\in \fS_n$, $\sigma\neq 1$.
Furthermore, $\Hom(\U^{\boxtimes n},(1\times \sigma)^*\U^{\boxtimes n})=0,$ for such $\sigma$. 
Now the dimensions of $\Hom(\U^{\boxtimes n},(1\times \sigma)^*\U^{\boxtimes n})$ and $\Hom(\U'^{\boxtimes n},(1\times \sigma)^*\U'^{\boxtimes n})$ are equal, by the invariance of the  dimension of sheaf  cohomologies of hyperholomorphic sheaves under twistor deformations 
\cite[Sec. 3.5]{kaledin-verbitsky-book}.
Hence
\[
\oplus_{\tau\in \fS_n}(1\times \tau)^*(\U'^{\boxtimes n})
\]
does not have any non-zero proper saturated $\fS_n$-invariant subsheaf, by \cite[Lemma 11.5]{markman-modular}.
The same is true for $(b\times b)^*\U'^{\boxtimes n}$,
as $b\times b$ is a birational morphism. 

Consider the commutative diagram
\[
\xymatrix{
(q\times  q)^*(q\times q)_*^{\fS_{n,\Delta}}(b\times b)^*\U'^{\boxtimes n}
\ar[r]^{p\circ ev\circ e} \ar[d]_{e} &
\oplus_{\tau\in\fS_n}(1\times \tau)^* (b\times b)^*\U'^{\boxtimes n}
\\
(q\times  q)^*(q\times q)_*(b\times b)^*\U'^{\boxtimes n}
\ar[r]_{ev} &
\oplus_{(\sigma,\tau)\in\fS_n\times\fS_n}(\sigma\times \tau)^* (b\times b)^*\U'^{\boxtimes n}. \ar[u]^{p}
}
\]
Above, $e$ is the natural injective homomorphism from the $\fS_{n,\Delta}$-invariant subsheaf with respect to the linearization $\chi\otimes\rho_\boxtimes$ of $\U'^{\boxtimes n}$.
The right arrow $p$ is the projection on the subset of direct summands.
For each $(\sigma,\tau)\in\fS_n\times\fS_n$ we have the evaluation homomorphism from the bottom left vector bundle to the corresponding direct summand of the bottom right vector bundle due to the equality $(q\times q)\circ (\sigma\times \tau)=(q\times q)$.
The homomorphism $ev$ is the resulting diagonal homomorphism. The homomorphism $ev$ restricts to an isomorphism over the complement of the ramification divisor of $q\times q$. 

The homomorphism  $e$ is $\fS_{n,\Delta}$-equivariant. The homomorphism $ev$ is
$\fS_n\times\fS_n$-equivariant. The homomorphism $p$  is
$\fS_n\times\fS_n$-equivariant, where the latter group acts on the top right vector bundle via the projection $\fS_n\times\fS_n\rightarrow \fS_n$ onto the second factor. It follows that the homomorphism $p\circ ev\circ e$
is $\fS_{n,\Delta}$-equivariant, where the latter group acts via the isomorphism $\fS_{n,\Delta}\rightarrow \fS_n$, given by $(\sigma,\sigma)\mapsto\sigma$ on the  top right vector bundle. 

The homomorphism $p\circ ev\circ e$ restricts to an isomorphism over the complement of the ramification divisor of $q\times q$. 
Hence, the pullback via $q\times q$ of any saturated non-zero proper subsheaf  of $(q\times q)_*^{\fS_{n,\Delta}}(b\times b)^*\U'^{\boxtimes n}$ maps via $p\circ ev\circ e$ to an $\fS_n$-invariant non-zero subsheaf  of the top right vector bundle, whose saturation is a proper subsheaf. We have seen that  such a subsheaf does not exist. 
It follows that $(q\times q)_*^{\fS_{n,\Delta}}(b\times b)^*\U'^{\boxtimes n}$ does not have any saturated non-zero proper subsheaf. Hence neither does  $E'$, as the latter is the tensor product of the former with a line bundle.
\end{proof}

\begin{thm}
\label{thm-algebraicity-in-the-double-orbit-of-a-reflection}
Every quadruple $(X_1,\eta_1,X_2,\eta_2)$ in the component $\fM_{\psi^{[n]}}^0$, given in (\ref{eq-M-psi-[n]}),
supports a compatible vector bundle (Definition \ref{def-quadruple-supports-a-compatible-vector-bundle}).
\end{thm}

\begin{proof}
Lemma \ref{lemma-psi-E-prime-maps-a-Kahler-class-to-such} verifies Assumption 
\ref{assumption-quadruple-is-in-fM-psi}(\ref{assumption-item-psi-E-maps-a-Kahler-class-to-same}) for the locally free sheaf $E'$ supported by the quadruple $(M'^{[n]},\eta_{M'^{[n]}},S'^{[n]},\eta_{S'^{[n]}})$ in the component $\fM^0_{\psi^{[n]}}$. 
Lemma \ref{lemma-stability-of-E-prime} verifies the stability condition in 
Assumption \ref{assumption-quadruple-is-in-fM-psi}(\ref{assumption-item-stability})
for $E'$.
The statement thus follows from Proposition \ref{prop-deforming-E}.
\end{proof}

Let $\Lambda_{K3}$ be the $K3$ lattice, $\Lambda$ the lattice given in (\ref{eq-K3-n-lattice}), and $\iota:O(\Lambda_{K3})\rightarrow O(\Lambda)$ the inclusion (\ref{eq-iota}). Let $u\in \Lambda_{K3}$
be a primitive element with $(u,u)>0$. Let $\rho_u$ be the reflection in the co-rank $1$ sublattice $u^\perp\subset \Lambda_{K3}$.
Let $\Mon(\Lambda)\subset O^+(\Lambda)$ be the subgroup which acts on the discriminant group $\Lambda^*/\Lambda$ via $\pm 1$.

Let $\fM_\Lambda^0$ be the connected component of the moduli space $\fM_\Lambda$ of marked pairs $(X,\eta)$, with $X$ of $K3^{[n]}$-type, containing the marked pair $(S^{[n]},\eta_{S^{[n]}})$, given in (\ref{eq-eta-S-[n]}).

\begin{cor}
\label{cor-algebraicity-in-the-double-orbit-of-a-reflection}
Let $\tilde{\psi}\in O^+(\Lambda_\QQ)$ be an isometry in the double orbit $\Mon(\Lambda)(-\iota(\rho_u))\Mon(\Lambda)$.
There exists a non-empty connected component $\fM_{\tilde{\psi}}^0$ of \ $\fM_{\tilde{\psi}}$
parametrizing quadruples $(X_1,\eta_1,X_2,\eta_2)$, with $(X_1,\eta_1)$ and $(X_2,\eta_2)$ in $\fM_\Lambda^0$, such that each quadruple in $\fM_{\tilde{\psi}}^0$ supports a compatible vector bundle (Definition \ref{def-quadruple-supports-a-compatible-vector-bundle}). 
\end{cor}

\begin{proof}
The double orbit $O^+(L)(-\rho_u) O^+(L)$ depends only on $(u,u)$, by \cite[Prop. 3.3]{buskin}. Hence, the double orbit 
$\Mon(\Lambda)(-\iota(\rho_u))\Mon(\Lambda)$ depends only on $(u,u)$, and so we may assume that 
$\tilde{\psi}=\gamma_1\iota(\eta_S\psi_\U\eta_M^{-1})\gamma_0$, for some $\gamma_1,\gamma_2\in \Mon(\Lambda)$, where $\psi_\U$ is the isometry in Proposition \ref{buskin-prop-4.19}. 
Set $\psi:=\eta_S\psi_\U\eta_M^{-1}$.
The isometry $\psi^{[n]}$, given in (\ref{eq-M-psi-[n]}), is equal to $\iota(\psi)$, by Corollary \ref{cor-psi-E-restricts-to-psi-U},
and so $\tilde{\psi}=\gamma_2\psi^{[n]}\gamma_1$. We get the isomorphism
$
f:\fM_{\psi^{[n]}}\rightarrow \fM_{\tilde{\psi}}
$ given by 
\[
f(X_1,\eta_1,X_2,\eta_2)= (X_1,\gamma_1^{-1}\eta_1,X_2,\gamma_2\eta_2),
\]
which maps connected components to connected components, since $\eta_i^{-1}\Mon(\Lambda)\eta_i$ is the monodromy group $\Mon^2(X_i)$ of $X_i$, $i=1,2$, by \cite[Theorem 1.6 and Lemma 4.10]{markman-monodromy}. 
Furthermore, the marked pair $(X_1,\gamma_1^{-1}\eta_1)$ belongs to the connected component $\fM_\Lambda^0$ of $(X_1,\eta_1)$, since $\gamma_1^{-1}$ belongs to $\Mon(\Lambda)$ and $\Mon(\Lambda)=\eta_1\Mon^2(X_1)\eta_1^{-1}$. Similarly,  $(X_2,\gamma_2\eta_2)$ belongs to the same connected component $\fM_\Lambda^0$ of $(X_2,\eta_2)$.
Clearly, if $q\in \fM_{\psi^{[n]}}$ supports a compatible vector bundle, then $f(q)$ supports (the same) compatible vector bundle. The statement now follows from Theorem \ref{thm-algebraicity-in-the-double-orbit-of-a-reflection}.
\end{proof}

%
\section{Algebraicity of a rational Hodge isometry}
\label{sec-algebraicity}
We prove Theorems \ref{main-thm} and \ref{thm-lifting-f-to-a-morphism-in-G} in this section.
Let $\G_{an}$ be the subgroupoid of $\G$ given in (\ref{eq-G-an}).
\hide{
, with the same objects, but morphisms 
in $\Hom_{\G_{an}}((X,\epsilon),(Y,\epsilon'))$
are compositions of two types of morphisms in $\G$.
\begin{enumerate}
\item
\label{item-parallel-transport-and-Hodge-isometry}
A parallel-transport operator $f:H^*(X,\QQ)\rightarrow H^*(Y,\QQ)$, which in addition is an isomorphism of Hodge structures.
\item
A degree reversing isometry $f:H^*(X,\QQ)\rightarrow H^*(Y,\QQ)$, which belongs to $\Hom_\G((X,\epsilon),(Y,\epsilon'))$ for any choice of orientations $\epsilon$ and $\epsilon'$, and which is induced by the correspondence 
$\pi_X^*{\sqrt{td_X}}\kappa(F)\pi_Y^*{\sqrt{td_X}}$, for some  possibly twisted locally free sheaf $F$  over $X\times Y$.
\end{enumerate}
}
Note that morphisms 
in $\Hom_{\G_{an}}((X,\epsilon),(Y,\epsilon'))$ are induced by analytic correspondences in $X\times Y$ and are degree preserving up to sign.

\begin{thm}
\label{thm-pre-fullness}
Let $X$ and $Y$ be irreducible holomorphic symplectic manifolds of $K3^{[n]}$-type and
let $f:H^2(X,\QQ)\rightarrow H^2(Y,\QQ)$ be a rational Hodge isometry.
For every orientations $\epsilon$ and $\epsilon'$ there exists a morphism $\phi$ in $\Hom_{\G_{an}}((X,\epsilon),(Y,\epsilon'))$,
such that 
$\widetilde{H}_0(\phi)$
is equal to $f$ or $-f$
\end{thm}

\begin{proof}
The set $\Hom_{\G_{an}}((X,\epsilon),(Y,\epsilon'))$ is independent of the orientations as is the functor $\widetilde{H}_0$, so the existence of $\phi$ is independent of the orientations.
We may assume that $\nu(f)=1$, possibly after replacing $f$ by $-f$. Choose markings $\eta_X$ and $\eta_Y$, such that $(X,\eta_X)$ and $(Y,\eta_Y)$ belong to the same connected component $\fM^0_\Lambda$ of $\fM_\Lambda$. 
Choose an orientation of $\Lambda$ and endow $H^2(X,\ZZ)$ and $H^2(Y,\ZZ)$ with the orientations $\epsilon$ and $\epsilon'$, so that $\eta_X$ and $\eta_Y$ are orientation preserving. We write $\Hom_{\G_{an}}(X,Y)$ instead of 
$\Hom_{\G_{an}}((X,\epsilon),(Y,\epsilon'))$, as in Convention \ref{convention-on-orientations-and-the-groupoid-G}, and we will continue to use the convention below.
Set $\psi:=\eta_Y\circ f\circ \eta_X^{-1}$. Then $\psi$ belongs to $O^+(\Lambda_\QQ)$ and so it decomposes as
\[
\psi=\psi_k\circ \psi_{k-1}\circ \cdots \circ \psi_1,
\]
where each $\psi_i$ belongs to the double orbit $\Mon(\Lambda)\iota(-\rho_u)\Mon(\Lambda)$, for some element $u\in \Lambda_{K3}$ with $(u,u)>0$, by Corollary \ref{cor-decomposition-of-a-rational-isometry}.
The proof proceeds by induction on $k$.

Case $k=0$. In this case $\psi=id$ and $f=\eta_Y^{-1}\eta_X$ is a Hodge isometry. There exists a parallel-transport operator $\phi:H^*(X,\QQ)\rightarrow H^*(Y,\QQ)$, which is a morphism in $\Hom_{\G_{an}}(X,Y)$, and such that $\phi$
restricts to $H^2(X,\QQ)$ as $f$, by Theorem \ref{thm-two-inseparable-marked-pairs}. In this case $\widetilde{H}_0(\phi)$
is also equal to the restriction of $\phi$ to $H^2(X,\QQ)$, by Remark \ref{remark-convention-behaves-well-with-respect-to-parallel-transport-operators}.

Assume that $k\geq 1$ and that the statement holds for $k-1$. 
Set $\psi'=\psi_{k-1}\circ \cdots \circ \psi_1$, so that $\psi=\psi_k\psi'$. Set $\ell_0:=P(X,\eta_X)$,
$\ell_{k-1}:=\psi'(\ell_0)$, and $\ell_k=\psi(\ell_0)$. Note that $\ell_k=P(Y,\eta_Y)$. 
Let $(X_{k-1},\eta_{k-1})$ be a marked pair in $\fM_\Lambda^0$, such that $P(X_{k-1},\eta_{k-1})=\ell_{k-1}$. Such a marked pair exists, by the surjectivity of the period map \cite{huybrects-basic-results}.
Set $f':=\eta_{k-1}^{-1}\circ\psi'\circ\eta_X$.
There exists a morphism $\phi'\in \Hom_{\G_{an}}(X,X_{k-1})$, such that 
$\widetilde{H}_0(\phi')=f'$,
by the induction hypothesis. The quadruple $(X_{k-1},\eta_{k-1},Y,\eta_Y)$ 
satisfies the condition that $\eta_Y^{-1}\psi_k\eta_{k-1}$ is a Hodge isometry, since
\[
\eta_Y^{-1}\psi_k\eta_{k-1}(H^{2,0}(X_{k-1}))=
\eta_Y^{-1}(\psi_k(\ell_{k-1}))=\eta_Y^{-1}(\psi(\ell_0))=\eta_Y^{-1}(P(Y,\eta_Y))=H^{2,0}(Y),
\]
but it may not belong to $\fM_{\psi_k}$, since $\eta_Y^{-1}\psi_k\eta_{k-1}$ need not map a K\"{a}hler class to a K\"{a}hler class.

There exists a connected component $\fM^0_{\psi_k}$ of $\fM_{\psi_k}$ in which every quadruple supports a compatible vector bundle, and such that $\Pi_i:\fM^0_{\psi_k}\rightarrow \fM_\Lambda$ has image in the connected component $\fM^0_\Lambda$, for $i=1,2$,
by Corollary \ref{cor-algebraicity-in-the-double-orbit-of-a-reflection}. 
Extend the pair $(X_{k-1},\eta_{k-1})$ to a quadruple $q:=(X_{k-1},\eta_{k-1},Y',\eta_{Y'})$ in $\fM^0_{\psi_k}$.
Such a point $q$ exists, by the surjectivity Lemma \ref{lemma-Pi-i-is-surjective}. 
Let $F_q$ be the compatible vector bundle supported by $q$.
Set $\phi_k':=\left[\kappa(F_q)\sqrt{td_{X_{k-1}\times Y'}}\right]_*:H^*(X_{k-1},\QQ)\rightarrow H^*(Y',\QQ)$.
Then 
$
\widetilde{H}_0\left(\phi_k'\right)=\eta_{Y'}^{-1}\circ \psi_k\circ \eta_{k-1},
$
by Definition \ref{def-quadruple-supports-a-compatible-vector-bundle} of the compatibility of $F_q$.
We have $P(Y',\eta_{Y'})=\psi_k(P(X_{k-1},\eta_{k-1}))=P(Y,\eta_Y)$, where the first equality is by Lemma \ref{lemma-periods-of-quadruple}. Hence, there exists a morphism $\phi_k\in \Hom_{\G_{an}}(Y',Y)$, 
such that $\widetilde{H}_0(\phi_k)=\eta_Y^{-1}\eta_{Y'}$, by the case $k=0$ of the statement. We conclude that
$\phi:=\phi_k\circ\phi_k'\circ\phi'$ is a morphism in $\Hom_{\G_{an}}(X,Y)$, and 
$\widetilde{H}_0(\phi)=f$.
\end{proof}

\begin{proof}[Proof of Theorem \ref{thm-lifting-f-to-a-morphism-in-G}]
The Theorem is a rephrasing of Theorem \ref{thm-pre-fullness}.
\end{proof}

\begin{proof}[Proof of Theorem \ref{main-thm}]
Let $f:H^2(X,\QQ)\rightarrow H^2(Y,\QQ)$ be a Hodge isometry.
We may assume that $\nu(f)=1$, possibly after replacing $f$ by $-f$ (and $\tilde{f}$ by $-\tilde{f}$ in the statement of the theorem). There exist orientations $\epsilon$ and $\epsilon'$ and a morphism $\phi\in \Hom_{\G_{an}}((X,\epsilon),(Y,\epsilon'))$, such that $\widetilde{H}_0(\phi)=f$, by Theorem \ref{thm-pre-fullness}. If $\phi$ is degree preserving, then $\widetilde{H}_0(\phi)$ is a rational scalar multiple of the restriction of $\phi$ to $H^2(X,\QQ)$, by Lemma \ref{lemma-widetilde-H-of-a-degree-preserving-morphism}, 
and so we choose $\tilde{f}$ to be the appropriate scalar multiple of $\phi$. If $\phi$ is degree reversing, then setting $\tilde{f}=\phi$ we get an analytic correspondence with the desired properties, by Lemma \ref{lemma-widetilde-H-of-a-degree-reversing-morphism}.
\end{proof}

%
\section{The Pontryagin product on the cohomology of IHSMs of $K3^{[n]}$-type}
\label{sec-pontryagin}

In Section \ref{sec-Pontryagin-product-in-general} we introduce the Pontryagin product on the cohomology of an IHSM with vanishing odd cohomology. The results still apply to the even cohomology if the odd cohomology does not vanish.
In Section \ref{sec-Pontryagin-product-fin-the-K3-[n]-type} we prove Conjecture \ref{conj-Pontryagin} for morphisms in the image of the functor $\Theta_n:\G^{[1]}_{an}\rightarrow \G^{[n]}_{an}$, given in (\ref{eq-functor-Theta-n}). In particular, morphisms in the image of $\Theta_n$, once normalized as in Conjecture \ref{conj-Pontryagin}, conjugate the cup product to itself, 
if they are degree preserving, and to the Pontryagin product if they are degree reversing.

%
\subsection{The Pontryagin product for an IHSM with vanishing odd cohomology}
\label{sec-Pontryagin-product-in-general}
Let $X$ be an IHSM, let $\LieAlg{g}\subset \LieAlg{sl}(H^*(X,\QQ))$ be its LLV algebra,  and let $h\in \LieAlg{g}$ be the grading operator given in (\ref{eq-grading-operator}). The commutator of $h$ in $\LieAlg{g}$ is $\bar{\LieAlg{g}}\oplus \QQ \cdot h$
and composing the isomorphism $\dot{\rho}^{-1}:\LieAlg{g}\cong \LieAlg{so}(\widetilde{H}(X,\QQ))$, given in (\ref{eq-dot-rho}), with restriction to $H^2(X,\QQ)$ yields an isomorphism
$\bar{\LieAlg{g}}\cong\LieAlg{so}(H^2(X,\QQ))$, by Lemma \ref{lemma-commutator-of-degree-operatoor}.
Assume that the odd cohomology of $X$ vanishes. If the second Betti number $b_2(X)$ of $X$ is even, assume that the monordomy group of $X$ contains reflections and $dim(X)=2n$ for odd $n$. The latter assumption holds true for all currently known examples of irreducible holomorphic symplectic manifolds with even $b_2(X)$ ($K3$ surfaces and IHSMs of O'Grady $6$ deformation type). The $\LieAlg{g}$ action on $H^*(X,\CC)$ integrates to an action of
$SO(\widetilde{H}(X,\CC))$ and the commutator of $h$ is $\CC^\times\times SO(H^2(X,\CC))$, where an element $t$ of the first factor $\CC^\times$ maps $\alpha$ to $t\alpha$ and $\beta$ to $\frac{1}{t}\beta$. 
The image of the second factor $SO(H^2(X,\CC))$ in $GL(H^*(X,\CC))$ 
acts on $H^*(X,\CC)$ via ring automorphisms, by \cite[Prop. 4.5(ii)]{looijenga-lunts}.
The coset $\{g\in SO(\widetilde{H}(X,\CC)) \ : \ ghg^{-1}=-h\}$ of the commutator is the product of the component of $O(\CC\alpha\oplus\CC\beta)$
with determinant $-1$ with the component of $O(H^2(X,\CC))$ with determinant $-1$. If the monodromy group contains an element $R$ of order $2$ acting as a reflection on $H^2(X,\CC)$, then the factor $O(H^2(X,\CC))$ acts on $H^*(X,\CC)$ via ring automorphisms. 
Such a monodromy operator $R$ acts on $\widetilde{H}(X,\QQ)$ via the operator $\widetilde{H}(R)$ stabilizing both $\alpha$ and $\beta$ and restricting to $H^2(X,\QQ)$ as the monodromy reflection. 
A choice of such an element $R\in \Mon(X)$ yields an action of $O(\widetilde{H}(X,\CC))$ on $H^*(X,\CC)$ 
\begin{equation}
\label{eq-Verbitsky-representation-rho-extended-to-O}
\rho:O(\widetilde{H}(X,\CC))\rightarrow GL(H^*(X,\CC))
\end{equation}
extending the above homomorphism $\rho:SO(\widetilde{H}(X,\CC))\rightarrow GL(H^*(X,\CC))$ and satisfying $\rho(\widetilde{H}(R))=R$.

Let $\tau$ be the element of $O(\widetilde{H}(X,\QQ))$ interchanging $\alpha$ and $\beta$ and acting on $H^2(X,\CC)$ via multiplication by $-1$. If $b_2(X)$ is odd, then $\tau$ belongs to $SO(\widetilde{H}(X,\QQ))$ and $\rho_\tau$ is well defined, and if $b_2(X)$ is even we define $\rho_\tau$ via (\ref{eq-Verbitsky-representation-rho-extended-to-O}).
Define the {\em Pontryagin product} $\star$ on $H^*(X,\QQ)$ by
\begin{equation}
\label{eq-Pontryagin-product}
\gamma\star\delta := \rho_\tau(\rho_\tau(\gamma)\cup\rho_\tau(\delta)).
\end{equation}

\begin{rem}
\label{remark-theVerbitsky-component-is-closed-under-Pontryagin-product}
The subring $SH^*(X,\CC)$ of $H^*(X,\CC)$ generated by $H^2(X,\CC)$ is invariant under 
monodromy operators and it is an LLV-subrepresentation, hence it is invariant under 
elements in the image of $\rho$. The defining Equation (\ref{eq-Pontryagin-product}) implies that 
$SH^*(X,\CC)$ is a subring with respect to the Pontryagin product.
\end{rem}

The values of the representation $\rho$ on the connected component with determinant $-1$ depends on our choice of the monodromy reflection $R$ (see \cite[Lemma 4.13]{markman-monodromy}). Thus, 
if $b_2(X)$ is even, then $\rho_\tau$ depends on the choice of $R$. Nevertheless, the Pontryagin product does not depend on this choice as we show next (and generalize in Remark \ref{remark-same-Pontryagin-product-in-terms-of-any-degree-reversing-isometry} below).

\begin{lem}
\label{lemma-Pontyagin-product-independent-of-choice-of-monodromy-reflection}
\begin{enumerate}
\item
\label{lemma-item-product-independent-of-R}
The Pontryagin product $\star$ is independent of the choice of the monodromy reflection $R$.
\item
\label{lemma-item-parallel-transport-operators-preserve-pontryagin-product}
Parallel-transport operators are isomorphisms of the Pontryagin rings.
\item
\label{lemma-item-unit-wrt-Pontryagin-product}
The unit with respect to $\star$ is $\rho_\tau(1)$ and is independent of the choice of $R$.
\end{enumerate}
\end{lem}

\begin{proof}
(\ref{lemma-item-product-independent-of-R})
The statement is clear if $b_2(X)$ is odd. Assume that $b_2(X)$ is even.
Let $R'\in \Mon(X)$ be another monodromy reflection and let $\rho':O(\widetilde{H}(X,\CC))\rightarrow GL(H^*(X,\CC))$ be the resulting representation. Let $f,f'\in SO(\widetilde{H}(X,\CC))$ be the elements satisfying 
$\tau=f\widetilde{H}(R)=f'\widetilde{H}(R')$.
Then $\rho_{f'}=\rho'_{f'}$, as the restrictions of $\rho$ and $\rho'$ to $SO(\widetilde{H}(X,\CC))$ are equal. Now, $f'=f\widetilde{H}(R)\widetilde{H}(R')=f\widetilde{H}(RR')$. So,
\[
\rho'_\tau=\rho_{f'}R'=\rho_{f\widetilde{H}(RR')}R'=\rho_f\rho_{\widetilde{H}(RR')}R'=
\rho_fRR\rho_{\widetilde{H}(RR')}R'=\rho_\tau(R\rho_{\widetilde{H}(RR')}R').
\]
Set $g:=R\rho_{\widetilde{H}(RR')}R'$. Then $\rho'_\tau=\rho_\tau g$, and $g$ is an automorphism of $H^*(X,\CC)$ with respect to cup product, since each of $R$, $\rho_{\widetilde{H}(RR')}$, and $R'$ is. Furthermore, 
$\rho_\tau g\rho_\tau g=(\rho'_\tau)^2=id$, so $g\rho_\tau g=\rho_\tau$. Now compute:
\begin{eqnarray*}
\rho'_\tau(\rho'_\tau(\gamma)\cup\rho'_\tau(\delta))
&=&
(\rho_\tau g)(\rho_\tau(g(\gamma))\cup\rho_\tau(g(\delta)))=
\rho_\tau((g\rho_\tau g)(\gamma)\cup (g\rho_\tau g)(\delta))
\\
&=&
\rho_\tau(\rho_\tau(\gamma)\cup\rho_\tau(\delta))=\gamma\star\delta.
\end{eqnarray*}

(\ref{lemma-item-parallel-transport-operators-preserve-pontryagin-product}) Let $f:H^*(X,\CC)\rightarrow H^*(Y,\CC)$ be a parallel-transport operator. We get the commutative diagram
\[
\xymatrix{
SO(\widetilde{H}(X,\CC))\ar[r] \ar[d]_{Ad_{\widetilde{H}(f)}} & GL(H^*(X,\CC))\ar[d]^{Ad_f}
\\
SO(\widetilde{H}(Y,\CC))\ar[r] & GL(H^*(Y,\CC)),
}
\]
where the horizontal arrows are the integration of the infinitesimal LLV Lie algebras actions. 
Clearly, $\widetilde{H}(f)$ conjugates $\tau$ to $\tau$. 
Parallel-transport operators conjugate a monodromy reflection to such, and so the representation $\rho$, 
given in (\ref{eq-Verbitsky-representation-rho-extended-to-O}), get conjugated via the pair $f$ and $\widetilde{H}(f)$ to $\rho':O(\widetilde{H}(Y,\CC))\rightarrow GL(H^*(Y,\CC))$, 
as in part (\ref{lemma-item-product-independent-of-R}), and so $f$ conjugates $\rho_\tau$ to $\rho'_\tau$. The statement of part (\ref{lemma-item-parallel-transport-operators-preserve-pontryagin-product}) now follows from part 
(\ref{lemma-item-product-independent-of-R}).

(\ref{lemma-item-unit-wrt-Pontryagin-product})
The unit of a ring is unique, and so $\rho_\tau(1)$ is independent of $R$, by Part \ref{lemma-item-product-independent-of-R}.
\end{proof}

We calculate the unit explicitly next.

\begin{lem}
\label{lemma-unit}
The unit of $H^*(X,\QQ)$ with respect to the $\star$ product is $c_X[pt]/n!$, where $\dim(X)=2n$ and $c_X$ is the Fujiki constant.
\end{lem}

\begin{proof}
Consider the $SO(\widetilde{H}(X,\QQ))$-equivariant embedding
\[
\Psi:SH^*(X,\QQ) \rightarrow \Sym^n(\widetilde{H}(X,\QQ)),
\]
given in (\ref{eq-taelmans-Psi}), where $SO(\widetilde{H}(X,\QQ))$ acts on $SH^*(X,\QQ)$ via the representation $\rho$ integrating the infinitesimal LLV action and via the natural representation on $\Sym^n(\widetilde{H}(X,\QQ))$.
If $b_2(X)$ is odd, then $\tau$ belongs to $SO(\widetilde{H}(X,\CC))$. If $b_2(X)$ is even,
then $X$ admits monodromy reflections, by assumptions, and $\Psi$ is equivariant with respect to those as well, and so
$\Psi$ is equivariant with respect to the $\rho_\tau$ action on its domain and the $\tau$ action of its codomain.
The element $1$ belongs to $SH^*(X,\CC)$ and $\Psi(1)=\alpha^n/n!$, by \cite[Proposition 3.5]{taelman}.
Now, $\tau(\alpha)=\beta$ and so 
$\Psi(\rho_\tau(1))=\rho_\tau(\Psi(1))=\beta^n/n!=\Psi(c_X[pt]/n!)$, where the last equality follows from \cite[Lemma 3.6]{taelman}. Hence, $\rho_\tau(1)=c_X[pt]/n!$.
\end{proof}

Let $\mu_t$ be
the element of $SO(\CC\alpha\oplus\CC\beta)$, given by $\mu_t(\alpha)=t^{-1}\alpha$ and $\mu_t(\beta)=t\beta$, and $\mu_t$ restricts to $H^2(X,\CC)$ as the identity. Then $\rho_{\mu_t}$ 
acts on $H^{2k}(X,\CC)[2n]$ via multiplication by $t^k$. 

\begin{lem}
\label{lemma-image-via-rho-of-operators-which-anti-commute-with-h}
Let $g$ be an element of  $O(\widetilde{H}(X,\CC))$.
\begin{enumerate}
\item
\label{lemma-item-ring-automorphism}
If $ghg^{-1}=h$, then $g(\alpha)=t\alpha$, for some $t\in \CC^\times$, and
$\rho_{\mu_t g}$ is an automorphism of the cohomology ring $H^*(X,\CC)$ with respect to the usual cup product.
\item
\label{lemma-item-conjugating-cup-product-to-pontryagin-product}
If $ghg^{-1}=-h$, then $g(\alpha)=t^{-1}\beta$, for some $t\in \CC^\times$, and 
\[
\rho_{\mu_t g}(\gamma\cup\delta) = 
\rho_{\mu_t g}(\gamma)\star\rho_{\mu_t g}(\delta),
\]
for all $\gamma, \delta\in H^*(X,\CC)$.
\end{enumerate}
\end{lem}

\begin{proof}
(\ref{lemma-item-ring-automorphism}) 
If $ghg^{-1}=h$ and $g(\alpha)=t\alpha$, then $\mu_t g$ belongs to the factor $OH^2(X,\CC)$ of the commutator of $h$, and so $\rho_{\mu_t g}$ acts by ring automorphisms. 

(\ref{lemma-item-conjugating-cup-product-to-pontryagin-product})
If $ghg^{-1}=-h$ and $g(\alpha)=t^{-1}\beta$, then $f:=\tau\mu_t g$ commutes with $h$ and the second equality below follows from part (\ref{lemma-item-ring-automorphism}).
\begin{eqnarray*}
\rho_{\mu_t g}(\gamma\cup \delta) & = &
\rho_{\tau f}(\gamma\cup \delta)=\rho_\tau(\rho_f(\gamma)\cup \rho_f(\delta))=
\rho_\tau(
\rho_\tau(\rho_{\tau f}(\gamma))\cup \rho_\tau(\rho_{\tau f}(\delta))
)
\\
&=&
\rho_{\tau f}(\gamma)\star \rho_{\tau f}(\delta)=\rho_{\mu_t g}(\gamma)\star\rho_{\mu_t g}(\delta).
\end{eqnarray*}
\end{proof}

\begin{rem}
\label{remark-same-Pontryagin-product-in-terms-of-any-degree-reversing-isometry}
Lemma \ref{lemma-image-via-rho-of-operators-which-anti-commute-with-h}(\ref{lemma-item-conjugating-cup-product-to-pontryagin-product}) equivalently states that the same Pontryagin product can be defined in terms of any degree reversing isometry $g\in O(\widetilde{H}(X,\QQ))$, by  
\[
\gamma\star\delta=\rho_{g^{-1}\mu_{t^{-1}}}(\rho_{\mu_tg}(\gamma)\cup \rho_{\mu_tg}(\delta)).
\]
\end{rem}
%
\subsection{The Pontryagin product for an IHSM of $K3^{[n]}$-type}
\label{sec-Pontryagin-product-fin-the-K3-[n]-type}
Assume from now on that $X$ is of $K3^{[n]}$-type.
The following lemma states that the functor $\Theta_n$, given in (\ref{eq-functor-Theta-n}), integrates the infinitesimal LLV Lie algebra action when restricted to automorphisms of determinant $1$.

\begin{lem}
\label{conjecture-Theta-n-is-integration-of-infinitesimal-LLV-action}
The restriction of $\Theta_n$ to the subgroup $S\Aut_{\G^{[1]}}(S,\epsilon)$, of elements of $\Aut_{\G^{[1]}}(S,\epsilon)$
of determinant $1$, is equal to the restriction of 
the composition of $\tilde{\iota}:SO(\widetilde{H}(S,\CC))\rightarrow SO(\widetilde{H}(S^{[n]},\CC))$, given in (\ref{eq-tilde-iota}),  with the representation $\rho$
given in (\ref{eq-Verbitsky-representation-rho-extended-to-O}). 
\end{lem}

\begin{proof}
The infinitesimal $LLV$-action $\dot{\rho}$, given in
(\ref{eq-dot-rho}), sends $e_{\theta(\lambda)}\in \LieAlg{so}(\widetilde{H}(S^{[n]},\QQ))$ to cup product with $\theta(\lambda)$, by the $\LieAlg{so}(\widetilde{H}(S^{[n]},\QQ))$-equivariance of $\Psi$ in (\ref{eq-taelmans-Psi}).
Thus, the equality $\Theta_n(\phi)=\rho_{\tilde{\iota}(\phi)}$ holds, when $\phi$ is 
cup product with $\exp(e_\lambda)$, $\lambda\in H^2(S,\QQ)$, by definition of $\Theta_n$.
The map $\Theta_n:\Aut_{\G^{[1]}}(S,\epsilon)\rightarrow \Aut_{\G^{[1]}}(S^{[n]},\epsilon^{[n]})$ is induced by an algebraic map, due to its topological nature via Grothendieck-Riemann-Roch, and so extends to the Zariski closure $O(\widetilde{H}(S,\CC))$ of $\Aut_{\G^{[1]}}(S,\epsilon)$. The statement thus reduces to the equality of the differentials of $\Theta_n$ and $\rho\circ\tilde{\iota}$. We already established the equality
$
d\Theta_n(e_\lambda)=\dot{\rho}(d\tilde{\iota}(e_\lambda))=\dot{\rho}(e_{\theta(\lambda)}).
$
We will denote $\dot{\rho}(e_{\theta(\lambda)})$ by $e_{\theta(\lambda)}$ as well, so that the latter equality becomes
\begin{equation}
\label{eq-differential-of-Theta-n-sends-e-lambda-to}
d\Theta_n(e_\lambda)=e_{\theta(\lambda)}.
\end{equation}

Choose $(M,\epsilon')$ and $\phi\in \Hom_{\G^{[1]}}((S,\epsilon),(M,\epsilon'))$, so that $\phi$ is degree reversing. 
We can choose, for example, $(M,\epsilon')=(S,\epsilon)$ and $\phi:=[ch(\Ideal{\Delta})\sqrt{td_{S\times S}}]_*$, where $\Ideal{\Delta}$ is the ideal sheaf of the diagonal.
Note that $\phi=\widetilde{H}(\phi)$, as $H^*(S,\QQ)=\widetilde{H}(S,\QQ)$ and similarly for $M$.
Let $t\in\QQ$, so that $\phi(\beta)=t\alpha$. Let $\lambda\in H^2(S,\QQ)$ be an element with $(\lambda,\lambda)\neq 0$. Then 
\begin{equation}
\label{eq-phi-conjugates-e-phi-lambda-to-e-dual-lambda}
\phi^{-1}e_{\phi(\lambda)}\phi=\frac{t(\lambda,\lambda)}{2}e_\lambda^\vee,
\end{equation} 
by Lemma \ref{lemma-Lefschetz}. Above $e_\lambda^\vee$ is the dual Lefschetz operator with respect to the grading operator of the LLV algebra $\LieAlg{g}_S$ of $S$.
The morphism $\widetilde{H}(\Theta_n(\phi))$ is degree reversing as well, by Corollaries \ref{cor-psi-E-restricts-to-psi-U} and \ref{cor-Theta_n-maps-G^[1]_an-to-G^[n]_an}. Hence, so is 
$\Theta_n(\phi)$, by Lemma \ref{lemma-preservation-of-grading-of-LLV-lattice-implies-that-of-cohomology}.
Say $(\widetilde{H}(\Theta_n(\phi)))(\beta)=t'\alpha$. Then 
\[
\Theta_n(\phi)^{-1}
e_{(\widetilde{H}(\Theta_n(\phi)))(\theta(\lambda))}
\Theta_n(\phi)=\frac{t'(\lambda,\lambda)}{2}e_{\theta(\lambda)}^\vee.
\]
by Lemma \ref{lemma-Lefschetz}. Above $e_{\theta(\lambda)}^\vee$ is the dual Lefschetz operator with respect to the grading operator of the LLV algebra $\LieAlg{g}_{S^{[n]}}$ of $S^{[n]}$ given in (\ref{eq-grading-operator}).
Note that $t'=\det(\phi)^{n+1}t$ and $\widetilde{H}(\Theta_n(\phi))=\det(\phi)^{n+1}\tilde{\iota}_\phi,$
by Corollary \ref{cor-psi-E-restricts-to-psi-U}. The second equality below follows (using also that $e_{\tilde{\iota}_\phi(\theta(\lambda))}=e_{\theta(\phi(\lambda))})$.
\[
Ad_{\Theta_n(\phi)^{-1}}(d\Theta_n(e_{\phi(\lambda)}))
\stackrel{(\ref{eq-differential-of-Theta-n-sends-e-lambda-to})}{=}
\Theta_n(\phi)^{-1}
e_{\theta(\phi(\lambda))}
\Theta_n(\phi)=\frac{t(\lambda,\lambda)}{2}e_{\theta(\lambda)}^\vee.
\]
Given a morphism 
$\phi\in\Hom_{\G^{[1]}}((S,\epsilon),(M,\epsilon'))$, we have
\[
d\Theta_n\circ Ad_\phi=Ad_{\Theta_n(\phi)}\circ d\Theta_n
\]
since 
$\Theta_n(\phi\psi\phi^{-1})=\Theta_n(\phi)\Theta_n(\psi)\Theta_n(\phi)^{-1}$. We conclude that
\[
d\Theta_n(Ad_{\phi^{-1}}(e_{\phi(\lambda)}))=\frac{t(\lambda,\lambda)}{2}e_{\theta(\lambda)}^\vee.
\]
The left hand side is $d\Theta_n\left(\frac{t(\lambda,\lambda)}{2}e_\lambda^\vee
\right)$, by Equation (\ref{eq-phi-conjugates-e-phi-lambda-to-e-dual-lambda}).
Hence, $d\Theta_n(e_\lambda^\vee)=e_{\theta(\lambda)}^\vee$.
Now $\dot{\rho}:\LieAlg{so}(\widetilde{H}(S^{[n]},\QQ))\rightarrow \LieAlg{g}_{S^{[n]}}$
maps the grading operator of $\LieAlg{so}(\widetilde{H}(S^{[n]},\QQ))$ to that of $\LieAlg{g}_{S^{[n]}}$ and the differential of
$\tilde{\iota}$ sends the grading operator of $\LieAlg{g}_S$ to that of $\LieAlg{so}(\widetilde{H}(S^{[n]},\QQ))$. 
Hence, the differential of $\rho\circ \tilde{\iota}$ sends $e_\lambda^\vee$ to $e_{\theta(\lambda)}^\vee$ and we get the equality
\[
d\Theta_n(e_\lambda^\vee)=d(\rho\circ\tilde{\iota})(e_\lambda^\vee).
\]
The statement follows, since $\LieAlg{so}(H^*(S,\QQ))$ is generated by the pairs $(e_\lambda,e^\vee_\lambda)$, as $\lambda$ varies over elements of $H^2(S,\QQ)$ with $(\lambda,\lambda)\neq 0$.
\end{proof}

Let $\R^{[n]}_{an}$ be the subgroupoid of $\G_{an}^{[n]}$ given in (\ref{eq-subgroupoid-R-[n]}).

\begin{prop}
\label{prop-restriction-of-wildtilde-H-0-is-full-as-well}
The restriction of the functor $\widetilde{H}_0$ to the subgroupoid $\R^{[n]}_{an}$  remains full. 
\end{prop}

\begin{proof}
It suffices to prove that all morphisms of $\G^{[n]}_{an}$, which were used in the proof of Theorem \ref{thm-lifting-f-to-a-morphism-in-G} to prove the fullness of $\widetilde{H}_0$, are already morphisms of $\R^{[n]}_{an}$. 
Parallel-transport operators, which are Hodge isometries, are morphisms in $\R^{[n]}_{an}$.
The morphisms other than parallel-transport operators where all deformations of those that appear in Corollary \ref{cor-Theta_n-maps-G^[1]_an-to-G^[n]_an}, hence conjugates of those that appear in Corollary \ref{cor-Theta_n-maps-G^[1]_an-to-G^[n]_an} via parallel-transport operators. It suffices to prove that the morphism $\Theta_n(\varphi)$ of $\G^{[n]}_{an}$ that appears in Corollary \ref{cor-Theta_n-maps-G^[1]_an-to-G^[n]_an} belongs to $\R^{[n]}_{an}$, in other words, that $\mu_{\chi(\Theta_n(\varphi))}\Theta_n(\varphi)$ conjugates the cup product of $H^*(S^{[n]},\QQ)$ to the Pontryagin product. 

Let $f:H^*(S,\Integers)\rightarrow H^*(M,\Integers)$ be a parallel-transport operator. Then $f\circ \varphi$ belongs to $\Aut_{\G^{[1]}}(M)$, $\mu_{\chi(\Theta_n(f\circ \varphi))}=\mu_{\chi(\Theta_n(\varphi))}$ and $\mu_{\chi(\Theta_n(\varphi))}\Theta_n(\varphi)$ conjugates the cup product of $H^*(M^{[n]},\QQ)$ to the Pontryagin product, if and only if 
$\mu_{\chi(\Theta_n(f\circ \varphi))}\Theta_n(f\circ \varphi)$ does, by Lemma \ref{lemma-Pontyagin-product-independent-of-choice-of-monodromy-reflection}(\ref{lemma-item-parallel-transport-operators-preserve-pontryagin-product}). Hence, it suffices to prove the statement of the Proposition for automorphisms in the image of $\Theta_n$.

The restriction of $\Theta_n$ to the subgroup $S\Aut_{\G^{[1]}}(S,\epsilon)$ of elements of $\Aut_{\G^{[1]}}(S,\epsilon)$
of determinant $1$ is equal to the restriction of 
the composition of $\tilde{\iota}:SO(\widetilde{H}(S,\CC))\rightarrow SO(\widetilde{H}(S^{[n]},\CC))$ with the representation $\rho$
given in (\ref{eq-Verbitsky-representation-rho-extended-to-O}), by 
Lemma \ref{conjecture-Theta-n-is-integration-of-infinitesimal-LLV-action}. 
The two agree on the whole of $\Aut_{\G^{[1]}}(S,\epsilon)$
if we choose the reflection $R$ in the definition of $\rho$ to be the image of a monodromy reflection of $S$. 
Hence, if $\varphi\in \Aut_{\G^{[1]}}(S,\epsilon)$ conjugates $h$ to $-h$, then
$\mu_{\chi(\Theta_n(\varphi))}\Theta_n(\varphi)=\rho(\mu_{\chi(\varphi)}\varphi)$ conjugates the cup product to the Pontryagin product and if $\varphi$ commutes with $h$, then $\mu_{\chi(\Theta_n(\varphi))}\Theta_n(\varphi)$ is a ring automorphism with respect to cup product, by Lemmas \ref{lemma-Pontyagin-product-independent-of-choice-of-monodromy-reflection} and  
\ref{lemma-image-via-rho-of-operators-which-anti-commute-with-h}.
\end{proof}

\begin{proof}[Proof of Corollary \ref{cor-Pontryagin-product-is-algebraic}]
Choose an IHSM $X$ and a twisted vector bundle $E$ over $X\times Y$ as in  Theorem \ref{thm-introduction-algebraicity-in-the-double-orbit-of-a-reflection}. There exists a rational number $t$, such the the isomorphism
$\phi:=\mu_t\circ [\kappa(E)\sqrt{td_{X\times Y}}]_*:H^*(X,\QQ)\rightarrow H^*(Y,\QQ)$ satisfies Equation (\ref{eq-normalized-FM-transformation-conjugate-cup-product-to-Pontryagin}), by the proof of 
Proposition \ref{prop-restriction-of-wildtilde-H-0-is-full-as-well}. Given classes $\gamma_1, \gamma_2\in H^*(Y,\QQ)$, we thus have
\begin{equation}
\label{eq-Pontriagin-product-of-gamma-1-and-gamma-2}
\phi(\phi^{-1}(\gamma_1)\cup\phi^{-1}(\gamma_2))=\gamma_1\star\gamma_2.
\end{equation}
The isomorphism $\mu_t$ is algebraic, since the Lefschetz standard conjecture holds for $Y$, by \cite{charles-markman}, and so the K\"{u}nneth factors of the diagonal in $H^i(Y,\QQ)\otimes H^{4n-i}(Y,\QQ)$, $0\leq i\leq 4n$, are algebraic, by \cite[Prop. 1.4.4]{kleiman}. Hence, the isomorphism $\phi$ is induced by an algebraic correspondence. Now $\phi^{-1}:H^*(Y,\QQ)\rightarrow H^*(X,\QQ)$ is equal to
$[\kappa(E^*)\sqrt{td_{Y\times X}}]_*\circ\mu_{t^{-1}}$ and is algebraic as well. Denote by 
$(\phi^{-1})^t:H^*(X,\QQ)\rightarrow H^*(Y,\QQ)$ the isomorphism induce by the class in $H^*(X\times Y,\QQ)$, which is the transpose of the class in $H^*(Y\times X,\QQ)$ inducing $\phi^{-1}$. Then 
$(\phi^{-1})^t\times (\phi^{-1})^t\times \phi: H^*(X^3,\QQ)\rightarrow H^*(Y^3,\QQ)$ is an algebraic correspondence taking the  the class $[\Delta]$ of the diagonal to an algebraic class $\Pont$
inducing the Pontryagin product. The latter statement is equivalent to the equality
\[
\phi\left(\pi_{3,*}\left(
\pi_1^*(\phi^{-1}(\gamma_1))\cup \pi_2^*(\phi^{-1}(\gamma_2))\cup [\Delta]
\right)\right)
=
\pi_{3,*}\left(
\pi_1^*(\gamma_1)\cup\pi_2^*(\gamma_2)\cup \Pont
\right),
\]
by Equation (\ref{eq-Pontriagin-product-of-gamma-1-and-gamma-2}). The above equation is verified as follows. Let
$a, b, c$ be classes in $H^*(X,\QQ)$. 
\begin{eqnarray*}
\phi\left(\pi_{3,*}\left(
\pi_1^*[\phi^{-1}(\gamma_1)\cup a]\cup \pi_2^*[\phi^{-1}(\gamma_2)\cup b]\cup \pi_3^*c
\right)\right)&=&
\\
\left(\int_X\phi^{-1}(\gamma_1)\cup a\right)
\left(\int_X\phi^{-1}(\gamma_2)\cup b\right)\phi(c) &=&
\\
\left(\int_X\gamma_1\cup (\phi^{-1})^t(a)\right)
\left(\int_X\gamma_2\cup (\phi^{-1})^t(b)\right)\phi(c) &=&
\\
\pi_{3,*}\left(
\pi_1^*(\gamma_1)\cup\pi_2^*(\gamma_2)\cup \left((\phi^{-1})^t\otimes (\phi^{-1})^t\otimes \phi
\right)(\pi_1^*(a)\cup \pi_2^*(b)\cup \pi_3^*(c))
\right).
\end{eqnarray*}
Finally, we observe that $[\kappa(E^*)\sqrt{td_{Y\times X}}]_*\circ\mu_{t^{-1}}=\mu_t\circ [\kappa(E^*)\sqrt{td_{Y\times X}}]_*,$
as $[\kappa(E^*)\sqrt{td_{Y\times X}}]_*$ is degree reversing, and $\kappa(E^*)=\kappa(E)$, by Corollary \ref{corollary-kappa-E-equal-kappa-E-dual}. Hence, $\phi=
(\phi^{-1})^t$. 
\end{proof}

\hide{
\begin{rem}
\label{rem-conjecture-ok-for-Verbitsky-component}
Let $\phi$ be an automorphism in $\Aut_{\G^{[1]}}(S,\epsilon)$.
The elements $\Theta_n(\phi)$ and $\rho(\iota(\phi))$ of $GL(H^*(S^{[n]},\QQ))$ 
both leave the subring $SH^*(S^{[n]},\QQ)$ generated by $H^2(S^{[n]},\QQ)$ invariant, and they restrict to the same element of
$GL(SH^*(S^{[n]},\QQ))$, by \cite[Theorems 4.10 and 4.11]{taelman}. In particular, Conjecture \ref{conjecture-Theta-n-is-integration-of-infinitesimal-LLV-action} holds for $n=2$, since $SH^*(S^{[2]},\QQ)=H^*(S^{[2]},\QQ)$. Furthermore, 
setting $f=\mu_{\chi(\phi)}\Theta_n(\phi)$ we have
\[
f(\gamma\cup\delta)=f(\gamma)\star f(\delta),
\]
for all $\gamma$, $\delta$ in $SH^*(S^{[n]},\QQ)$, whenever $\phi$ is degree reversing. Indeed, the above equality follows from Lemma \ref{lemma-image-via-rho-of-operators-which-anti-commute-with-h}(\ref{lemma-item-conjugating-cup-product-to-pontryagin-product}) and the equality of the restrictions of $\Theta_n(\phi)$ and $\rho(\iota(\phi))$ to $SH^*(S^{[n]},\QQ)$.
\end{rem}
}

%

\section{Acknowledgements}
This work was partially supported by a grant  from the Simons Foundation (\#427110). 
Part of this work was carried out while the author was visiting the Mathematics Institute at Bonn University. I thank Daniel Huybrechts for the invitation, his hospitality, and for stimulating conversations.
I thank Georg Oberdieck for his comments on an earlier draft of this paper. I thank Chenyu Bai for pointing out the need for Lemma \ref{lemma-parallel-transports-with-same-restriction-to-H-2}. 



\begin{thebibliography}{B-N-R}

\bibitem[A]{addington} Addington, N.: {\em New derived symmetries of some hyperk\"{a}hler varieties.\/} 
Algebr. Geom. 3 (2016), no. 2, 223--260.



\bibitem[Be1]{beauville-varieties-with-zero-c-1}
Beauville, A.: {\em Varietes K\"ahleriennes dont la premiere classe de Chern 
est nulle.}  J. Diff. Geom. 18, p. 755--782 (1983).



\bibitem[Bec]{beckmann} Beckmann, T.: {\em Derived categories of hyper-K\"{a}hler manifolds: extended Mukai vector
and integral structure.\/} Preprint, 
arXiv:2103.13382.


\bibitem[BNS]{BNS} Boissi\'{e}re, S.; Nieper-Wi{\ss}kirchen, M.; Sarti, A.:
{\em Higher dimensional Enriques varieties and automorphisms of generalized Kummer varieties.\/}
J. Math. Pures Appl. (9) 95 (2011), no. 5, 553--563. 

\bibitem[BKR]{BKR} Bridgeland, T., King, A., Reid, M.:
{\em The McKay correspondence as an equivalence of derived categories.\/}
J. Amer. Math. Soc. 14 (2001), no. 3, 535--554.

\bibitem[Bu]{buskin} Buskin, N.: 
{\em Every rational Hodge isometry between two K3 surfaces is
algebraic.\/} J. Reine Angew. Math. 755 (2019), 127--150.

\bibitem[C]{caldararu-thesis} C\u{a}ld\u{a}raru, A.:
{Derived categories of twisted sheaves on
Calabi-Yau manifolds.\/} Thesis, Cornell Univ., May 2000.

\bibitem[CM]{charles-markman} Charles, F., Markman, E.: {\em The standard conjectures for holomorphic symplectic varieties deformation equivalent to Hilbert schemes of K3 surfaces.\/} Compos. Math. 149 (2013), no. 3, 481--494.

\bibitem[G]{gerstein} Gerstein, L.: {\em Basic Quadratic Forms.\/} Graduate Studies in Mathematics, 90. American Mathematical Society, Providence, RI, 2008.



\bibitem[Hai]{haiman} Haiman, M.: {\em Hilbert schemes, polygraphs and the Macdonald positivity conjecture.\/} J. Amer. Math. Soc. 14 (2001), no. 4, 941--1006.


\bibitem[HT]{hassett-tschinkel-lagrangian-planes} Hassett, B., Tschinkel, Y.:
{\em Hodge theory and Lagrangian planes in generalized Kummer fourfolds.\/}
Mosc. Math. J. 13 (2013), no. 1, 33--56.


\bibitem[HKLR]{HKLR} Hitchin, N. J., Karlhede, A., Lindstr\"{o}m, U., and Ro\v{c}ek, M.:
{\em Hyper-K\"{a}hler metrics and supersymmetry.\/}
Comm. Math. Phys. Volume 108, Number 4 (1987), 535--589. 


\bibitem[Hu1]{huybrects-basic-results}
Huybrechts, D.: 
{\em Compact hyper-K\"{a}hler Manifolds: Basic results.\/}
Invent. Math. 135 (1999), no. 1, 63-113 and
Erratum: Invent. Math. 152 (2003), no. 1, 209--212. 

\bibitem[Hu2]{huybrechts-torelli} Huybrechts, D.: {\em 
A global Torelli Theorem for hyper-K\"{a}hler manifolds [after M. Verbitsky].\/} 
S\'{e}minaire Bourbaki: Vol. 2010/2011.  Ast\'{e}risque No. 348 (2012), Exp. No. 1040, 375--403.

\bibitem[Hu3]{huybrechts-FM} Huybrechts, D.: {\em 
Fourier-Mukai Transforms in Algebraic Geometry.\/} Oxford Mathematical Monographs. Oxford University Press, Oxford, 2006.


\bibitem[Hu4]{huybrechts-survey} Huybrechts, D.: {\em Compact hyper-K\"{a}hler manifolds.\/}
 Calabi-Yau manifolds and related geometries (Nordfjordeid, 2001),  
161--225, Universitext, Springer, Berlin, 2003. 

\bibitem[Hu5]{huybrechts-rational-hodge-isometries} Huybrechts, D.: {\em Motives of isogenous K3 surfaces. \/}
Comment. Math. Helv. 94 (2019), no. 3, 445--458.

\bibitem[Hu6]{huybrechts-kahler-cone} Huybrechts, D.: {\em The K\"{a}hler cone of a compact hyperk\"{a}hler manifold. \/}
Math. Ann. 326 (2003), no. 3, 499--513.

\bibitem[HL]{huybrechts-lehn} Huybrechts, D., Lehn, M: {\em
The geometry of moduli spaces of sheaves.\/} Second edition. Cambridge University Press, Cambridge, 2010.

\bibitem[K]{kleiman} Kleiman, S.: {\em Algebraic cycles and the Weil conjectures. Dix expos\'{e}s sur la cohomologie des sch\'{e}mas, \/} 359--386, Adv. Stud. Pure Math., 3, North-Holland, Amsterdam, 1968. 

\bibitem[LL]{looijenga-lunts}  Looijenga, E., Lunts, V.:
{\em A Lie algebra attached to a projective variety.\/} Invent. Math. 129
(1997), no. 2, 361--412. 

\bibitem[Ma1]{markman-integral-generators} Markman, E.: {\em  Integral generators for the cohomology ring
of moduli spaces of sheaves over Poisson surfaces.\/
} Adv. in Math. 208 (2007) 622--646.

\bibitem[Ma2]{markman-monodromy} Markman, E.: {\em On the monodromy of moduli spaces of sheaves on K3 surfaces.\/}
J. Algebraic Geom. 17 (2008), no. 1, 29--99.

\bibitem[Ma3]{markman-exceptional} Markman, E.: {\em  Prime exceptional divisors on holomorphic symplectic varieties and
monodromy reflections.\/} Kyoto J. Math. 53 (2013), no. 2, 345--403.

\bibitem[Ma4]{markman-survey} Markman, E.:
{\em A survey of Torelli and monodromy results for holomorphic-symplectic varieties.\/}
In  ``Complex and Differential Geometry'', W. Ebeling et. al. (eds.),
Springers Proceedings in Math. 8, (2011), pp 257--323.
Available at arXiv:1101.4606.

\bibitem[Ma5]{markman-BBF-class-as-characteristic-class} Markman, E.:
{\em The Beauville-Bogomolov class as a characteristic class.\/}
 J. of Algebraic Geometry, 29 (2020) 199--245.
 
\bibitem[Ma6]{markman-modular}  Markman, E.:
{\em Stable vector bundles on a hyper-K\"{a}hler manifold with a rank $1$ obstruction map are modular.\/}
Electronic preprint arXiv:2107.13991v4.

\bibitem[Ma7]{markman-universal-family}  Markman, E.:
{\em On the existence of universal families of marked irreducible
holomorphic symplectic manifolds.\/} Kyoto J. Math. 61 (2021), no. 1,
207--223.


\bibitem[Ma8]{markman-generalized-kummers} Markman, E.: {\em The monodromy of generalized Kummer varieties and algebraic cycles on their intermediate Jacobians.\/}   Journal of the European Mathematical Society, 
DOI:10.4171/JEMS/1199.


\bibitem[Mu1]{mukai-hodge} Mukai, S.:
{\em On the moduli space of bundles on K3 surfaces I},
Vector bundles on algebraic varieties,
Proc. Bombay Conference, 1984, Tata Institute of Fundamental Research Studies,
no. 11, Oxford University Press, 1987, pp. 341--413.

\bibitem[Mu2]{mukai-ICM} Mukai, S.:
{\em 
Vector bundles on a K3 surface.\/} Proceedings of the International Congress of Mathematicians, Vol. II (Beijing, 2002), 495?502, Higher Ed. Press, Beijing, 2002.

\bibitem[Mu3]{mukai-duality-of-K3} Mukai, S.:
{\em Duality of polarized K3 surfaces.\/} New trends in algebraic geometry (Warwick, 1996), 311--326, London Math. Soc. Lecture Note Ser., 264, Cambridge Univ. Press, Cambridge, 1999.


\bibitem[N]{nikulin} Nikulin, V.: {\em On correspondences between surfaces of K3 type.\/} (Russian) Izv. Akad. Nauk SSSR Ser. Mat. 51 (1987), no. 2, 402--411, 448; translation in Math. USSR-Izv. 30 (1988), no. 2, 375--383.

\bibitem[Ogu]{oguiso} Oguiso, K.: {\em No cohomologically trivial non-trivial automorphisms of generalized Kummer manifolds.\/}
Nagoya Math. J. 239, 110--122 (2020).


\bibitem[P1]{ploog} Ploog, D: {\em Equivariant autoequivalences for finite group actions.\/} Adv. Math. 216 (2007), no. 1, 62--74.

\bibitem[P2]{ploog-thesis} Ploog, D: {\em Groups of autoequivalences of derived categories of smooth projective varieties.\/} 
Logos Berlin (2005).

\bibitem[Ra]{rapagnetta} Rapagnetta, A.: {\em On the Beauville form of the known irreducible symplectic varieties.\/} 
Math. Ann. 340 (2008), no. 1, 77--95.

\bibitem[S]{soldatenkov} Soldatenkov, A.: {\em On the Hodge structures of compact hyper-K\"{a}hler manifolds.\/}
Math. Res. Lett. 28 (2021), no. 2, 623--635.

\bibitem[Ta]{taelman} Taelman, L.: {\em Derived equivalences of hyper-K\"{a}hler varieties.\/}
Preprint arXiv:1906.08081v2.

\bibitem[To]{toda} Toda, Y.: {\em Deformations and Fourier-Mukai transforms.\/}
J. Differential Geom. 81 (2009), no. 1, 197--224.

\bibitem[Ve1]{verbitsky-mirror-symmetry} Verbitsky, M.: 
{\em Mirror symmetry for hyper-K\"{a}hler manifolds.\/}
Mirror symmetry, III (Montreal, PQ, 1995), 115--156, 
AMS/IP Stud. Adv. Math., 10, Amer. Math. Soc., Providence, RI, 1999. 

\bibitem[Ve2]{kaledin-verbitsky-book} Verbitsky, M.
{\em Hyperholomorphic sheaves and new examples of hyperkaehler manifolds,\/} 
alg-geom/9712012. 
In the book: hyper-K\"{a}hler manifolds, by Kaledin, D. and Verbitsky, M., 
Mathematical Physics (Somerville), 12. International Press, 
Somerville, MA, 1999. 

\bibitem[Ve3]{verbitsky-torelli} Verbitsky, M.
{\em Mapping class group and a global Torelli theorem for hyper-K\"{a}hler
manifolds.\/} Duke Math. J. 162 (2013),
no. 15, 2929--2986.


\bibitem[Ve4]{verbitsky-cohomology}  Verbitsky, M.: 
{\em Cohomology of compact hyper-K\"{a}hler manifolds and its 
applications.\/} 
Geom. Funct. Anal.  6  (1996),  no. 4, 601--611. 

\end{thebibliography}
\end{document}